\documentclass[a4paper]{amsart}
\pdfoutput 1
\usepackage{xcolor}
\usepackage{microtype}
\usepackage{verbatim}
\usepackage{amssymb}
\usepackage[a4paper,margin=20ex]{geometry}
\usepackage{url}

\usepackage[colorlinks,pdfpagelabels,pdfstartview = FitH,bookmarksopen = true,bookmarksnumbered = true,linkcolor = blue,plainpages = false,hypertexnames = false,citecolor = red] {hyperref}

\allowdisplaybreaks

\title[Nonlocal gradient potential estimates]{Gradient regularity and first-order potential estimates for a class of nonlocal equations}

\author{Tuomo Kuusi}
\address{Department of Mathematics and Statistics, University of Helsinki}
\email{tuomo.kuusi@helsinki.fi}
\author{Simon Nowak}
\address{Fakult\"at f\"ur Mathematik, Universit\"at Bielefeld, Postfach 100131, D-33501 Bielefeld, Germany}
\email{simon.nowak@uni-bielefeld.de}
 \author{Yannick Sire}
 \address{Department of Mathematics, Johns Hopkins University, Baltimore, MD, 21218}
\email{ysire1@jhu.edu}
\date{}

\newtheorem{theorem}{Theorem}[section]

\newtheorem{proposition}[theorem]{Proposition}
\newtheorem{lemma}[theorem]{Lemma}

\newtheorem{corollary}[theorem]{Corollary}
\theoremstyle{definition}
\newtheorem{definition}[theorem]{Definition}
\newtheorem{remark}[theorem]{Remark}

\numberwithin{equation}{section}

\def\eqn#1$$#2$${\begin{equation}\label#1#2\end{equation}}
\def\charfn_#1{{\raise1.2pt\hbox{$\chi_{\kern-1pt\lower3pt\hbox{{$\scriptstyle#1$}}}$}}}

\newcommand\blfootnote[1]{%
	\begingroup
	\renewcommand\thefootnote{}\footnote{#1}%
	\addtocounter{footnote}{-1}%
	\endgroup
}

 \DeclareMathOperator*{\osc}{osc}

\newcommand{\ern}{\mathbb{R}^n}

\def\loc{\operatorname{loc}}

\def\mean#1{\mathchoice%
          {\mathop{\kern 0.2em\vrule width 0.6em height 0.69678ex depth -0.58065ex
                  \kern -0.8em \intop}\nolimits_{\kern -0.4em#1}}%
          {\mathop{\kern 0.1em\vrule width 0.5em height 0.69678ex depth -0.60387ex
                  \kern -0.6em \intop}\nolimits_{#1}}%
          {\mathop{\kern 0.1em\vrule width 0.5em height 0.69678ex
              depth -0.60387ex
                  \kern -0.6em \intop}\nolimits_{#1}}%
          {\mathop{\kern 0.1em\vrule width 0.5em height 0.69678ex depth -0.60387ex
                  \kern -0.6em \intop}\nolimits_{#1}}}

\def\vintslides_#1{\mathchoice%
          {\mathop{\kern 0.1em\vrule width 0.5em height 0.697ex depth -0.581ex
                  \kern -0.6em \intop}\nolimits_{\kern -0.4em#1}}%
          {\mathop{\kern 0.1em\vrule width 0.3em height 0.697ex depth -0.604ex
                  \kern -0.4em \intop}\nolimits_{#1}}%
          {\mathop{\kern 0.1em\vrule width 0.3em height 0.697ex depth -0.604ex
                  \kern -0.4em \intop}\nolimits_{#1}}%
          {\mathop{\kern 0.1em\vrule width 0.3em height 0.697ex depth -0.604ex
                  \kern -0.4em \intop}\nolimits_{#1}}}

\newcommand{\aveint}[2]{\mathchoice%
          {\mathop{\kern 0.2em\vrule width 0.6em height 0.69678ex depth -0.58065ex
                  \kern -0.8em \intop}\nolimits_{\kern -0.45em#1}^{#2}}%
          {\mathop{\kern 0.1em\vrule width 0.5em height 0.69678ex depth -0.60387ex
                  \kern -0.6em \intop}\nolimits_{#1}^{#2}}%
          {\mathop{\kern 0.1em\vrule width 0.5em height 0.69678ex depth -0.60387ex
                  \kern -0.6em \intop}\nolimits_{#1}^{#2}}%
          {\mathop{\kern 0.1em\vrule width 0.5em height 0.69678ex depth -0.60387ex
                  \kern -0.6em \intop}\nolimits_{#1}^{#2}}}
                  
\newtoks\by
\newtoks\paper
\newtoks\book
\newtoks\jour
\newtoks\yr
\newtoks\pages
\newtoks\vol
\newtoks\publ

\def\name[#1, #2]{#1 #2}
\def\ota{{\hbox{\bf ???}}}
\def\cLear{\by=\ota\paper=\ota\book=\ota\jour=\ota\yr=\ota
\pages=\ota\vol=\ota\publ=\ota}
\def\endpaper{\the\by, \textit{\the\paper},
{\the\jour} \textbf{\the\vol} (\the\yr), \the\pages.\cLear}
\def\endbook{\the\by, \textit{\the\book},
\the\publ, \the\yr.\cLear}
\def\endpap{\the\by, \textit{\the\paper}, \the\jour.\cLear}
\def\endproc{\the\by, \textit{\the\paper}, \the\book, \the\publ,
\the\yr, \the\pages.\cLear}

\begin{document}
\begin{abstract}
We consider nonlocal equations of order larger than one with measure data and prove gradient regularity in Sobolev and H\"older spaces as well as pointwise bounds of the gradient in terms of Riesz potentials, leading to fine regularity results in many commonly used function spaces. The kernel of the integral operators involves a H\"older dependence in the variables and is not assumed to be translation invariant.
\end{abstract}
\maketitle

\tableofcontents

\section{Introduction}
\subsection{Overview}
This paper is devoted to the regularity and fine pointwise properties of the gradient of solutions to nonlocal equations of the type 
\begin{equation} \label{nonlocaleq} 
			L_{A} u = \mu  \text{ in } \Omega \subset \ern,
\end{equation}
	where the nonlocal operator $L_A$ is formally defined by the formula 
\begin{equation} \label{nonlocalop}
L_A u(x) = p.v. \int_{\mathbb R^n} \frac{A(x,y)}{|x-y|^{n+2s}}  (u(x)-u(y))  dy,
\end{equation}
while $\mu$ is a signed Radon measure on $\ern$ with finite total mass and $A$ is a coefficient. \blfootnote{The first author has been supported by the Academy of Finland and the European Research Council (ERC) under the European Union's Horizon 2020 research and innovation program (grant agreement No 818437). The second author gratefully acknowledges the support of the German Research Foundation - SFB 1283/2 2021 - 317210226. The third author is partially supported by the Simons foundation and the NSF grant DMS-2154219.}

We assume throughout the paper that $s \in (1/2,1)$, so that the order of the operator $L_A$ can be given by any number strictly between one and two. The aim of the present contribution is two-fold:

\begin{itemize}
\item We prove gradient regularity of solutions in the scales of Sobolev and H\"older spaces. 
\item We derive so-called potential estimates for the gradient of solutions, leading to fine regularity estimates in a large variety of function spaces. 
\end{itemize} 

Integral operators and therefore nonlocal equations are ubiquitous in several areas of pure and applied mathematics such as for instance stochastic processes and Levy flights (see \cite{bertoin}), classical harmonic analysis (see \cite{landkof}), phase transitions (see \cite{fife}), models of diffusion and their analysis (see \cite{zaslavsky}), conformal geometry (see \cite{CG,Case-Chang,GZ}), physics of materials and relativistic models (see \cite{LiebYau1,LiebYau2}) and image processing (see \cite{GOsh}). 

Starting with seminal contributions by Caffarelli and his school (see e.g.\ \cite{Silvestre,CafSil,CSCPAM,CSARMA,CSannals,CaffVass,CCV,BCF,CStinga}), the regularity of solutions to nonlocal equations in divergence and nondivergence form turned out to be crucial to understand qualitative properties of the solutions in question. Our contribution pertains to this by now very active and wide area of PDE theory.
In fact, without being exhaustive at all, further results concerning H\"older regularity of solutions to \eqref{nonlocaleq} can for example be found in \cite{KassCalcVar,DKP,DKP2,CabS1,RSJMPA,Serra,RSDuke,RSJDE,KMJEMS,BFVCalcVar,BLS,CCDuke,DFPJDE,CKCPDE,FallCalcVar,MeH,CKWCalcVar,BKJMA}, while various results concerning Sobolev regularity were for instance proved in \cite{DongKim,KMS1,SchikorraMA,Grubb,CozziSob,BL,NowakNA,MSY,FMSYPDEA,DongLiu,MeV,MeI}. For a more detailed discussion of related previous regularity results, we refer to section \ref{prev}.


In addition, several years ago together with Mingione two of the authors proved in \cite{KMS2} the existence of so-called SOLA solutions to \eqref{nonlocaleq} in the spirit of \cite{BGJFA} as well as the associated zero-order potential estimates, i.e.\ pointwise bounds for the solution itself in terms of a truncated version of the Riesz potential of the data, showing that the zero-order potential estimates known for local second-order elliptic equations (see e.g.\ \cite{LSW,KM,TW}) have analogues in the nonlocal setting.
In the setting of possibly very general local second-order elliptic and parabolic equations and even systems, it is by now well-established that such potential estimates remain true for the gradient of solutions if the coefficients are sufficiently regular, see e.g.\ \cite{GWG,Min,DM1,DM2,DKMF,KuMiG,KuMi,KuMiV,BCDKS,BYP,DZ22,DFJMPA}.

A main achievement of the present paper consists of proving that similar pointwise \emph{gradient potential estimates} are also valid for a large class of nonlocal equations of any order between one and two. This will in particular imply that the first-order regularity properties of such general nonlocal operators coincide with the sharp ones of the fractional Laplacian $(-\Delta)^s$, since at least formally, the inverse of $(-\Delta)^s$ on the whole space $\ern$ is given by the classical Riesz potential of order $2s$, so that the gradient corresponds to Riesz potentials of order $2s-1$, which are also the ones appearing in our gradient potential estimates below.

Notably, we manage to establish such first-order potential estimates for a rather broad class of coefficients $A$, including H\"older continuous ones with \emph{arbitrarily small} H\"older exponent, which from the perspective of local elliptic equations can be considered to be somewhat surprising, see section \ref{prev}. In particular, although our gradient potential estimates seem to be new already in that case, $A$ is not assumed to be translation invariant. \par In addition to dealing with the nonlocality of the equation, another severe obstacle to overcome in order to prove such a gradient potential estimate in our nonlocal setting is that due to the lower order of the equation, in contrast to local second-order elliptic operators the gradient of solutions to \eqref{nonlocaleq} is not naturally associated with the nonlocal operator \eqref{nonlocalop}. In fact, as will be evident soon, a priori the gradient of weak solutions or SOLA to \eqref{nonlocaleq} is not even known to be locally integrable. \par 
For this and various other reasons, along the way of proving the gradient potential estimate we additionally establish regularity results on the gradient $\nabla u$ of solutions $u$ to \eqref{nonlocaleq}. These regularity estimates can be considered to be interesting also for their own sake, in the sense that they also provide control on the \emph{oscillations} of $\nabla u$, while the gradient potential estimate we prove yields precise control of the \emph{size} of $\nabla u$ in terms of the data, leading to sharp estimates in virtually any reasonable function space.

\vspace{0.2mm}

\subsection{Basic setup}

We consider nonlocal equations posed in domains (=\hspace{0.1mm} open subsets) $\Omega \subset \mathbb R^n$, where for the sake of simplicity we assume that $n \geq 2$. For $\Lambda \geq 1,$ we define $\mathcal{L}_0(\Lambda)$ as the class of all measurable coefficients $A:\mathbb{R}^n \times \mathbb{R}^n \to \mathbb{R}$ that satisfy the following conditions:
\begin{itemize}
	\item There exists a constant  $\Lambda \geq 1$ such that 
	$$\Lambda^{-1} \leq A(x,y) \leq \Lambda \text{ for almost all } x,y \in \mathbb{R}^n .$$
	\item  $
	A(x,y)=A(y,x) \text{ for almost all } x,y \in \mathbb{R}^n.
	$
\end{itemize}


For $s \in (0,1)$ and $A \in \mathcal{L}_0(\Lambda)$, we recall the definition of our nonlocal operator:

$$L_A u(x) = p.v. \int_{\mathbb  R^n} \frac{A(x,y)}{|x-y|^{n+2s}} (u(x)-u(y))  dy.$$


We are going to consider solutions which are a priori assumed to belong to the standard fractional Sobolev spaces $W^{s,p}$ for certain values of $p$, cf.\ section \ref{fracSob1} for a precise definition of these spaces. In addition, in order to control the growth of solutions at infinity, for $s \in (0,1)$ we also consider the tail spaces
$$L^1_{2s}(\mathbb{R}^n):= \left \{u \in L^1_{loc}(\mathbb{R}^n) \mathrel{\Big|} \int_{\mathbb{R}^n} \frac{|u(y)|}{1+|y|^{n+2s}}dy < \infty \right \}$$
that were introduced in \cite{existence}. We note that a function  $u \in L^1_{loc}(\mathbb{R}^n)$ belongs to $L^1_{2s}(\mathbb{R}^n)$ if and only if the \emph{nonlocal tails} of $u$ given by
$$ \int_{\mathbb{R}^n \setminus B_{R}(x_0)} \frac{|u(y)|}{|x_0-y|^{n+2s}}dy$$
are finite for all $R>0$, $x_0 \in \mathbb{R}^n$.
Moreover, we denote by $\mathcal{M}(\ern)$ the set of all signed Radon measures $\mu$ on $\ern$ that satisfy $|\mu|(\ern)<\infty$. \par
Concerning our precise concept of solution, we adopt the following definition of SOLA (=\hspace{0.1mm} solutions obtained by limiting approximations) from \cite{KMS2}. 
\begin{definition} \label{SOLA}
	Consider a bounded domain $\Omega \subset \mathbb{R}^n$. Moreover, let $\mu \in \mathcal{M}(\mathbb{R}^n)$ and $g \in W^{s,2}_{loc}(\mathbb{R}^n) \cap L^1_{2s}(\mathbb{R}^n)$. We say that a function $u$ is a SOLA of the Dirichlet problem 
	\begin{equation} \label{NonlocalDir}
		\begin{cases} \normalfont
			L_{A} u = \mu & \text{ in } \Omega \\
			u=g & \text{ a.e. in } \mathbb{R}^n \setminus \Omega,
		\end{cases}
	\end{equation}
	if the following conditions are satisfied:
	\begin{itemize}
		\item $u \in W^{h,q}(\Omega)$ for any $h \in (0,s)$ and any $q \in \big [1,\frac{n}{n-s} \big )$.
		\item $u$ is a distributional solution of $L_A u=\mu$ in $\Omega$, i.e.\
		\begin{equation} \label{weaksolx5}
			\int_{\mathbb{R}^n} \int_{\mathbb{R}^n} \frac{A(x,y)}{|x-y|^{n+2s}} (u(x)-u(y))(\varphi(x)-\varphi(y))dydx = \int_{\Omega} \varphi d\mu \quad \forall \varphi \in C_0^\infty(\Omega).
		\end{equation}
		\item $u=g$ a.e.\ in $\mathbb{R}^n \setminus \Omega$.
		\item There exist sequences of functions $\{u_j\}_{j=1}^\infty \subset W^{s,2}(\mathbb{R}^n)$, $\{\mu_j\}_{j=1}^\infty \subset C_0^\infty(\mathbb{R}^n)$, $\{g_j\}_{j=1}^\infty \subset C_0^\infty(\mathbb{R}^n)$ such that each $u_j$ weakly solves the Dirichlet problem
		$$
		\begin{cases} \normalfont
			L_A u_j = \mu_j & \text{ in } \Omega \\
			u_j=g_j & \text{ a.e. in } \mathbb{R}^n \setminus \Omega.
		\end{cases}
		$$
	\item $u_j$ converges to $u$ a.e.\ in $\mathbb{R}^n$ and locally in $L^q(\mathbb{R}^n)$.
	\item The sequence $\{g_j\}_{j=1}^\infty$ converges to $g$ in the following sense: For any $z \in \mathbb{R}^n$ and any $r>0$, we have
	\begin{equation} \label{gapp}
		g_j \to g \text{ in } W^{s,2}(B_r(z)), \quad \lim_{j \to \infty} \int_{B_r(z)} \frac{|g_j(y)-g(y)|}{|z-y|^{n+2s}} dy=0.
	\end{equation}
	\item The sequence $\{\mu_j\}_{j=1}^\infty \subset C_0^\infty(\mathbb{R}^n)$ converges weakly to $\mu$ in the sense of measures in $\Omega$ and additionally satisfies
	\begin{equation} \label{measineq}
		\limsup_{j \to \infty} |\mu_j|(B) \leq |\mu|(\overline B) \text{ for any ball } B.
	\end{equation}
	\end{itemize}
\end{definition}
\begin{remark} \normalfont
	We note that by \cite[Theorem 1.1]{KMS2}, for any $\mu \in \mathcal{M}(\mathbb{R}^n)$ and any $g \in W^{s,2}_{loc}(\mathbb{R}^n) \cap L^1_{2s}(\mathbb{R}^n)$, there indeed \emph{always exists} a SOLA of (\ref{NonlocalDir}). 
\end{remark}
If $\mu$ possesses more regularity, that is, if $\mu \in L^\frac{2n}{n+2s}(\Omega)$ and thus belongs to the dual of $W^{s,2}$, then it is possible to prove the existence of weak solutions belonging to the energy space $W^{s,2}$, see e.g.\ \cite[Theorem 2.3]{DKP}, \cite[Remark 3]{existence} or \cite[Proposition 4.1]{MeN}. \par 
Motivated by this observation, we also have the following definition.
\begin{definition}
Given $\mu \in L^\frac{2n}{n+2s}_{loc}(\Omega)$, we say that $u \in W^{s,2}_{loc}(\Omega) \cap L^1_{2s}(\mathbb{R}^n)$ is a weak solution to the equation $L_A u=\mu$ in $\Omega$, if $u$ satisfies \eqref{weaksolx5}. 
\end{definition}
Note that in this case by standard density arguments the class of admissible test functions can actually be extended to the space $W_c^{s,2}(\Omega)$ containing all functions that belong to $W^{s,2}(\Omega)$ and are compactly supported in $\Omega$.

Since we aim to obtain gradient estimates, we need an additional regularity assumption on the coefficient.

\begin{definition} \label{Calpha}
	For $\alpha \in (0,1)$ and $\Gamma>0$, we say that $A$ is $C^\alpha$ in $\Omega$ with respect to $\Gamma$, if for all $x,y \in \Omega$ and any $h \in \mathbb{R}^n$ with $|h|< \min \{\textnormal{dist}(x,\partial \Omega),\textnormal{dist}(y,\partial \Omega)\}$, we have
	\begin{equation} \label{holdkernel}
	|A(x+h,y+h)-A(x,y)| \leq \Gamma |h|^\alpha.
	\end{equation}
\end{definition}
In particular, this condition is satisfied if $A \in C^\alpha(\overline \Omega \times \overline \Omega)$ or if $A$ is translation invariant in $\Omega$ in the sense of the following definition.
\begin{definition} \label{def:TI}
	Let $\Omega \subset \ern$ be a domain and $\Lambda\geq 1$. We say that $A \in \mathcal{L}_0(\Lambda)$ is translation invariant in $\Omega$, if there exists a measurable function $a:\mathbb{R}^n \to \mathbb{R}$ such that $A(x,y)=a(x-y)$ for all $x,y \in \Omega$. We denote the class of all such translation invariant coefficients in $\Omega$ that belong to $\mathcal{L}_0(\Lambda)$ by $\mathcal{L}_1(\Lambda,\Omega)$. Moreover, we set $\mathcal{L}_1(\Lambda):=\mathcal{L}_1(\Lambda,\ern)$.
\end{definition}

We remark that the condition (\ref{holdkernel}) is also satisfied by some more general choices of coefficients besides H\"older continuous or translation invariant ones. For example, if $A^\prime \in \mathcal{L}_0(\Lambda)$ belongs to $C^\alpha(\overline \Omega \times \overline \Omega)$ and $A_0$ belongs to the class $\mathcal{L}_1(\Lambda,\Omega)$, then the coefficient
\begin{equation} \label{mixedcoeff}
A(x,y)=A^\prime(x,y)A_0(x,y)
\end{equation}
belongs to $\mathcal{L}_0(\Lambda^2,\Omega)$ and is $C^\alpha$ in $\Omega$ with respect to some $\Gamma>0$. Therefore, since $A_0$ is not required to satisfy any smoothness assumption, the coefficient $A$ in \eqref{mixedcoeff} might not belong to $C^\alpha(\overline \Omega \times \overline \Omega)$, but nevertheless satisfies the condition (\ref{holdkernel}).

\vspace{-2mm}
\subsection{Main results}

The first difficulty one faces when trying to prove gradient estimates in the nonlocal setting is that the gradient of a SOLA or even weak solution is a priori not even known to be integrable. This is in sharp contrast with the setting of local second-order elliptic equations, where the gradient of any SOLA as defined e.g.\ in \cite{DM2} is locally integrable by definition. In our case, proving the gradient integrability is instead already highly nontrivial and follows from our first main result, which is concerned with higher differentiability.

\begin{theorem}[Higher differentiability under measure data] \label{HD}
	Let $s \in (1/2,1)$, $\mu \in \mathcal{M}(\mathbb{R}^n)$, $g \in W^{s,2}_{loc}(\mathbb{R}^n) \cap L^1_{2s}(\mathbb{R}^n)$ and let $u$ be a SOLA of (\ref{NonlocalDir}) in a bounded domain $\Omega \subset \ern$. In addition, assume that $A \in \mathcal{L}_0(\Lambda)$ is $C^\alpha$ in $\Omega$ with respect to some $\alpha \in (0,1)$ and some $\Gamma>0$. Then for all 
	\begin{equation} \label{rangesm}
	0 < \gamma<\min \{\alpha,2s-1\}, \quad q \in \bigg [1,\frac{n}{n+1+\gamma-2s} \bigg ),
	\end{equation}
	we have $u \in W^{1+\gamma,q}_{loc}(\Omega)$. Moreover, for any $R >0$ and any $x_0 \in \Omega$ such that $B_R(x_0) \Subset \Omega$, we have
	\begin{equation} \label{HDE}
		\begin{aligned}
		& \left (\mean{B_{R/2}(x_0)} |\nabla u|^q dx \right )^{1/q} + R^{\gamma }\left (\mean{B_{R/2}(x_0)} \int_{B_{R/2}(x_0)} \frac{|\nabla u(x)-\nabla u(y)|^q}{|x-y|^{n+\gamma q}} dydx \right )^{1/q} \\ & \leq \frac{C}{R} \left (\left (\mean{B_{R}(x_0)} |u|^q dx \right )^{1/q} + R^{2s}\int_{\mathbb{R}^n \setminus B_R(x_0)} \frac{|u(y)|}{|x_0-y|^{n+2s}}dy + R^{2s-n} |\mu|(B_R(x_0)) \right ),
		\end{aligned}
	\end{equation}
	where $C$ depends only on $n,s,\Lambda,\alpha,\Gamma,\gamma,q$ and $\max\{\textnormal{diam}(\Omega),1\}$.
\end{theorem}
\begin{remark}
We note that in Theorem \ref{HD}, the limit case $\gamma=2s-1$, $q=1$ is unattainable already in the case when $A$ is constant. In fact, the fundamental solution $E_s$ of the fractional Laplacian $(-\Delta)^s$ solves the equation $$(-\Delta)^s E_s=\delta \text{ in } \ern,$$ where $\delta$ is the Dirac measure charging the origin, and is known to satisfy $$E_s(x) \approx |x|^{2s-n},$$ see e.g.\ \cite[Theorem 8.4]{Garofalo}. Since we thus have $|\nabla E_s(x)| \approx |x|^{2s-1-n},$ $\nabla E_s$ clearly does not belong to $L^\frac{n}{n-2s+1}_{loc}(\Omega)$ whenever $\Omega \subset \ern$ contains the origin. Thus, since by the fractional Sobolev embedding we have $W^{2s-1,1}_{loc}(\Omega) \hookrightarrow L^\frac{n}{n-2s+1}_{loc}(\Omega)$, we observe that $\nabla u \notin W^{2s-1,1}_{loc}(\Omega)$.
\end{remark}

For $\mu \in \mathcal{M}(\ern)$, recall that the classical Riesz potential of order $\beta \in (0,n)$ of $|\mu|$ is defined by
$$
	I^{|\mu|}_{\beta}(x_0):= \int_{\ern} \frac{d|\mu|(y)}{|x_0-y|^{n-\beta}}, \quad x_0 \in \ern.
$$
Since we are concerned with equations posed in bounded domains, for our purposes it is convenient to also define a truncated version of the classical Riesz potentials.
\begin{definition} \label{Riesz}
Let $\mu \in \mathcal{M}(\ern)$. For any $x_0 \in \mathbb{R}^n$ and any $R>0$, we define the truncated Riesz potential of order $\beta \in (0,n)$ of $|\mu|$ by
$$
I^{|\mu|}_{\beta}(x_0,R):= \int_0^R \frac{|\mu|(B_t(x_0))}{t^{n-\beta}}\frac{dt}{t}.
$$
\end{definition}
It is easy to see that this truncated version of the Riesz potential is consistent with the classical one in the sense that for any $R>0$ and any $x_0 \in \ern$, we have
\begin{equation} \label{truncated}
I^{|\mu|}_{\beta}(x_0,R) \lesssim I^{|\mu|}_{\beta}(x_0).
\end{equation}

\vspace{1mm}
Under similar assumptions as in our higher differentiability result Theorem \ref{HD}, we then have the following pointwise \emph{potential estimate} for the gradient of SOLA to \eqref{NonlocalDir}.
\begin{theorem}[Gradient potential estimate] \label{GPE}
Let $s \in (1/2,1)$, $\mu \in \mathcal{M}(\mathbb{R}^n)$, $g \in W^{s,2}_{loc}(\mathbb{R}^n) \cap L^1_{2s}(\mathbb{R}^n)$ and let $u$ be a SOLA of (\ref{NonlocalDir}) in a bounded domain $\Omega \subset \ern$. In addition, assume that $A \in \mathcal{L}_0(\Lambda)$ is $C^\alpha$ in $\Omega$ with respect to some $\alpha \in (0,1)$ and some $\Gamma>0$. Then for almost every $x_0 \in \Omega$ and any $R >0$ such that $B_R(x_0) \subset \Omega$, we have the estimate
\begin{equation} \label{NGPE}
|\nabla u(x_0)| \leq C \left ( R^{-1} \mean{B_R(x_0)}|u|dx + R^{2s-1} \int_{\mathbb{R}^n \setminus B_R(x_0)} \frac{|u(y)|}{|x_0-y|^{n+2s}}dy + I^{|\mu|}_{2s-1}(x_0,R) \right ),
\end{equation}
where $C$ depends only on $n,s,\Lambda,\alpha,\Gamma$ and $\max\{\textnormal{diam}(\Omega),1\}$.
\end{theorem}

In addition to giving precise pointwise information in the presence of general measure data, potential estimates such as (\ref{NGPE}) also have strong implications concerning regularity estimates of Calder\'on-Zygmund-type. In fact, a main strength of potential estimates is that "passing through potentials" enables us to detect finer scales that are difficult to reach by more traditional methods. In order to make this point precise, recall that the Lorentz spaces $L^{p,\sigma}(\Omega)$ are defined as follows.

\begin{definition}
For any domain $\Omega \subset \mathbb{R}^n$, any $p \in [1,\infty)$ and any $\sigma \in (0,\infty]$ we define the Lorentz space $L^{p,\sigma}(\Omega)$ as the space of measurable functions $f: \Omega \to \mathbb{R}^n$ such that the quasinorm
$$ ||f||_{L^{p,\sigma}(\Omega)}:=
	\begin{cases} \normalfont
	 p^\frac{1}{\sigma} \left (\int_0^\infty t^\sigma |\{x \in \Omega \mid |f(x)| \geq t \}|^\frac{\sigma}{p} \frac{dt}{t} \right )^\frac{1}{\sigma}  & \text{ if } \sigma<\infty\\
	\sup_{t>0} t \hspace{0.2mm} | \{x \in \Omega \mid |f(x)| \geq t \} |^\frac{1}{p} & \text{ if } \sigma=\infty
	\end{cases} $$
is finite.
\end{definition}

In particular, we note that for any $p \in [1,\infty)$, by Cavalieri's principle we have $L^{p,p}(\Omega)=L^p(\Omega)$, so that the Lorentz spaces refine the scale of $L^p$ spaces. Moreover, it is easy to see that $L^{p,\sigma}(\Omega) \hookrightarrow L^{p,\varrho}(\Omega)$ for all $p \in [1,\infty)$ and all $\sigma,\varrho \in (0,\infty]$ with $\sigma<\varrho$. If in addition $\Omega$ is a bounded domain, then we also have $L^{p_1,\sigma_1}(\Omega) \hookrightarrow L^{p_0,\sigma_0}(\Omega)$ for all $1 \leq p_0 < p_1< \infty$ and all $\sigma_0,\sigma_1 \in (0,\infty]$.

The following fine regularity results then follow directly from \eqref{NGPE} by taking into account \eqref{truncated} and the mapping properties of the Riesz potential $I^{|\mu|}_{2s-1}$ given by \cite[Proposition 2.8]{Cianchi}.

\begin{corollary} \label{Lorentzreg}
Let $s \in (1/2,1)$ and let $u$ be a SOLA of (\ref{NonlocalDir}) in a bounded domain $\Omega \subset \ern$. In addition, assume that $A \in \mathcal{L}_0(\Lambda)$ is $C^\alpha$ in $\Omega$ with respect to some $\alpha \in (0,1)$ and some $\Gamma>0$. 
\begin{itemize}
	\item We have the implication 
	\begin{equation} \label{Lorentz1}
	\mu \in \mathcal{M}(\ern) \implies \nabla u \in L^{\frac{n}{n-2s+1},\infty}_{loc}(\Omega).
	\end{equation}
	\item For any $p \in \left (1,\frac{n}{2s-1} \right )$ and any $\sigma \in (0,\infty]$, we have the implication 
	\begin{equation} \label{Lorentz2}
	\mu \in L^{p,\sigma}(\Omega) \implies \nabla u \in L^{\frac{np}{n-(2s-1)p},\sigma}_{loc}(\Omega).
	\end{equation}
	\item We have the Lipschitz criterion
	\begin{equation} \label{Lorentz3}
	\mu \in L^{\frac{n}{2s-1},1}(\Omega) \implies \nabla u \in L^{\infty}_{loc}(\Omega).
	\end{equation}
\end{itemize}
\end{corollary}
For instance, the implication \eqref{Lorentz1} sharpens Theorem \ref{HD} in the case when $\gamma=0$, which only yields $\nabla u \in L^p_{loc}(\Omega)$ for any $p <\frac{n}{n-2s+1}$.

Moreover, in the setting of Corollary \ref{Lorentzreg}, the implication \eqref{Lorentz2} applied with $p=\sigma$ together with the embedding $L^{\frac{np}{n-(2s-1)p},p}_{loc}(\Omega) \hookrightarrow L^{\frac{np}{n-(2s-1)p}}_{loc}(\Omega)$ in particular yields the following slightly coarser implication in the standard $L^p$ spaces: For any $p \in \left (1,\frac{n}{2s-1} \right )$, we have
\begin{equation} \label{Lpest}
\mu \in L^{p}(\Omega) \implies \nabla u \in L^{\frac{np}{n-(2s-1)p}}_{loc}(\Omega).
\end{equation}

More generally, the potential estimate \eqref{NGPE} implies estimates on $\nabla u$ in any function space in which the mapping properties of the Riesz potential are known, which is the case also in many other commonly used rearrangement invariant function spaces such as Orlicz spaces, see \cite{Cianchi}.

If $\mu$ is slightly more regular than assumed in the Lipschitz criterion \eqref{Lorentz3}, then we obtain H\"older regularity for $\nabla u$, as our final main result shows.

\begin{theorem}[$C^{1,\gamma}$ regularity] \label{Holdgrad}
	Let $\Omega \subset \mathbb{R}^n$ be a bounded domain. Let $s \in (1/2,1)$ and suppose that $A \in \mathcal{L}_0(\Lambda)$ is $C^\alpha$ in $\Omega$ with respect to some $\alpha \in (0,1)$ and some $\Gamma>0$ and that $\mu \in L^p(\Omega)$ for some $p > \frac{n}{2s-1}$. Moreover, let $u \in W^{s,2}_{loc}(\Omega) \cap L^1_{2s}(\ern)$ be a weak solution to $L_{A} u = \mu \text{ in } \Omega.$
	Then for any $$0<\gamma < \min\{\alpha,2s-n/p-1\},$$ we have $u \in C^{1,\gamma}_{loc}(\Omega)$. In addition, for any $R>0$ and any $x_0 \in \Omega$ with $B_{R}(x_0) \Subset \Omega$ and any $\rho \in (0,1)$, we have the estimate
	\begin{equation} \label{C1gest}
		\begin{aligned}
			& [\nabla u]_{C^\gamma(B_{\rho R}(x_0))} \\ & \leq \frac{C}{R^{1+\gamma}} \Bigg ( \mean{B_{R}(x_0)} |u| dx + R^{2s} \int_{\mathbb{R}^n \setminus B_{R}(x_0)} \frac{|u(y)| }{|x_0-y|^{n+2s}}dy + R^{2s-\frac{n}{p}} ||\mu||_{L^p(B_R(x_0))} \Bigg ),
		\end{aligned}
	\end{equation}
	where $C$ depends only on $n,s,\Lambda,\alpha,\Gamma,p,\gamma,\rho$ and $\max\{\textnormal{diam}(\Omega),1\}$. 
\end{theorem}

\subsection{Previous results and differential stability effect} \label{prev}
Let us compare our main results with existing related results in the literature, discussing the strength of our assumption \eqref{holdkernel} on the coefficient $A$ and some remaining open questions along the way.

First of all, we stress that our gradient potential estimates from Theorem \ref{GPE} as well as the Lorentz regularity results from Corollary \ref{Lorentzreg} are completely new and nontrivial even in the case when $A$ is very regular, since to the best of our knowledge all previous estimates concerning nonlocal measure data problems of the type \eqref{nonlocaleq} were obtained below the gradient level. In the case when the measure $\mu$ does not belong to the dual of $W^{s,2}$, the same is true for our higher differentiability result given by Theorem \ref{HD}.

Concerning such estimates on the solution itself instead of its gradient, recently similar estimates to the ones proved in \cite{KMS2} were obtained in \cite{KLL}, where it was also demonstrated how such zero-order potential estimates can be used to prove an analogue of Wiener's test concerning the regularity of boundary points. More zero-order pointwise estimates for solutions to various types of nonlocal equations were e.g.\ established in the papers \cite{GVPisa,KKL21,Verbitsky}. Furthermore, some existence and regularity results below the gradient level in the case when the coefficient $A$ in \eqref{nonlocalop} is constant or translation invariant were obtained in e.g.\ \cite{KPU,KRJFA,CVJDE,Petitta}.

Next, let us discuss some related results in the case when $\mu$ is given by some function that at least belongs to $L^\frac{2n}{n+2s}$, so that it is natural to consider weak solutions instead of SOLA. First regularity results in this case were of De Giorgi-Nash-Moser-type, showing that weak solutions to $L_A u=0$ are locally H\"older continuous whenever $A \in \mathcal{L}_0(\Lambda)$, see e.g.\ \cite{KassCalcVar,CCV,DKP}.

Concerning Sobolev regularity, in \cite{CozziSob} it was proved that if $\mu \in L^2$ and $A$ is $C^s$, then $u$ belongs to $W^{2s-\varepsilon,2}_{\loc}$ for any $\varepsilon>0$. While the assumption that A is $C^s$ might appear natural from the perspective of local second-order equations in divergence form, where the maximal differentiability gain is precisely dictated by the differentiability prescribed on the coefficients, our Theorem \ref{HD} in particular yields a similar differentiability gain almost up to order $2s$ under the weaker assumption that $A$ is $C^{2s-1}$. Since at least formally, nonlocal operators of the type \eqref{nonlocalop} converge to local second-order operators of the type $\textnormal{div}(B \nabla u)$ as $s \to 1$ (see \cite{FKV} for some rigorous results in this direction), this weaker assumption can be considered to be somewhat surprising.

The at first sight unexpected weak requirement on $A$ can be explained as follows. In fact, in \cite{KMS1} again two of the authors together with Mingione discovered a purely nonlocal phenomenon, namely that weak solutions $u \in W^{s,2}$ of \eqref{nonlocaleq} belong to $W^{s+\varepsilon,2+\varepsilon}_{loc}$ for some small $\varepsilon=\varepsilon(n,s,\Lambda)>0$ whenever $A \in \mathcal{L}_0(\Lambda)$ and $\mu$ is integrable enough. In other words, weak solutions self-improve automatically both at the integrability and differentiability scales already in the case of merely bounded measurable coefficients, which is in sharp contrast to the case of local second-order elliptic equations of the type $\textnormal{div}(B \nabla u)=\mu$. Namely, in the case of such local equations such an improvement in general occurs only along the integrability scale (see \cite{meyers}), but not along the differentiability scale (see \cite{KMS1} for a simple counterexample) under merely measurable coefficients. Therefore, the differentiability gain under such weak assumptions on $A$ is indeed a \emph{purely nonlocal phenomenon}. We also want to mention that an alternative proof of this self-improvement  of weak solutions to \eqref{nonlocaleq} was given by Schikorra in \cite{SchikorraMA}.

If one imposes more regularity on the coefficient $A$, then this differential stability effect of nonlocal operators of the type \eqref{nonlocalop} becomes even more visible. In fact, in \cite{MSY} it was observed that if $A$ is H\"older continuous, then in our case when $s \in (1/2,1)$ weak solutions $u \in W^{s,2}$ of \eqref{nonlocaleq} belong to $W^{t,p}_{loc}$ for any $s \leq t < 1$ and $2 \leq p < \infty$ whenever $\mu$ is sufficiently integrable, gaining a substantial amount of differentiability. In \cite{MeI,MeV}, the second author then proved that this Sobolev regularity result of purely nonlocal type remains valid if $A$ has vanishing mean oscillation, allowing in particular for discontinuous coefficients. Again, this is in sharp contrast with the regularity properties for local second-order equations in divergence form, for which still no differentiability gain at all is available under merely VMO or continuous coefficients.

As hinted at above, the mentioned differential stability effect is intrinsically also contained in our main results. In fact, in the case of local second-order equations in divergence form it is easy to see that under a $C^\alpha$ assumption, a differentiability gain of at most order $\alpha$ is attainable. On the other hand, in all our main results we almost gain $1-s+\alpha$ derivatives whenever $A$ is $C^\alpha$, so that the gain of the additional almost $1-s$ derivatives is attained by taking advantage of the special structure of the operator \eqref{nonlocaleq} and not by exploiting the decay of the coefficient, enabling us to prove gradient estimates while assuming only an arbitrary small amount of decay of the coefficient, which in particular is \emph{independent} of $s$.

Nonetheless, we remark that since in the case of local equations gradient potential estimates can be proved under the slightly weaker assumption that the coefficients are Dini instead of H\"older continuous (see e.g.\ \cite{GWG,DM2}), an interesting question is if the estimate \eqref{NGPE} remains valid if the $C^\alpha$ assumption \eqref{holdkernel} is relaxed to a Dini-type condition, which would essentially correspond to taking the intrinsic additional regularization properties discussed above to the borderline case. Another interesting question is if our main results remain valid if our pointwise $C^\alpha$ assumption \eqref{holdkernel} is relaxed to an integral one as assumed in the context of obtaining $C^{1,\gamma}$ regularity in the paper \cite{FeRo23}, which appeared after the completion of the present work.

Next, let us compare also our $C^{1,\gamma}$ regularity result Theorem \ref{Holdgrad} with the previous literature.
Namely, the conclusion of Theorem \ref{Holdgrad} was previously obtained in the papers \cite{FallCalcVar,FMSYPDEA} under somewhat different assumptions on the coefficient $A$. More precisely, in these papers it is assumed that $A$ is of the form
\begin{equation} \label{altform}
	A(x,y):=\widehat A \left (x,|x-y|,\frac{x-y}{|x-y|} \right ),
\end{equation}
where $\widehat A$ is a function defined on $\ern \times [0,\infty) \times \mathbb{S}^{n-1}$ that is required to be $C^\alpha$ in the first variable, continuous in the second variable, and measurable in the third variable. Under these assumptions for $A$ of the form \eqref{altform} locally we have
$$ |A(x+h,y+h)-A(x,y)| = \left |\widehat A \left (x+h,|x-y|,\frac{x-y}{|x-y|} \right ) -\widehat A \left (x,|x-y|,\frac{x-y}{|x-y|} \right ) \right |  \lesssim |h|^\alpha ,$$
so that the smoothness assumptions from \cite{FallCalcVar,FMSYPDEA} imply our $C^\alpha$ assumption \eqref{holdkernel}. However, the converse is not true: For instance, if $A$ is of the form \eqref{mixedcoeff} with $A^\prime \in C^\alpha(\overline \Omega \times \overline \Omega)$ and $A_0 \in \mathcal{L}_1(\Lambda)$, then $A$ in general does not fit into the framework from \cite{FallCalcVar,FMSYPDEA} unless $A_0(x,y)=a(x-y)$ satisfies additional assumptions such as continuity or radial symmetry of $a$. \par In other words, our $C^\alpha$ assumption \eqref{holdkernel} is less restrictive than the smoothness assumptions in \cite{FallCalcVar,FMSYPDEA}, so that also Theorem \ref{Holdgrad} is a new result. Nevertheless, it is noteworthy that in the aforementioned papers certain degeneracies of the coefficient are allowed, so that the results in \cite{FallCalcVar,FMSYPDEA} are valid for certain coefficients that might not belong to the class $\mathcal{L}_0(\Lambda)$. For this reason, an interesting question for future investigation is if our results remain valid if $A$ is allowed to exhibit similar degeneracies. \par In addition, we remark that in \cite{FMSYPDEA} also implications of the type \eqref{Lpest} in $L^p$ spaces were obtained under the assumptions mentioned above, so that in the case of non-degenerate coefficients, our implication \eqref{Lorentz2} with $p=\sigma$ sharpens this $L^p$ regularity result by capturing the optimal regularity on the finer scale of Lorentz spaces. Moreover, it follows from our results that the implication \eqref{Lpest} remains valid under the more general smoothness assumption (\ref{holdkernel}) on $A$.

Since in the context of local equations all of our main results have counterparts for nonlinear equations of $p$-Laplacian-type (see e.g.\ \cite{MinCZMD,AKMARMA} for higher differentiability estimates under measure data, \cite{KuMiARMA1,KuMi} for gradient estimates in terms of Riesz potentials and for instance \cite{DiBen,Manfredi} for the classical $C^{1,\gamma}$ estimates), an interesting question is if our main results have counterparts for nonlocal equations of fractional $p$-Laplacian-type as considered in e.g.\ \cite{DKP,KMS2,BL,BLS,KLL}. However, while concerning higher differentiability some results are available in the case when the right-hand side $\mu$ is sufficiently regular (see \cite{BL}), at this point establishing or disproving Lipschitz or even $C^{1,\gamma}$ estimates for such equations seems to be an open problem even in the homogeneous case when $\mu=0$, where so far solutions are only known to be almost Lipschitz for appropriate values of $s$ and $p$, see \cite{BLS}.

Finally, we note that in the case when $s \in (0,1/2]$ no gradient potential estimates of the precise type \eqref{NGPE} are valid, since the Riesz potential is not defined for negative orders. Nevertheless, we want to mention that $C^{1,\gamma}$ estimates remain valid for nonlocal equations if the coefficient and the data satisfy a H\"older assumption, see \cite{FeRo23}. Therefore, an interesting question is if gradient regularity results in other function spaces such as Sobolev spaces can also be established in the case $s \in (0,1/2]$ provided that the data is assumed to be more regular than in the present paper.



\vspace{0.2cm}

\subsection{Approach}
Let us give a brief heuristic explanation of the philosophy of our approach, highlighting the similarities and differences in comparison to known approaches. \par 
Broadly speaking, the main difficulties arising in addition to the ones already present in the case of local second-order elliptic equations in divergence form originate from the following two sources:
\begin{itemize}
	\item The nonlocality of the equation, leading to nonlocal tail terms.
	\item The lower order of the equation, resulting in a lack of obvious energy estimates on $\nabla u$.
\end{itemize}
For the sake of simplicity, let us first focus on the case when $\mu=0$, as proving precise estimates in this homogeneous case lays the foundation of being able to prove sharp results under general measure data later on. In addition, in this section we shall focus on weak solutions instead of SOLA, since the latter case can in all relevant cases be recovered by approximation. \par
Our starting point is the observation that due to the assumption that the coefficient $A$ is $C^\alpha$ in $\Omega$, around any fixed point $z \in \Omega$, $A$ is locally close to the "frozen" translation invariant coefficient
$$A_z(x,y):=(A(x-y+z,z)+A(y-x+z,z))/2$$
in a $C^\alpha$ sense. In the context of local second-order elliptic equations in divergence form, such a closeness to the translation invariant case can usually be combined with standard energy methods in order to prove that at small scales, the gradient of any solution is close to the gradient of a corresponding approximate solution to an equation with constant coefficients. However, in our lower-order nonlocal setting, no similar obvious comparison estimate is available for the gradient of $u$, since $\nabla u$ is not naturally associated with the nonlocal operator \eqref{nonlocalop}. Instead, we only have energy estimates up to order $s$ at our disposal. \par 
We circumvent this issue by invoking methods which are more typical for proving gradient estimates for fully nonlinear equations as done in e.g.\ \cite{CaffFNAnnals,CSARMA}. In fact, in these papers the lack of a comparison estimate on the gradient level is compensated by exploiting the fact that for any solution $u$ to a fully nonlinear equation of order larger than one and any affine function $\ell$, the function $u-\ell$ solves the same equation. In view of approximation and geometric iteration, this fact can then be used in order to prove the desired H\"older estimates for $\nabla u$ in the alternative form
$$||u-\ell_{r,z}||_{L^\infty(B_r(z))} \lesssim r^{1+\gamma}$$
for some affine function $\ell_{r,z}$, which is equivalent to $C^{1,\gamma}$ regularity of $u$. \par 
Unfortunately, in our case for a solution to $L_A u =0$ and some affine function $\ell$, $u-\ell$ might in general not be a solution of the same equation unless $A$ is translation invariant, so that geometric iteration in terms of affine functions is not immediately applicable. Nevertheless, such approaches have already been successfully adapted also to second-order elliptic equations in divergence form, for instance to obtain large-scale regularity results in the context of stochastic homogenization for such local equations, see for example \cite{AKMBook}.
In order to carry out such an iteration in terms of affine functions successfully in our nonlocal setting, we apply a somewhat similar strategy as in the local divergence form case. First of all, we locally rewrite the equation in the form 
\begin{equation} \label{rewrite}
L_{A_z} u=L_{A_z-A} u,
\end{equation}
conceptualizing it as an equation driven by a translation invariant kernel at the cost of introducing the error term $L_{A_z-A} u$. Due to the translation invariance of $A_z$, the operator $L_{A_z}$ kills affine functions as in the fully nonlinear case. Therefore, for $r>0$ small and any affine function $\ell$, for the weak solution of
$$
	\begin{cases} \normalfont
		L_{A_z} v = 0 & \text{ in } B_{r}(z) \\
		v = u-\ell & \text{ a.e. in } \mathbb{R}^n \setminus B_{r}(z),
	\end{cases}
$$
locally we have
\begin{equation} \label{rewrite1}
	L_{A_z} (u-\ell-v)=L_{A_z-A} u.
\end{equation}

Testing this equation with $w:=u-\ell-v$ itself then shows in particular that the $L^2$ norm of $w$ can be controlled by the error term $L_{A_z-A} u$. \par 
Since due to the translation invariance of $A_z$, $\nabla v$ is already known to satisfy H\"older estimates, all the quantities in the estimates arising in the homogeneous case have one of the following two origins: Either they originate from the H\"older estimates on $\nabla v$, in which case they can be iterated in a similar way as in the fully nonlinear case, or from estimating the error term.
While the terms arising due to the presence of the error term cannot be iterated as in the fully nonlinear case, they are nevertheless stable under the iteration, which is because the error term can be shown to be essentially of lower order. At small scales, this is achieved by exploiting the $C^\alpha$ assumption on $A$, while at large scales sufficient decay is guaranteed by the fact that far away from the diagonal, the kernel of the operator ceases to be singular. \par Roughly speaking, implementing these ideas then leads to decay estimates of the type
$$\mean{B_r(z)} |u-\ell_{r,z}|dx \lesssim r^{1+\gamma}$$
for some affine function $\ell_{r,z}$.

In view of Campanato's characterization of $C^{1,\gamma}$ regularity in terms of affine functions (see \cite{Campanato}), this leads to gradient oscillation decay estimates of the form
\begin{equation} \label{gradoscdecayx}
		\begin{aligned}
		& \osc_{B_{\rho 2^{-N}R}(x_0)} \nabla u \\ & \lesssim \rho^\gamma (2^{-N}R)^{-1} \Bigg ( \mean{B_{2^{-N} R}(x_0)} |u-\ell| dx + (2^{-N} R)^{2s} \int_{\mathbb{R}^n \setminus B_{2^{-N} R}(x_0)} \frac{|u(y)-\ell(y)| }{|x_0-y|^{n+2s}}dy \\
		& \quad + 2^{-\alpha N} \sum_{k=0}^{N} 2^{k(\alpha -2s)} \mean{B_{2^{k-N}R}(x_0)} |u-c|dx+ (2^{-N} R)^{2s} \int_{\mathbb{R}^n \setminus B_{R}(x_0)} \frac{|u(y)-c|}{|x_0-y|^{n+2s}}dy \Bigg )
	\end{aligned}
\end{equation}
whenever the equation holds in $B_{2^{-N}R}(x_0)$ and $A$ is $C^\alpha$ in $B_{8R}(x_0)$. Here $\rho$ and $R$ are sufficiently small, while $N \geq 0$ is an arbitrary integer, $\ell$ is an arbitrary affine function and $c$ is an arbitrary real number. In the special case when $N=\ell=c=0$, this estimate implies Theorem \ref{Holdgrad} for $\mu=0$ by standard covering arguments. \par However, the exact form of the estimate (\ref{gradoscdecayx}) turns out to be crucial in order to prove the higher differentiability result given by Theorem \ref{HD} and the potential estimate in Theorem \ref{GPE} in later sections, as it precisely encodes the different types of information given by the left-hand side and right-hand side of \eqref{rewrite}. More concretely, the presence of the affine function $\ell$ in the first two terms on the right-hand side of (\ref{gradoscdecayx}) again allows to iterate these terms in essentially any first-order setting, while the last two terms are stable in the context of any such iteration, since the $C^\alpha$ decay of the coefficient and the lack of singularity of the kernel in the off-diagonal regime are encoded in them. \par 
Let us now proceed to the case of nonlocal equations of the form \eqref{nonlocaleq} with general measure data $\mu$. In order to prove the higher differentiability estimate from Theorem \ref{HD}, we adapt certain covering methods commonly used in the context of local second-order elliptic equations (see e.g.\ \cite{KrMin,KrMin1,MinCZMD,Min,AKMARMA}) to our nonlocal setting, enabling us to prove estimates in Nikolskii spaces $N^{t,q}$. These Nikolskii estimates then imply the desired Sobolev regularity estimates in view of well-known embeddings. \par 
More concretely, we fix some increment $h \in \ern \setminus \{0\}$ with $|h|$ small. Moreover, we fix some $\beta \in (0,1)$ to be chosen and cover our domain by finitely many balls $B_{|h|^\beta}(z_j)$ with controlled overlap. For some $r \approx |h|^\beta$ and any $j$, we then consider the weak solution $v_j$ of the associated homogeneous equation $L_A v_j=0$ in $B_{r}(z_j)$ satisfying $v_j=u$ in $\ern \setminus B_{r}(z_j)$ and estimate the $L^q$ norm of the second-order quotients of $u$ over $B_{|h|^\beta}(z_j)$ as follows
\begin{equation} \label{triv}
	\left (\int_{B_{|h|^\beta}(z_j)} |\tau_h^2 u|^qdx \right )^{1/q} \lesssim \left (\int_{B_{r}(z_j)} |u-v_j|^qdx \right )^{1/q} + |h| \left (\int_{B_{2|h|^\beta}(z_j)} |\tau_h \nabla v_j|^qdx \right )^{1/q}.
\end{equation}
In view of known comparison estimates from \cite{KMS2}, the first term on the right-hand side always decays sufficiently if $\beta$ is chosen close enough to $1$. On the other hand, the second term on the right-hand side of \eqref{triv} can be estimated by means of our H\"older estimate \eqref{gradoscdecayx}, leading to an incremental decay gain, which can be iterated due to the presence of the constant $c$ and the affine function $\ell$ in the estimate \eqref{gradoscdecayx}. \par
Since the balls in \eqref{triv} depend on $|h|$, this iteration needs to be preceded by a covering argument, which is particularly delicate in our nonlocal setting. This is because due to the presence of the tail terms in \eqref{gradoscdecayx}, exploiting the finite overlap of the balls in the covering inflated by only a fixed constant factor as done in the local case does not suffice to conclude. Instead, we additionally need to ensure that the increasing overlap of the balls coming from the tail terms is compensated by the increasing lack of singularity of the kernel of the nonlocal operator far from the diagonal. \par 
Implementing these iteration and covering arguments then leads to higher differentiability in the range \eqref{rangesm} as desired, proving Theorem \ref{HD}. Furthermore, the obtained higher differentiability estimate allows to upgrade the mentioned comparison estimate from \cite{KMS2} to the gradient level, which serves as an important tool in our proof of the gradient potential estimate \eqref{NGPE}. \par 
In fact, combining this gradient comparison estimate with \eqref{gradoscdecayx} for appropriate choices of $\ell$ and $c$ enables us to deduce \emph{gradient excess decay estimates} which capture both the nonlocality and the measure data present in our setting. More precisely, these estimates take the form
\begin{equation} \label{gradoscdecayIntro}
	\begin{aligned}
		E_R(\nabla u,x_0,2^{-(N+m)}) & \lesssim 2^{-\gamma m} E_R(\nabla u,x_0,2^{-N}) \\
		& \quad + 2^{-\gamma m} 2^{-\alpha N} \sum_{k=0}^{N} 2^{k(1+\alpha -2s)} |(\nabla u)_{B_{2^{k-N}R}(x_0)}| \\
		& \quad + 2^{-\gamma m} (2^{-N}R)^{2s-1} \int_{\mathbb{R}^n \setminus B_{R}(x_0)} \frac{|u  - \left ( u \right )_{B_{R}(x_0)}| }{|x_0-y|^{n+2s}}dy \\
		& \quad +(2^{-N}R)^{2s-1-n} |\mu|(B_{2^{-N} R}(x_0)),
	\end{aligned}
\end{equation}
where $m \geq 1,$ $N \geq 0$ are arbitrary integers and for any integer $i \geq 0$, the nonlocal gradient excess decay functional $E_R(\nabla u,x_0,2^{-i})$ is defined by
\begin{align*}
	E_R(\nabla u,x_0,2^{-i}):& = \sum_{k=0}^{i} 2^{(1-2s)k} \mean{B_{2^{k-i}R}(x_0)} |\nabla u - \left ( \nabla u \right )_{B_{2^{k-i}R}(x_0)} | dx .
\end{align*}
With this excess decay estimate at our disposal, the gradient potential estimates from Theorem \ref{GPE} can then be proved by adapting iteration arguments commonly used in the context of local equations (see e.g.\ \cite{DM1,DM2}) to our nonlocal setting. Roughly speaking, also this iteration follows a similar philosophy as the ones discussed above, in the sense that the first term on the right-hand side of \eqref{gradoscdecayIntro} can be iterated and eventually reabsorbed into the left-hand side, while the second and third terms sum geometrically under such an iteration. Finally, iterating the last term leads to the appearance of the Riesz potential of $|\mu|$ of order $2s-1$ as expected, so that the proof can be concluded by passing to the limit and applying the Lebesgue differentiation theorem.

\section{Preliminaries and basic regularity results} \label{prelim}

\subsection{Notation} \label{notation}
For convenience, let us fix some notation which we use throughout the paper. By $C$ we denote a general positive constant which possibly varies from line to line and only depends on the parameters indicated in the statement to be proved. 
In some proofs, in order to avoid confusion we may also indicate different constants by using subscripts, i.e.\ $C_i$, $i \in \mathbb{N}_0$, while dependences on parameters of the constants will often be shown in parentheses. \par As usual, by
$$ B_r(x_0):= \{x \in \mathbb{R}^n \mid |x-x_0|<r \}$$
we denote the open euclidean ball with center $x_0 \in \mathbb{R}^n$ and radius $r>0$. \par 
Moreover, if $E \subset \mathbb{R}^n$ is measurable, then by $|E|$ we denote the $n$-dimensional Lebesgue-measure of $E$. If $0<|E|<\infty$, then for any $u \in L^1(E)$ we define
$$ (u)_E:= \mean{E} u(x)dx := \frac{1}{|E|} \int_{E} u(x)dx.$$
Furthermore, for any measurable function $\psi:\mathbb{R}^n \to \mathbb{R}$ and any $h \in \mathbb{R}^n$, we define
$$ \psi_h(x):=\psi(x+h), \quad \tau_h \psi(x):=\psi_h(x)-\psi(x), \quad \tau_h^2 \psi(x):= \tau_h (\tau_h \psi(x))=\psi_{2h}(x)+\psi(x)-2\psi_h(x).$$
In addition, given a signed Radon measure $\mu$ on $\ern$, as usual we define the variation of $\mu$ as the measure defined by
$$ |\mu|(E):=\mu^+(E) + \mu^-(E), \quad E \subset \ern \text{ measurable},$$
where $\mu^+$ and $\mu^-$ are the positive and negative parts of $\mu$, respectively. If $|\mu|(\ern)<\infty$, then we say that $\mu$ has finite total variation or finite total mass.

Finally, given a domain $\Omega \subset \ern$, throughout the paper we conceptualize functions $g \in L^1(\Omega)$ as signed Radon measures on $\ern$ by extending $g$ by $0$ to $\ern$ if necessary and denoting
$$ g(E):=\int_E g dx, \quad  E \subset \ern \text{ measurable},$$
note that in this case for any measurable set $E \subset \ern$ we have $$|g|(E)=\int_E |g| dx.$$

\subsection{Fractional Sobolev spaces}\label{fracSob}
This section is devoted to the definitions and several known useful properties of various types of fractional Sobolev spaces.
\subsubsection{$W^{s,p}$ spaces} \label{fracSob1}
\begin{definition}
	Let $\Omega \subset \mathbb{R}^n$ be a domain. For $p \in [1,\infty)$ and $s \in (0,1)$, we define the fractional Sobolev space
	$$W^{s,p}(\Omega):=\left \{u \in L^p(\Omega) \mid [u]_{W^{s,p}(\Omega)}<\infty \right \}$$
	with norm
	$$ ||u||_{W^{s,p}(\Omega)} := \left (||u||_{L^p(\Omega)}^p + [u]_{W^{s,p}(\Omega)}^p \right )^{1/p} ,$$
	where
	$$ [u]_{W^{s,p}(\Omega)}:=\left (\int_{\Omega} \int_{\Omega} \frac{|u(x)-u(y)|^p}{|x-y|^{n+sp}}dydx \right )^{1/p} .$$
	In addition, for $p \in [1,\infty)$ and $s \in (1,2)$ we define
	$$W^{s,p}(\Omega):=\left \{u \in W^{1,p}(\Omega) \mid [\nabla u]_{W^{s-1,p}(\Omega)}<\infty \right \}$$
	with norm
	$$ ||u||_{W^{s,p}(\Omega)} := \left (||u||_{L^p(\Omega)}^p + ||\nabla u||_{L^p(\Omega)}^p + [\nabla u]_{W^{s-1,p}(\Omega)}^p \right )^{1/p} ,$$
	where
	$$ [\nabla u]_{W^{s,p}(\Omega)}:=\left (\int_{\Omega} \int_{\Omega} \frac{|\nabla u(x)-\nabla u(y)|^p}{|x-y|^{n+sp}}dydx \right )^{1/p} .$$
	Moreover, we define the corresponding local fractional Sobolev spaces by
	$$ W^{s,p}_{loc}(\Omega):= \left \{ u \in L^p_{loc}(\Omega) \mid u \in W^{s,p}(\Omega^\prime) \text{ for any domain } \Omega^\prime \Subset \Omega \right \}.$$
	Also, we define the space 
	$$W^{s,p}_0(\Omega):= \left \{u \in W^{s,p}(\mathbb{R}^n) \mid u = 0 \text{ in } \mathbb{R}^n \setminus \Omega \right \}.$$
\end{definition}

We use the following fractional Poincar\'e inequality, see \cite[Section 4]{Mingione}.
\begin{lemma} \label{Poincare} (fractional Poincar\'e inequality)
	Let $s \in (0,1)$, $p \in [1,\infty)$, $r>0$ and $x_0 \in \mathbb{R}^n$. For any $u \in W^{s,p}(B_r(x_0))$, we have
	$$ \int_{B_r(x_0)} \left | u(x)- (u)_{B_r(x_0)} \right |^p dx \leq C r^{sp} \int_{B_r(x_0)} \int_{B_r(x_0)} \frac{|u(x)-u(y)|^p}{|x-y|^{n+sp}}dydx,$$
	where $C=C(s,p)>0$.
\end{lemma}
We also use another Poincar\'e-type inequality concerning functions that belong to the space $W^{s,2}_0(\Omega)$, see \cite[Lemma 2.3]{MeH}. 
\begin{lemma} \label{Friedrichs} (fractional Friedrichs-Poincar\'e inequality)
	Let $s \in(0,1)$ and consider a bounded domain $\Omega \subset \mathbb{R}^n$. For any $u \in W^{s,2}_0(\Omega)$, we have
	$$
		\int_{\mathbb{R}^n} |u(x)|^2 dx \leq C |\Omega|^{\frac{2s}{n}} \int_{\mathbb{R}^n} \int_{\mathbb{R}^n} \frac{|u(x)-u(y)|^2}{|x-y|^{n+2s}}dydx,
	$$
	where $C=C(n,s)>0$.
\end{lemma}

\subsubsection{Nikolskii spaces}
When proving higher differentiability, it is often more convenient to work with another type of fractional Sobolev spaces, namely the Nikolskii spaces $N^{s,p}$ which are defined by means of difference quotients as follows.

\begin{definition}
	Let $\Omega \subset \mathbb{R}^n$ be a domain. Given $h \in \mathbb{R}^n$ and $k \in \mathbb{N}$, we then define
	$$ \Omega_{h,k}:= \{x \in \Omega \mid x+ih \in \Omega \text{ for any } i \in \{1,...,k\}\}.$$
	For $p \in [1,\infty)$ and $s \in (0,2)$, we define the Nikolskii space
	$$N^{s,p}(\Omega):=\left \{u \in L^p(\Omega) \mathrel{\Big|} [u]_{N^{s,p}(\Omega)} < \infty \right \}$$
	with norm
	$$ ||u||_{N^{s,p}(\Omega)} := ||u||_{L^p(\Omega)} + [u]_{N^{s,p}(\Omega)} ,$$
	where
	$$ [u]_{N^{s,p}(\Omega)}:= \sup_{|h|>0} |h|^{-s} ||\tau^2_h u||_{L^p(\Omega_{h,2})}.$$
	\end{definition}
	\begin{remark} \label{Niksmall} 
	For any $\delta>0$, define 
	$$ [u]_{N^{s,p}_\delta(\Omega)} :=
		\sup_{0<|h|<\delta} |h|^{-s} ||\tau^2_h u||_{L^p(\Omega_{h,2})}.$$
	It is easy to see that for any $\delta>0$, the norm given by
	$$ ||u||_{N^{s,p}_\delta(\Omega)} := ||u||_{L^p(\Omega)} + [u]_{N^{s,p}_\delta(\Omega)} $$ is equivalent to the standard norm $||u||_{N^{s,p}(\Omega)}$ in the sense that there exists a constant $C>0$ depending only on $s$ and $\delta$, such that for any $s \in (0,2)$ we have
	$$||u||_{N^{s,p}_\delta(\Omega)} \leq ||u||_{N^{s,p}(\Omega)} \leq C ||u||_{N^{s,p}_\delta(\Omega)},$$
	see e.g.\ \cite[Remark 2]{CozziSob} or \cite[Lemma 2.2]{BL}.
	\end{remark}
	
	Although the spaces $W^{s,p}$ and $N^{s,p}$ are in general not identical, they are nevertheless closely related. This is a common theme of different types of Besov spaces in general, of which both the spaces $W^{s,p}$ and $N^{s,p}$ are special cases of. \par 
	More precisely, on the one hand for any smooth domain $\Omega$, and $s \in (0,2)$ and any $p \in [1,\infty)$, we have the inclusion
	$$ W^{s,p}(\Omega) \subset N^{s,p}(\Omega),$$
	see \cite[Proposition 3]{CozziSob}. While the opposite inclusion is in general not true, it is almost true, as the following Proposition shows, see \cite[Proposition 4]{CozziSob}.
	\begin{proposition} \label{WNrel}
		Let $\Omega \subset \ern$ be a smooth domain. For all $0<t<s<2$ and any $p \in [1,\infty)$, we have the inclusion
		$$ N^{s,p}(\Omega) \subset W^{t,p}(\Omega)$$
		and there exists a constant $C>0$ which depends only on $n,t,s$ and $p$ such that
		$$ ||u||_{W^{t,p}(\Omega)} \leq C ||u||_{N^{s,p}(\Omega)}.$$
	\end{proposition}

\subsection{The nonlocal tail}
In this section, we gather some technical results concerning the nonlocal tails which appear naturally when studying nonlocal operators of the type \eqref{nonlocalop}. \par 
The following two result relate nonlocal tails with different centers on possibly different scales and will be used frequently throughout the paper, sometimes without explicit reference. For proofs of these two results, we refer to \cite[Lemma 2.2 and Lemma 2.3]{BLS}.
\begin{lemma} \label{tailestz}
	Let $s \in (0,1)$. Then for any $u \in L^1_{2s}(\ern)$, all $0<r<R$ and any $x_0 \in \ern$ we have
	\begin{align*}
		\sup_{x \in B_r(x_0)} \int_{\ern \setminus B_R(x_0)} \frac{|u(y)|}{|x-y|^{n+2s}}dy \leq \left (\frac{R}{R-r} \right )^{n+2s} \int_{\ern \setminus B_R(x_0)} \frac{|u(y)|}{|x_0-y|^{n+2s}}dy.
	\end{align*}
\end{lemma}
\begin{lemma} \label{tr}
Let $s \in (0,1)$. Then for any $u \in L^1_{2s}(\ern)$, all $0<r<R$ and all $x_0,x_1 \in \ern$ such that $B_r(x_0) \subset B_R(x_1)$ we have
\begin{align*}
& r^{2s} \int_{\ern \setminus B_r(x_0)} \frac{|u(y)|}{|x_0-y|^{n+2s}}dy \\ & \leq r^{2s} \left (\frac{R}{R-|x_1-x_0|} \right )^{n+2s} \int_{\ern \setminus B_R(x_1)} \frac{|u(y)|}{|x_1-y|^{n+2s}}dy + r^{-n} \int_{B_R(x_1)} |u|dx.
\end{align*}
\end{lemma}
The following Lemma allows to move between nonlocal tails on different scales in a more precise way and will also be used frequently throughout the paper.
\begin{lemma} \label{tailest}
Let $s \in (0,1)$, $x_0 \in \mathbb{R}^n$, $R>0$ and $N \in \mathbb{N}$. For any function $u \in L^1_{2s}(\mathbb{R}^n)$, we have
\begin{align*}
& (2^{-N}R)^{2s} \int_{\mathbb{R}^n \setminus B_{2^{-N}R}(x_0)} \frac{|u(y)-(u)_{B_{2^{-N}R}(x_0)}|} {|x_0-y|^{n+2s}} dy \\ & \leq C \left (\sum_{k=1}^{N} 2^{-2sk} \mean{B_{2^{k-N}R}(x_0)} |u-(u)_{B_{2^{k-N}R}(x_0)}|dx + (2^{-N}R)^{2s} \int_{\mathbb{R}^n \setminus B_R(x_0)} \frac{|u(y)-(u)_{B_R(x_0)}|} {|x_0-y|^{n+2s}} dy \right ),
\end{align*}
where $C$ depends only on $n$ and $s$.
\end{lemma}
\begin{proof}
Splitting the tail into annuli leads to
	\begin{equation} \label{First}
\begin{aligned}
	& (2^{-N}R)^{2s} \int_{\mathbb{R}^n \setminus B_{2^{-N} R}(x_0)} \frac{|u(y)-(u)_{B_{2^{-N}R}(x_0)}|} {|x_0-y|^{n+2s}} dy \\
	& = \sum_{k=1}^{N} \int_{B_{2^{k-N}R}(x_0) \setminus B_{2^{k-1-N}R}(x_0)} \frac{|u(y)-(u)_{B_{2^{-N}R}(x_0)}|}{|x_0-y|^{n+2s}} dy \\ & \quad + (2^{-N}R)^{2s} \int_{\mathbb{R}^n \setminus B_{R}(x_0)} \frac{|u(y)-(u)_{B_{2^{-N} R}(x_0)}|} {|x_0-y|^{n+2s}} dy \Bigg ) \\
	& \leq C \Bigg (\sum_{k=1}^{N} 2^{-2sk} \mean{B_{2^{k-N}R}(x_0)} |u-(u)_{B_{2^{-N} R}(x_0)}|dx \\ & \quad + (2^{-N}R)^{2s} \int_{\mathbb{R}^n \setminus B_{R}(x_0)} \frac{|u(y)-(u)_{B_{2^{-N} R}(x_0)}|} {|x_0-y|^{n+2s}} dy \Bigg ).
\end{aligned}
\end{equation}
For $k\in \{1,...,N\}$, we further estimate
\begin{align*}
	& \mean{B_{2^{k-N}R}(x_0)} |u-(u)_{B_{2^{-N} R}(x_0)}|dx \\ 
	& \leq \mean{B_{2^{k-N}R}(x_0)} |u-(u)_{B_{2^{k-N}R}(x_0)}|dx + \sum_{j=1}^k |(u)_{B_{2^{j-N}R}(x_0)}-(u)_{B_{2^{j-1-N}R}(x_0)}| \\
	& \leq \mean{B_{2^{k-N}R}(x_0)} |u-(u)_{B_{2^{k-N}R}(x_0)}|dx + 2^n \sum_{j=1}^k \mean{B_{2^{j-N}R}(x_0)} |u-(u)_{B_{2^{j-N}R}(x_0)}|dx \\
	& \leq 2^{n+1} \sum_{j=1}^k \mean{B_{2^{j-N}R}(x_0)} |u-(u)_{B_{2^{j-N}R}(x_0)}|dx.
\end{align*}
Reverting the order of summation and summing the geometric series now leads to
\begin{align*}
	& \sum_{k=1}^{N} 2^{-2sk} \mean{B_{2^{k-N}R}(x_0)} |u-(u)_{B_{2^{-N}R}(x_0)}|dx \\
	& \leq 2^{n+1} \sum_{k=1}^{N} \sum_{j=1}^k 2^{-2sk} \mean{B_{2^{j-N}R}(x_0)} |u-(u)_{B_{2^{j-N}R}(x_0)}|dx \\
	& \leq 2^{n+1} \sum_{j=1}^N 2^{-2sj} \mean{B_{2^{j-N}R}(x_0)} |u-(u)_{B_{2^{j-N}R}(x_0)}|dx \sum_{k=0}^\infty 2^{-2sk} \\
	& = C \sum_{k=1}^N 2^{-2sk} \mean{B_{2^{k-N}R}(x_0)} |u-(u)_{B_{2^{k-N}R}(x_0)}|dx.
\end{align*}
Moreover, we estimate
\begin{align*}
& (2^{-N}R)^{2s} \int_{\mathbb{R}^n \setminus B_{R}(x_0)} \frac{|u(y)-(u)_{B_{2^{-N} R}(x_0)}|}{|x_0-y|^{n+2s}} dy \\
& \leq (2^{-N}R)^{2s} \int_{\mathbb{R}^n \setminus B_{R}(x_0)} \frac{|u(y)-(u)_{B_{R}(x_0)}|}{|x_0-y|^{n+2s}} dy + C 2^{-2sN} \sum_{k=1}^N |(u)_{B_{2^{k-N} R}(x_0)} - (u)_{B_{2^{k-1-N} R}(x_0)}| \\
& \leq (2^{-N}R)^{2s} \int_{\mathbb{R}^n \setminus B_{R}(x_0)} \frac{|u(y)-(u)_{B_{R}(x_0)}|}{|x_0-y|^{n+2s}} dy + C \sum_{k=1}^N 2^{-2sk} \mean{B_{2^{k-N}R}(x_0)} |u-(u)_{B_{2^{k-N}R}(x_0)}|dx.
\end{align*}
Combining the previous two displays with \eqref{First} now yields the claim, finishing the proof.
\end{proof}

The next Lemma is a higher-order version of the previous one.
\begin{lemma} \label{tailestaffine}
	Let $s \in (0,1)$, $x_0 \in \mathbb{R}^n$, $R>0$ fix some integer $N \geq 0$. For any function $u \in W^{1,1}(B_R(x_0)) \cap L^1_{2s}(\mathbb{R}^n)$, we have
	\begin{align*}
		& \mean{B_{2^{-N}R}(x_0)} |u(x)-(u)_{B_{2^{-N}R}(x_0)}-(\nabla u)_{B_{2^{-N}R}(x_0)} \cdot (x-x_0)| dx \\
		& \quad + (2^{-N}R)^{2s} \int_{\mathbb{R}^n \setminus B_{2^{-N}R}(x_0)} \frac{|u(y)-(u)_{B_{2^{-N}R}(x_0)}-(\nabla u)_{B_{2^{-N}R}(x_0)} \cdot (y-x_0)|}{|x_0-y|^{n+2s}} dy \\ 
		& \leq C 2^{-N} R \Bigg ( \sum_{k=0}^{N} 2^{(1-2s)k} \mean{B_{2^{k-N}R}(x_0)} |\nabla u-(\nabla u)_{B_{2^{k-N}R}(x_0)}|dx + 2^{(1-2s)N} \sum_{k=0}^{N} \mean{B_{2^{k-N}R}(x_0)} |\nabla u|dx \\ & \quad + (2^{-N} R)^{2s-1}  \int_{\mathbb{R}^n \setminus B_R(x_0)} \frac{|u(y)-(u)_{B_R(x_0)}|} {|x_0-y|^{n+2s}} dy \Bigg ),
	\end{align*}
	where $C$ depends only on $n$ and $s$.
\end{lemma}
\begin{proof}
	First of all, a splitting into annuli yields
	\begin{equation} \label{annuli1}
		\begin{aligned}
			& \mean{B_{2^{-N}R}(x_0)} |u(x)-(u)_{B_{2^{-N}R}(x_0)}-(\nabla u)_{B_{2^{-N}R}(x_0)} \cdot (x-x_0)| dx \\
			& \quad + (2^{-N}R)^{2s} \int_{\mathbb{R}^n \setminus B_{2^{-N}R}(x_0)} \frac{|u(y)-(u)_{B_{2^{-N}R}(x_0)}-(\nabla u)_{B_{2^{-N}R}(x_0)} \cdot (y-x_0)|}{|x_0-y|^{n+2s}} dy \\
			& \leq C \sum_{k=0}^N 2^{-2sk} \mean{B_{2^{k-N} R}(x_0)} |u(x)- \left ( u \right )_{B_{2^{-N}R}(x_0)} - \left (\nabla u \right )_{B_{2^{-N}R}(x_0)} \cdot (x-x_0)| dx \\& \quad +(2^{-N} R)^{2s} \int_{\mathbb{R}^n \setminus B_{R}(x_0)} \frac{|u(y)- \left ( u \right )_{B_{2^{-N}R}(x_0)} - \left (\nabla u \right )_{B_{2^{-N}R}(x_0)} \cdot (y-x_0)| }{|x_0-y|^{n+2s}}dy.
		\end{aligned}
	\end{equation}
	
	Let us estimate the first term on the right-hand side of (\ref{annuli1}). Observing that the affine function $\ell_0(x):= (\nabla u )_{B_{2^{-N}R}(x_0)} \cdot (x-x_0)$ satisfies
	\begin{equation} \label{zeromean}
		(\ell_0)_{B_{\rho}(x_0)} = \mean{B_{\rho}} \underbrace{(\nabla u )_{B_{2^{-N}R}(x_0)} \cdot y}_{\text{odd function}} dy 
		=0 \quad \text{for all } \rho>0,
	\end{equation}
	together with Poincar\'e's inequality for any $j \in \{ 0,...,N \}$ we obtain that
	\begin{align*}
		& \mean{B_{2^{k-N} R}(x_0)} |u- \left ( u \right )_{B_{2^{-N}R}(x_0)} - \left (\nabla u \right )_{B_{2^{-N}R}(x_0)} \cdot (x-x_0)| dx \\
		& = \mean{B_{2^{k-N} R}(x_0)} |u-\ell_0- \left ( u-\ell_0 \right )_{B_{2^{-N}R}(x_0)} | dx \\
		& \leq 2^n \sum_{j=0}^k \mean{B_{2^{j-N} R}(x_0)} |u-\ell_0- \left ( u -\ell_0 \right )_{B_{2^{j-N}R}(x_0)} | dx \\
		& \leq 2^n \sum_{j=0}^k 2^{j-N} R \mean{B_{2^{j-N} R}(x_0)} |\nabla (u-\ell_0) | dx \\
		& = 2^n 2^{-N} R \sum_{j=0}^k 2^{j} \mean{B_{2^{j-N} R}(x_0)} |\nabla u - (\nabla u )_{B_{2^{-N}R}(x_0)}| dx.
	\end{align*}
	Along with reverting the order of summation and summing the geometric series we deduce that
	\begin{align*}
		& \sum_{k=0}^N 2^{-2sk} \mean{B_{2^{k-N} R}(x_0)} |u(x)- \left ( u \right )_{B_{2^{-N}R}(x_0)} - \left (\nabla u \right )_{B_{2^{-N}R}(x_0)} \cdot (x-x_0)| dx \\
		& \leq C 2^{-N} R \sum_{k=0}^N 2^{-2sk} \sum_{j=0}^k 2^{j} \mean{B_{2^{j-N} R}(x_0)} |\nabla u - (\nabla u )_{B_{2^{-N}R}(x_0)}| dx \\
		& \leq C 2^{-N} R \sum_{j=0}^N 2^{(1-2s)j} \mean{B_{2^{j-N} R}(x_0)} |\nabla u - (\nabla u )_{B_{2^{-N}R}(x_0)}| dx \\
		& \leq C 2^{-N} R \sum_{j=0}^N 2^{(1-2s)j} \sum_{k=0}^j \mean{B_{2^{k-N} R}(x_0)} |\nabla u - (\nabla u )_{B_{2^{k-N}R}(x_0)}| dx \\ 
		& \leq C 2^{-N} R \sum_{k=0}^N 2^{(1-2s)k} \mean{B_{2^{k-N} R}(x_0)} |\nabla u - (\nabla u )_{B_{2^{k-N}R}(x_0)}| dx.
	\end{align*}
	For the second term on the right-hand side of (\ref{annuli1}), we use the classical Poincar\'e inequality to obtain
	\begin{align*}
		& (2^{-N} R)^{2s} \int_{\mathbb{R}^n \setminus B_{R}(x_0)} \frac{|u(y)- \left ( u \right )_{B_{2^{-N}R}(x_0)} - \left (\nabla u \right )_{B_{2^{-N}R}(x_0)} \cdot (y-x_0)| }{|x_0-y|^{n+2s}}dy \\
		& \leq (2^{-N} R)^{2s} \int_{\mathbb{R}^n \setminus B_{R}(x_0)} \frac{|u(y)- \left ( u \right )_{B_{R}(x_0)}|}{|x_0-y|^{n+2s}}dy \\
		& \quad + C 2^{-2sN} \sum_{k=1}^N |\left (u \right )_{B_{2^{k-N}R}(x_0)} -\left (u \right )_{B_{2^{k-1-N}R}(x_0)}|+ C 2^{-2sN} R |\left (\nabla u \right )_{B_{2^{-N}R}(x_0)}| \\
		& \leq (2^{-N} R)^{2s} \int_{\mathbb{R}^n \setminus B_{R}(x_0)} \frac{|u(y)- \left ( u \right )_{B_{R}(x_0)}|}{|x_0-y|^{n+2s}}dy \\
		& \quad + C 2^{-2s N} \sum_{k=1}^N \mean{B_{2^{k-N}R}(x_0)}|u -\left (u \right )_{B_{2^{k-N}R}(x_0)}|dx+ C 2^{-N} R 2^{(1-2s)N} |\left (\nabla u \right )_{B_{2^{-N} R}(x_0)}| \\
		& \leq (2^{-N} R)^{2s} \int_{\mathbb{R}^n \setminus B_{R}(x_0)} \frac{|u(y)- \left ( u \right )_{B_{R}(x_0)}|}{|x_0-y|^{n+2s}}dy + C 2^{-N} R 2^{-2sN} \sum_{k=1}^N 2^k \mean{B_{2^{k-N}R}(x_0)}|\nabla u|dx \\
		& \quad + C 2^{-N} R 2^{(1-2s) N} \mean{B_{2^{-N}R}(x_0)}|\nabla u|dx \\
		& \leq (2^{-N} R)^{2s} \int_{\mathbb{R}^n \setminus B_{R}(x_0)} \frac{|u(y)- \left ( u \right )_{B_{R}(x_0)}|}{|x_0-y|^{n+2s}}dy + C 2^{-N} R 2^{(1-2s) N} \sum_{k=0}^N \mean{B_{2^{k-N}R}(x_0)}|\nabla u|dx .
	\end{align*}
Combining the previous two displays with \eqref{annuli1} now proves the claim.
\end{proof}

\subsection{Some known regularity results}
In this section, we mention some regularity results which can easily be derived by taking into account known results and techniques from the literature.

\begin{proposition}[Local boundedness]\label{thm:bnd}
Let $s \in (0,1)$, $A \in \mathcal{L}_0(\Lambda)$, $r>0$, $x_0 \in \ern$, consider some $\mu \in L^p(B_r(x_0))$ for some $p>\frac{n}{2s}$ and assume that $u \in W^{s,2}(B_r(x_0)) \cap L^1_{2s}(\ern)$ is a weak solution of $L_A u = \mu$ in $B_r(x_0)$. Then the following estimate holds true
$$
\sup_{B_{r/2}(x_0)} |u| \leq C \left (  \mean{B_r(x_0)} |u|  dx + r^{2s} \int_{\ern \setminus B_r(x_0)} \frac{|u(y)|}{|x_0-y|^{n+2s}}dy + r^{2s-\frac{n}{p}} ||\mu||_{L^p(B_r(x_0))} \right) ,
$$
where the constant $C$ depends only on $n,s,\Lambda,p$.
\end{proposition}
\begin{proof}
	As discussed in \cite[Theorem 2.11]{MeH}, the techniques developed in \cite{DKP,BP} can easily be modified in order to obtain the estimate
	$$
	\sup_{B_{r/2}(x_0)} |u| \leq C \left (  \left (\mean{B_r(x_0)} |u|^2  dx \right )^{1/2} + r^{2s} \int_{\ern \setminus B_r(x_0)} \frac{|u(y)|}{|x_0-y|^{n+2s}}dy + r^{2s-\frac{n}{p}} ||\mu||_{L^p(B_r(x_0))} \right) ,
	$$
	where $C$ depends only on $n,s,\Lambda,p$. The $L^2$-average of $u$ on the right-hand side of the previous inequality can then be replaced by the $L^1$ average of $u$ in view of applying Young's inequality along with an iteration argument as done in \cite[Corollary 2.1]{KMS2}, so that the claim holds.
\end{proof}
The next result yields H\"older estimates in the case when the coefficient $A$ is locally close to being translation invariant and is essentially proved in \cite{MeH}.
\begin{proposition}[Cordes-Nirenberg-type estimate] \label{homreg1}
	Let $s \in (0,1)$, $\Lambda \geq 1$ and $p>\frac{n}{2s}$. Then for any $\beta \in \left (0,\min \left \{2s-\frac{n}{p},1 \right \} \right)$, there exists some small enough $\delta=\delta(\beta,n,s,\Lambda,p)>0$ such that the following is true. Assume that $A \in \mathcal{L}_0(\Lambda)$ and that $\mu \in L^p(B_2)$. Moreover, assume that there exists a coefficient $\widetilde A \in \mathcal{L}_1(B_1,\Lambda)$ such that
	\begin{equation} \label{boundsxzs1}
		||A-\widetilde A||_{L^\infty(\mathbb{R}^n \times \mathbb{R}^n)} \leq \delta.
	\end{equation}
	Then for any weak solution $u \in W^{s,2}(B_2) \cap L^1_{2s}(\mathbb{R}^n)$ of
	$	L_A u=\mu \text{ in } B_2$,
	we have
	\begin{equation} \label{Holdest}
	[u]_{C^{\beta}(B_{1/2})} \leq C \left ( \mean{B_2} |u|dx + \int_{\mathbb{R}^n \setminus B_2} \frac{|u(y)|}{|y|^{n+2s}}dy + ||\mu||_{L^p(B_2)} \right),
	\end{equation}
	where $C$ depends only on $n,s,\Lambda,\beta,p$.
\end{proposition}

\begin{proof}
	Fix $\beta \in \left (0,\min \left \{2s-\frac{n}{p},1 \right \} \right)$ and let $\delta=\delta(\beta,n,s,\Lambda,p)>0$ be given by \cite[Proposition 4.2]{MeH}.
	Let $$M:= \sup_{B_1} |u| + \int_{\mathbb{R}^n \setminus B_1} \frac{|u(y)|}{|y|^{n+2s}}dy + \frac{||\mu||_{L^p(B_1)}}{\delta} $$
	and observe that $u_M:=u(x)/M$ belongs to $W^{s,2}(B_2) \cap L^1_{2s}(\mathbb{R}^n)$ and is a weak solution of $L_A u=\mu_M$ in $B_2$, where $\mu_M:=\mu/M$ satisfies $$||\mu_M||_{L^q(B_1)} \leq \delta.$$ Thus, by \cite[Proposition 4.2]{MeH} we have the H\"older estimate
	$$ [u_M]_{C^{\beta}(B_{1/2})} \leq C. $$Multiplying both sides of the previous estimate by $M$ and then applying Proposition \ref{thm:bnd} now yields
	\begin{align*}
	[u]_{C^{\beta}(B_{1/2})} & \leq C \left ( \sup_{B_1} |u| + \int_{\mathbb{R}^n \setminus B_1} \frac{|u(y)|}{|y|^{n+2s}}dy + ||\mu||_{L^p(B_1)} \right) \\
	& \leq C \left ( \mean{B_2} |u|dx + \int_{\mathbb{R}^n \setminus B_2} \frac{|u(y)|}{|y|^{n+2s}}dy + ||\mu||_{L^p(B_2)} \right).
	\end{align*}
	Therefore, the proof is finished.
\end{proof}

\section{H\"older estimates for $\nabla u$} \label{C1,gamma}

\subsection{Translation invariant case}

We have the following result for nonlocal equations with globally translation invariant coefficients belonging to the class $\mathcal{L}_1(\Lambda)$ defined in Definition \ref{def:TI}, see \cite[Corollary 1.3]{DZ19} or \cite[Section 6.2]{RosOtonBounded}.
\begin{proposition} \label{TIreg}
	Let $s \in (1/2,1)$, $A_0 \in \mathcal{L}_1(\Lambda)$, $f \in L^\infty(B_1)$ and let $v \in W^{s,2}(B_1) \cap L^\infty(\mathbb{R}^n)$ be a weak solution to $L_{A_0} v = f$ in $B_1$. Then we have the estimate
	$$
		||v||_{C^{1,2s-1}(B_{1/2})} \leq C(||v||_{L^\infty(\mathbb{R}^n)} + ||f||_{L^\infty(B_1)}),
	$$
	where $C$ depends only on $n,s,\Lambda$.
\end{proposition}

By a cutoff argument, we obtain the corresponding regularity result for possibly unbounded weak solutions.
\begin{corollary} \label{TIreg1}
	Let $s \in (1/2,1)$, $A_0 \in \mathcal{L}_1(\Lambda)$ and let $v \in W^{s,2}(B_2) \cap L_{2s}^1 (\mathbb{R}^n)$ be a weak solution to $L_{A_0} v = 0$ in $B_4$. Then we have the estimate
	$$
		||v||_{C^{1,2s-1}(B_{1/2})} \leq C \left (\mean{B_4} |v|dx + \int_{\mathbb{R}^n \setminus B_4} \frac{|v(y)|}{|y|^{n+2s}} dy \right ),
	$$
	where $C$ depends only on $n,s,\Lambda$.
\end{corollary}

\begin{proof}
	Fix a cutoff function $\eta \in C^\infty(\mathbb{R}^n)$ such that $\eta=1$ in $B_{3/2}$ and $\eta=0$ in $\mathbb{R}^n \setminus B_2$. Then the function $w:=v \eta$ belongs to $W^{s,2}(B_1) \cap L^\infty(\mathbb{R}^n)$ and is a weak solution of the equation $L_{A_0} w = f$ in $B_1$, where
	$$ f(x):=\int_{\mathbb{R}^n \setminus B_{3/2}}A_0(x,y) \frac{v(y)(1-\eta(y))}{|x-y|^{n+2s}}dy.$$
	In view of Lemma \ref{tailestz}, we have
	$$ ||f||_{L^\infty(B_1)} \leq C \int_{\mathbb{R}^n \setminus B_{3/2}} \frac{|v(y)|}{|y|^{n+2s}}dy.$$
	Thus, together with Proposition \ref{TIreg}, Proposition \ref{thm:bnd} and Lemma \ref{tr} we obtain
	\begin{align*}
		||v||_{C^{1,2s-1}(B_{1/2})} & = ||w||_{C^{1,2s-1}(B_{1/2})} \\ & \leq C \left (||v||_{L^\infty(B_2)} + ||f||_{L^\infty(B_1)} \right ) \\
		& \leq C \left (\mean{B_4} |v|dx + \int_{\mathbb{R}^n \setminus B_{4}} \frac{|v(y)|}{|y|^{n+2s}}dy \right ),
	\end{align*}
	finishing the proof.
\end{proof}

\begin{corollary} \label{cor:locTIreg} 
Let $s \in (1/2,1)$, $R>0$, $x_0 \in \mathbb{R}^n$, $A_0 \in \mathcal{L}_1(\Lambda)$ and let $v \in W^{s,2}(B_1) \cap L_{2s}^1 (\mathbb{R}^n)$ be a weak solution to $L_{A_0} v = 0$ in $B_{R}(x_0)$. 
Then $v \in C^{1,2s-1}(B_{R/2}(x_0))$ and
\begin{equation} \label{C2sTI}
[\nabla v]_{C^{2s-1}(B_{R/2}(x_0))} \leq C_1 R^{-2s} \left ( \mean{B_R(x_0)} |v |   dx + R^{2s} \int_{\mathbb{R}^n \setminus B_{R}(x_0)} \frac{|v(y)|}{|x_0-y|^{n+2s}}dy \right),
\end{equation}
\begin{equation} \label{LipschitzTI}
||\nabla v||_{L^\infty(B_{R/2}(x_0))} \leq C_2 R^{-1} \left ( \mean{B_R(x_0)} |v |   dx + R^{2s} \int_{\mathbb{R}^n \setminus B_{R}(x_0)} \frac{|v(y)|}{|x_0-y|^{n+2s}}dy \right),
\end{equation}
where $C_1$ and $C_2$ depend only on $n,s,\Lambda$.
In addition, for any $z \in B_{R/2}(x_0)$, any $r>0$ with $B_r(z) \subset B_{R/2}(x_0)$ and the affine function $\ell_{z}(x):= v(z)+\nabla v(z) \cdot (x-z)$, we have
\begin{equation} \label{affdecTI}
	||v-\ell_{z}||_{L^\infty(B_{r}(z))} \leq C_3 \left (\frac{r}{R} \right )^{2s} \bigg ( \mean{B_R(x_0)} |v| dx + R^{2s} \int_{\mathbb{R}^n \setminus B_{R}(x_0)} \frac{|v(y)|}{|x_0-y|^{n+2s}}dy \bigg ) ,
\end{equation}
where $C_3$ depends only on $n,s,\Lambda$.
\end{corollary}

\begin{proof}
On the one hand, the estimates \eqref{C2sTI} and \eqref{LipschitzTI} follow from Corollary \ref{TIreg1} and Proposition \ref{thm:bnd} by standard scaling, covering and interpolation arguments similar to the ones applied in the proof of Theorem \ref{Holdgrad} in section \ref{Holdsec}. \par 
On the other hand, \eqref{affdecTI} is a standard equivalent formulation of the estimate \eqref{C2sTI}, see e.g.\ \cite[Appendix A]{FRR23}.
\end{proof}

\subsection{Comparison estimate} \label{compsec}
As we discuss in the next remark, a useful property of nonlocal operators of order larger than one with translation invariant kernels is that they kill affine functions.
\begin{remark} \label{rem:affine}
Let $s \in (1/2,1)$ and observe that for any affine function $\ell=\mathcal{A} \cdot x + \mathcal{B}$ ($\mathcal{A} \in \mathbb{R}^n$, $\mathcal{B} \in \mathbb{R}$) and any bounded open set $\Omega \subset \mathbb{R}^n$, we have 
\begin{align*}
	\int_{\Omega} \int_{\Omega} \frac{|\ell(x)-\ell(y)|^2}{|x-y|^{n+2s}}dydx \leq |\mathcal{A}|^2 \int_{\Omega} \int_{\Omega} \frac{dydx}{|x-y|^{n+2s-2}}< \infty.
\end{align*}
Since $s>1/2$, for any $x_0 \in \mathbb{R}^n$ and any $R>0$ we also have
\begin{align*}
	\int_{\mathbb{R}^n \setminus B_R(x_0)} \frac{|\ell(y)|}{|x_0-y|^{n+2s}}dy \leq \int_{\mathbb{R}^n \setminus B_R} \frac{|\mathcal{A}|}{|y|^{n+2s-1}}dy + \int_{\mathbb{R}^n \setminus B_R} \frac{|\mathcal{A} \cdot x_0|+ |\mathcal{B}|}{|y|^{n+2s}}dy< \infty.
\end{align*}
Therefore, we conclude that $\ell \in W^{s,2}_{loc}(\mathbb{R}^n) \cap L^1_{2s}(\mathbb{R}^n)$. \par In addition, observe that for any translation invariant coefficient $A_0 \in \mathcal{L}_1(\Lambda)$ with $A_0(x,y)=a(x-y)$ and any $x \in \mathbb{R}^n$, we have
\begin{align*}
	\int_{\mathbb{R}^n} \frac{A_0(x,y)}{|x-y|^{n+2s}} (\ell(x)-\ell(y)) dy
	& = \int_{\mathbb{R}^n} \frac{a(x-y)}{|x-y|^{n+2s}} (\ell(x)-\ell(y)) dy \\
	& = \int_{\mathbb{R}^n} \frac{a(z)}{|z|^{n+2s}} (\ell(x)-\ell(x-z)) dz \\
	& = \int_{\mathbb{R}^n} \frac{a(z)}{|z|^{n+2s}} \mathcal{A} \cdot z dz =0,
\end{align*}
where we used that the integrand of the last integral is odd due to the fact that $a(z)=a(-z)$ for any $z \in \mathbb{R}^n$, so that the integral vanishes. \par
Thus, for any test function $\varphi \in C_0^\infty(\mathbb{R}^n)$, we have
\begin{align*}
& \int_{\mathbb{R}^n} \int_{\mathbb{R}^n} \frac{A_0(x,y)}{|x-y|^{n+2s}} (\ell(x)-\ell(y))(\varphi(x)-\varphi(y))dydx \\
& = 2 \int_{\mathbb{R}^n} \left (\int_{\mathbb{R}^n} \frac{A_0(x,y)}{|x-y|^{n+2s}} (\ell(x)-\ell(y)) dy \right ) \varphi(x)dx =0.
\end{align*}
 \par
In other words, for any translation invariant coefficient $A_0 \in \mathcal{L}_1(\Lambda)$, any affine function $\ell$ is a weak solution of $L_{A_0} \ell =0$ in $\mathbb{R}^n$.
\end{remark}

We now prove a comparison estimate suitable to deduce first-order estimates for nonlocal equations with coefficients that are possibly not translation invariant.

\begin{lemma}[Freezing the coefficient] \label{compest}
Let $s \in (1/2,1)$, fix some integer $N \geq 0$ and let $\mu \in L^p(B_{2})$ for some $p > \frac{n}{2s}$. Moreover, assume that $A \in \mathcal{L}_0(\Lambda)$ satisfies 
\begin{equation} \label{holdkernel00}
	\sup_{x,y,h \in B_{2^{N+3}}} |A(x+h,y+h)-A(x,y)| \leq \Gamma |h|^\alpha
\end{equation} for some $\alpha \in (0,1)$ and some $\Gamma>0$. In addition, for any $\beta \in (s,\min \{2s-\alpha,2s-n/p,1\})$ there exists some small enough $\delta=\delta(n,s,\Lambda,\beta,p)>0$, such that the following is true. Suppose that there exists a coefficient $\widetilde A \in \mathcal{L}_1(B_1,\Lambda)$ such that
\begin{equation} \label{boundsxzs}
	||A-\widetilde A||_{L^\infty(\mathbb{R}^n \times \mathbb{R}^n)} \leq \delta
\end{equation} and assume that $u \in W^{s,2}(B_2) \cap L^1_{2s}(\mathbb{R}^n)$ is a weak solution of $L_A u =\mu$ in $B_2$. In addition, fix some affine function $\ell=\mathcal{A} \cdot x + \mathcal{B}$, some $z \in B_{1/8}$ and some $r \in (0,1/8]$.
Furthermore, let $v \in W^{s,2}(B_2) \cap L^1_{2s}(\mathbb{R}^n)$ be the unique weak solution of 
\begin{equation} \label{constcof5}
	\begin{cases} \normalfont
		L_{A_z} v = 0 & \text{ in } B_{r}(z) \\
		v = u _\ell & \text{ a.e. in } \mathbb{R}^n \setminus B_{r}(z),
	\end{cases}
\end{equation}
where $u_\ell := u-\ell$ and $$ A_z (x,y) :=
\frac{1}{2} \left (A(x-y+z,z)+A(y-x+z,z) \right ) .$$
Then for any $c \in \mathbb{R}$ we have
\begin{align*}
	& \left (r^{-n} \int_{\mathbb{R}^n} |u_\ell -v|^2 dx \right )^\frac{1}{2} \\
	& \leq C r^{\alpha+\beta} \left ( (\Gamma+2^{-2sN}) \sum_{k=1}^{N+1} 2^{k(\alpha -2s)} \mean{B_{2^{k}}} |u-c|dx + (\Gamma+1) \int_{\mathbb{R}^n \setminus B_{2^{N+1}}} \frac{|u(y)-c|}{|y|^{n+2s}}dy + (\Gamma+1) ||\mu||_{L^p(B_2)} \right ) \\ & \quad + Cr^{2s-\frac{n}{p} }||\mu||_{L^p(B_2)} ,
\end{align*}
where $C$ depends only on $n,s,\Lambda,\alpha,\beta,p$.
\end{lemma}

\begin{proof}
	We stress that throughout the proof, all constants $C$ will only depend on $n,s,\Lambda,\alpha,\beta,p$, but not on $\Gamma$. \par
	First of all, observe that in view of Remark \ref{rem:affine}, we have $u_\ell \in W^{s,2}(B_2) \cap L^1_{2s}(\mathbb{R}^n)$, so that by \cite[Remark 3]{existence}, the unique weak solution of (\ref{constcof5}) indeed exists. Moreover, let $\delta=\delta(n,s,\Lambda,\beta,p)>0$ be given by Proposition \ref{homreg1}. \par
	Fix some $z \in B_{1/8}$ and some $r \in (0,1/8]$. By (\ref{holdkernel00}), for all $x,y \in B_{2^{N+2}}(z)$ we have
	\begin{align*}
		|A(x-y+z,z)-A(x,y)| \leq \Gamma |z-y|^\alpha, \\ |A(z,y-x+z)-A(x,y)| \leq \Gamma |z-x|^\alpha . \end{align*}
	Thus, by additionally taking into account the symmetry of $A$, we obtain that the coefficient $ A_z$
	belongs to $\mathcal{L}_1(\Lambda)$ and satisfies 
	\begin{equation} \label{holddec}
		|A(x,y)-A_z(x,y)| \leq \Gamma \left (\frac{|z-x|^\alpha + |z-y|^\alpha}{2} \right)
	\end{equation} for all $x,y \in B_{2^{N+2}}(z)$. Therefore, for all $\rho \in (0,2^{N+2}]$ we obtain
	\begin{equation} \label{ralpha}
		||A-A_z||_{L^\infty(B_{\rho}(z) \times B_{\rho}(z))} 
		\leq \Gamma \rho^\alpha.
	\end{equation}
	
	Now observe that in view of Remark \ref{rem:affine}, $w:=u_\ell- v$ belongs to $W_0^{s,2}(B_{r}(z))$ and satisfies
	\begin{align*}
		& \int_{\mathbb{R}^n} \int_{\mathbb{R}^n} \frac{(w(x)-w(y))^2}{|x-y|^{n+2s}}dydx \\
		& \leq \Lambda \int_{\mathbb{R}^n} \int_{\mathbb{R}^n} A_z(x,y) \frac{(w(x)-w(y))^2}{|x-y|^{n+2s}}dydx \\
		& = \Lambda \bigg (\int_{\mathbb{R}^n} \int_{\mathbb{R}^n} A_z(x,y) \frac{(u(x)-u(y))(w(x)-w(y))}{|x-y|^{n+2s}}dydx \\
		& \quad - \underbrace{\int_{\mathbb{R}^n} \int_{\mathbb{R}^n} A_z(x,y) \frac{(\ell(x)-\ell (y))(w(x)-w(y))}{|x-y|^{n+2s}}dydx}_{=0} \bigg ) \\
		& \quad - \underbrace{\int_{\mathbb{R}^n} \int_{\mathbb{R}^n} A_z(x,y) \frac{(v(x)-v(y))(w(x)-w(y))}{|x-y|^{n+2s}}dydx}_{=0} \bigg ) \\
		& = \Lambda \bigg (\int_{\mathbb{R}^n} \int_{\mathbb{R}^n} (A_z(x,y)-A(x,y)) \frac{(u(x)-u(y))(w(x)-w(y))}{|x-y|^{n+2s}}dydx \\
		& \quad + \int_{\mathbb{R}^n} \int_{\mathbb{R}^n} A(x,y) \frac{(u(x)-u(y))(w(x)-w(y))}{|x-y|^{n+2s}}dydx \bigg )  \\
		& = \Lambda \underbrace{\int_{\mathbb{R}^n} \int_{\mathbb{R}^n} (A_z(x,y)-A(x,y)) \frac{(u(x)-u(y))(w(x)-w(y))}{|x-y|^{n+2s}}dydx}_{=: I} \\
		& \quad +\Lambda \underbrace{\int_{B_r(z)} \mu(x) w(x)dx}_{:=I^\prime}.
	\end{align*}
	We estimate $I$ as follows
	\begin{align*}
		I & \leq \int_{\mathbb{R}^n} \int_{\mathbb{R}^n} |A_z(x,y)-A(x,y)| \frac{|u(x)-u(y)||w(x)-w(y)|}{|x-y|^{n+2s}}dydx \\
		& \leq \underbrace{ \int_{B_{2 r}(z)} \int_{B_{2r}(z)} |A_z(x,y)-A(x,y)| \frac{|u(x)-u(y)||w(x)-w(y)|}{|x-y|^{n+2s}}dydx}_{=: I_{1}} \\
		& \quad + 2 \underbrace{ \int_{B_{r}(z)} \int_{B_{1/2} \setminus B_{2r}(z)} |A_z(x,y)-A(x,y)| \frac{|u(x)-u(z)||w(x)|}{|x-y|^{n+2s}}dydx}_{=: I_{2}} \\
		& \quad + 2 \underbrace{  \int_{B_{r}(z)} \int_{B_{1/2} \setminus B_{2r}(z)} |A_z(x,y)-A(x,y)| \frac{|u(y)-u(z)||w(x)|}{|x-y|^{n+2s}}dydx}_{=: I_{3}} \\
		& \quad + 2 \underbrace{ \int_{B_{r}(z)} \int_{\mathbb{R}^n \setminus B_{1/2}} |A_z(x,y)-A(x,y)| \frac{|u(x)-c||w(x)|}{|x-y|^{n+2s}}dydx}_{=: I_{4}} \\
		& \quad + 2 \underbrace{  \int_{B_{r}(z)} \int_{\mathbb{R}^n \setminus B_{1/2}} |A_z(x,y)-A(x,y)| \frac{|u(y)-c||w(x)|}{|x-y|^{n+2s}}dydx}_{=: I_{5}}.
	\end{align*}
	
	By using the Cauchy-Schwarz inequality, (\ref{ralpha}) and Proposition \ref{homreg1} with respect to $\beta \in (s,1)$ (which is applicable due to the smallness assumption (\ref{boundsxzs})), for $I_{1}$ we obtain
	\begin{align*}
		I_{1} & \leq ||A-A_z||_{L^\infty (B_{2r}(z) \times B_{2r}(z))} \left (\int_{B_{2r}(z)} \int_{B_{2r}(z)} \frac{(u(x)-u(y))^2}{|x-y|^{n+2s}}dydx \right )^\frac{1}{2} \\ & \quad \times \left (\int_{\mathbb{R}^n} \int_{\mathbb{R}^n} \frac{(w(x)-w(y))^2}{|x-y|^{n+2s}}dydx \right )^\frac{1}{2} \\
		& \leq \Gamma r^{\alpha} [u-c]_{C^{0,\beta}(B_{1/2})} \left ( \int_{B_{2r}(z)} \int_{B_{2r}(z)} \frac{dydx}{|x-y|^{n+2s-2\beta}} \right )^\frac{1}{2} \\ & \quad \times \left (\int_{\mathbb{R}^n} \int_{\mathbb{R}^n} \frac{(w(x)-w(y))^2}{|x-y|^{n+2s}}dydx \right )^\frac{1}{2} \\
		& \leq C \Gamma r^{\alpha + \beta -s +n/2} \left ( \mean{B_{2}} |u-c| dx + \int_{\mathbb{R}^n \setminus B_{2}} \frac{|u(y)-c|}{|y|^{n+2s}}dy + ||\mu||_{L^p(B_2)} \right) \\ & \quad \times \left (\int_{\mathbb{R}^n} \int_{\mathbb{R}^n} \frac{(w(x)-w(y))^2}{|x-y|^{n+2s}}dydx \right )^\frac{1}{2}.
	\end{align*}
	Let $m_0 \in \mathbb{N}$ be the smallest integer such that $2^{m_0} r >1/2$. 
	For $I_{2}$, by using Lemma \ref{tailestz}, (\ref{ralpha}), Proposition \ref{homreg1}, the Cauchy-Schwarz inequality and the fractional Friedrichs-Poincar\'e inequality (Lemma \ref{Friedrichs}), we have
	\begin{align*}
		I_{2} & \leq \int_{B_{r}(z)} \int_{B_{1/2} \setminus B_{2r}(z)} |A_z(x,y)-A(x,y)| \frac{|u(x)-u(z)||w(x)|}{|x-y|^{n+2s}}dydx \\
		& \leq \sum_{k=1}^{m_0-1} \int_{B_{r}(z)} \int_{B_{2^{k+1}r }(z) \setminus B_{2^{k}r}(z)} |A_z(x,y)-A(x,y)| \frac{|u(x)-u(z)||w(x)|}{|x-y|^{n+2s}}dydx \\
		& \leq C \sum_{k=1}^{m_0-1} 2^{-k(n+2s)} r^{-(n+2s)} \int_{B_{r}(z)} \int_{B_{2^{k+1}r }(z)} |A_z(x,y)-A(x,y)||u(x)-u(z)||w(x)|dydx \\
		& \leq C \sum_{k=1}^{m_0-1} 2^{-2sk} r^{2s} ||A-A_z||_{L^\infty (B_{2^{k+1}r }(z) \times B_{2^{k+1}r }(z))} [u]_{C^{0,\beta}(B_{1/2})} \\ & \quad \times \int_{B_{r}(z)}|x-z|^\beta |w(x)|dx \\
		& \leq C \Gamma r^{\alpha + \beta-2s +n/2} \sum_{k=2}^{\infty} 2^{k(\alpha-2s)} [u-c]_{C^{0,\beta}(B_{1/2})} \left (\int_{B_{r }(z)}|w(x)|^2dx \right )^\frac{1}{2} \\
		& \leq C \Gamma r^{\alpha + \beta-s+n/2} \left ( \mean{B_{2}} |u-c| dx + \int_{\mathbb{R}^n \setminus B_{2}} \frac{|u(y)-c|}{|y|^{n+2s}}dy + ||\mu||_{L^p(B_2)} \right) \\ & \quad \times \left (\int_{\mathbb{R}^n} \int_{\mathbb{R}^n} \frac{(w(x)-w(y))^2}{|x-y|^{n+2s}}dydx \right )^\frac{1}{2}.
	\end{align*}
	Similarly, by using Lemma \ref{tailestz}, (\ref{ralpha}), Proposition \ref{homreg1}, the Cauchy-Schwarz inequality and Lemma \ref{Friedrichs}, 
	for $I_{3}$ we obtain
	
	\begin{align*}
		I_{3} & \leq \int_{B_{r}(z)} \int_{B_{1/2} \setminus B_{2r}(z)} |A_z(x,y)-A(x,y)| \frac{|u(y)-u(z)||w(x)|}{|x-y|^{n+2s}}dydx \\
		& \leq \sum_{k=1}^{m_0-1} \int_{B_{r}(z)} \int_{B_{2^{k+1}r }(z) \setminus B_{2^{k}r}(z)} |A_z(x,y)-A(x,y)| \frac{|u(y)-u(z)||w(x)|}{|x-y|^{n+2s}}dydx \\
		& \leq C \sum_{k=1}^{m_0-1} 2^{-k(n+2s)} r^{-(n+2s)} \int_{B_{r}(z)} \int_{B_{2^{k+1}r }(z)} |A_z(x,y)-A(x,y)||u(y)-u(z)||w(x)|dydx \\
		& \leq C \sum_{k=1}^{m_0-1} 2^{-2sk} r^{2s} ||A-A_z||_{L^\infty (B_{2^{k+1}r}(z) \times B_{2^{k+1}r}(z))} [u-c]_{C^{0,\beta}(B_{4})} \\ & \quad \times \int_{B_{2^{k+1}r}(z)}|y-z|^\beta dy \int_{B_{r}(z)} |w(x)|dx \\
		& \leq C \Gamma r^{\alpha + \beta-2s +n/2} \sum_{k=2}^{\infty} 2^{k(\alpha + \beta -2s)} [u-c]_{C^{0,\beta}(B_{1/2})} \left (\int_{B_{r }(z)}|w(x)|^2dx \right )^\frac{1}{2} \\
		& \leq C \Gamma r^{\alpha + \beta-s+n/2} \left ( \mean{B_{2}} |u-c| dx + \int_{\mathbb{R}^n \setminus B_{2}} \frac{|u(y)-c|}{|y|^{n+2s}}dy + ||\mu||_{L^p(B_2)} \right) \\ & \quad \times \left (\int_{\mathbb{R}^n} \int_{\mathbb{R}^n} \frac{(w(x)-w(y))^2}{|x-y|^{n+2s}}dydx \right )^\frac{1}{2}.
	\end{align*}
	Similarly, along with Proposition \ref{thm:bnd} we estimate $I_4$ as follows
	\begin{align*}
		I_{4} & \leq \int_{B_{r}(z)} \int_{\mathbb{R}^n \setminus B_{1/2}} |A_z(x,y)-A(x,y)| \frac{|u(x)-c||w(x)|}{|x-y|^{n+2s}}dydx \\
		& \leq \sum_{k=0}^{N+1} \int_{B_{r}(z)} \int_{B_{2^{k}} \setminus B_{2^{k-1}}} |A_z(x,y)-A(x,y)| \frac{|u(x)-c||w(x)|}{|y|^{n+2s}}dydx \\
		& \quad + C \int_{B_{r}(z)} \int_{\mathbb{R}^n \setminus B_{2^{N+1}}} |A_z(x,y)-A(x,y)| \frac{|u(x)-c||w(x)|}{|y|^{n+2s}}dydx \\
		& \leq C \sum_{k=0}^{N+1} 2^{-k(n+2s)} \int_{B_{r}(z)} \int_{B_{2^{k}}} |A_z(x,y)-A(x,y)||u(x)-c||w(x)|dydx \\
		& \quad + C \Lambda \int_{B_{r}(z)}|u(x)-c||w(x)|dx \int_{\mathbb{R}^n \setminus B_{2^{N+1}}} \frac{dy}{|y|^{n+2s}} \\
		& \leq C \sum_{k=0}^{N+1} 2^{-2sk} ||A-A_z||_{L^\infty (B_{2^{k+1}}(z) \times B_{2^{k+1}}(z))} \int_{B_{r}(z)}|u(x)-c||w(x)|dx \\
		& \quad + C 2^{-2sN} \int_{B_{r}(z)}|u(x)-c||w(x)|dx \\
		& \leq C \Gamma r^{n/2} \sum_{k=0}^{\infty} 2^{k(\alpha-2s)} ||u-c||_{L^\infty(B_1)} \left (\int_{B_{r }(z)}|w(x)|^2dx \right )^\frac{1}{2} \\
		& \quad + C 2^{-2sN} r^{n/2} ||u-c||_{L^\infty(B_1)} \left (\int_{B_{r}(z)}|w(x)|^2dx \right )^\frac{1}{2} \\
		& \leq C (\Gamma+2^{-2sN}) r^{n/2+s} \left (\mean{B_{2}} |u-c| dx + \int_{\mathbb{R}^n \setminus B_{2}} \frac{|u(y)-c|}{|y|^{n+2s}}dy + ||\mu||_{L^p(B_2)} \right ) \\ &\quad \times \left (\int_{\mathbb{R}^n} \int_{\mathbb{R}^n} \frac{(w(x)-w(y))^2}{|x-y|^{n+2s}}dydx \right )^\frac{1}{2}.
	\end{align*}
	Finally, we also estimate $I_5$ in a similar manner:
	\begin{align*}
		I_{5} & \leq \int_{B_{r}(z)} \int_{\mathbb{R}^n \setminus B_{1/2}} |A_z(x,y)-A(x,y)| \frac{|u(y)-c||w(x)|}{|x-y|^{n+2s}}dydx \\
		& \leq \sum_{k=0}^{N+1} \int_{B_{r}(z)} \int_{B_{2^{k}} \setminus B_{2^{k-1}}} |A_z(x,y)-A(x,y)| \frac{|u(y)-c||w(x)|}{|y|^{n+2s}}dydx \\
		& \quad + C \int_{B_{r}(z)} \int_{\mathbb{R}^n \setminus B_{2^{N+1}}} |A_z(x,y)-A(x,y)| \frac{|u(y)-c||w(x)|}{|y|^{n+2s}}dydx \\
		& \leq C \sum_{k=0}^{N+1} 2^{-k(n+2s)} \int_{B_{r}(z)} \int_{B_{2^{k}}} |A_z(x,y)-A(x,y)||u(y)-c||w(x)|dydx \\
		& \quad + C \Lambda \int_{B_{r}(z)} \int_{\mathbb{R}^n \setminus B_{2^{N+1}}} \frac{|u(y)-c||w(x)|}{|y|^{n+2s}}dydx  \\
		& \leq C \sum_{k=0}^{N+1} 2^{-2sk} ||A-A_z||_{L^\infty (B_{2^{k+1}}(z) \times B_{2^{k+1}}(z))} \mean{B_{2^{k}}} |u(y)-c|dy \int_{B_{r}(z)} |w(x)|dx \\
		& \quad + C \int_{\mathbb{R}^n \setminus B_{2^{N+1}}} \frac{|u(y)-c|}{|y|^{n+2s}}dy \int_{B_{r}(z)} |w(x)|dx \\
		& \leq C \Gamma r^{n/2} \sum_{k=0}^{N+1} 2^{k(\alpha -2s)} \mean{B_{2^{k}}} |u(y)-c|dy \left (\int_{B_{r }(z)}|w(x)|^2dx \right )^\frac{1}{2} \\
		& \quad +C r^{n/2} \int_{\mathbb{R}^n \setminus B_{2^{N+1}}} \frac{|u(y)-c|}{|y|^{n+2s}}dy \left (\int_{B_{r }(z)}|w(x)|^2dx \right )^\frac{1}{2} \\
		& \leq C r^{n/2+s} \left (\Gamma \sum_{k=1}^{N+1} 2^{k(\alpha -2s)} \mean{B_{2^{k}}} |u(y)-c|dy  + \int_{\mathbb{R}^n \setminus B_{2^{N+1}}} \frac{|u(y)-c|}{|y|^{n+2s}}dy \right) \\ & \quad \times \left (\int_{\mathbb{R}^n} \int_{\mathbb{R}^n} \frac{(w(x)-w(y))^2}{|x-y|^{n+2s}}dydx \right )^\frac{1}{2} .
	\end{align*}

	Next, using that $p>\frac{n}{2s}>\frac{2n}{n+2s}$, in view of H\"older's inequality and the fractional Sobolev inequality (see \cite[Theorem 6.5]{Hitch}), for $I^\prime$ we have
	\begin{align*}
		I^\prime & \leq \left (\int_{B_r(z)} |\mu|^\frac{2n}{n+2s} dx \right )^\frac{n+2s}{2n} \left (\int_{B_r(z)} |w|^\frac{2n}{n-2s} dx \right )^\frac{n-2s}{2n} \\
		& \leq C r^{2s-\frac{n}{p}+\frac{n}{2}} \left (\int_{B_r(z)} |\mu|^p dx \right )^\frac{1}{p} \left (\int_{\mathbb{R}^n} |w|^\frac{2n}{n-2s} dx \right )^\frac{n-2s}{2n} \\
		& \leq C r^{2s-\frac{n}{p}+\frac{n}{2}} ||\mu||_{L^p(B_2)} \left (\int_{\mathbb{R}^n} \int_{\mathbb{R}^n} \frac{(w(x)-w(y))^2}{|x-y|^{n+2s}}dydx \right )^\frac{1}{2}.
	\end{align*}
	
	By combining the above estimates with the fractional Friedrichs-Poincar\'e inequality, we arrive at the comparison estimate
	\begin{align*}
		& \left (\int_{\mathbb{R}^n} |u_\ell -v|^2 dx \right )^\frac{1}{2} \\
		& \leq r^s \left (\int_{\mathbb{R}^n} \int_{\mathbb{R}^n} \frac{(w(x)-w(y))^2}{|x-y|^{n+2s}}dydx \right )^\frac{1}{2} \\
		& \leq C \Gamma r^{\alpha + \beta+n/2} \left ( \mean{B_{1}} |u-c| dx + \int_{\mathbb{R}^n \setminus B_{1}} \frac{|u(y)-c|}{|y|^{n+2s}}dy + ||\mu||_{L^p(B_2)} \right) \\
		& \quad + C (\Gamma+2^{-2sN}) r^{n/2+2s} \left (\mean{B_{2}} |u-c| dx + \int_{\mathbb{R}^n \setminus B_{2}} \frac{|u(y)-c|}{|y|^{n+2s}}dy + ||\mu||_{L^p(B_2)} \right ) \\
		& \quad + C r^{n/2+2s} \left (\Gamma \sum_{k=1}^{N+1} 2^{k(\alpha -2s)} \mean{B_{2^{k}}} |u(y)-c|dy  + \int_{\mathbb{R}^n \setminus B_{2^{N+1}}} \frac{|u(y)-c|}{|y|^{n+2s}}dy \right) \\
		& \quad + C r^{2s-\frac{n}{p}+\frac{n}{2}} ||\mu||_{L^p(B_2)} \\
		& \leq C \Gamma r^{\alpha + \beta+n/2} \left ( \sum_{k=1}^{N+1} 2^{-2sk} \mean{B_{2^{k}}} |u-c|dx + \int_{\mathbb{R}^n \setminus B_{2^{N+1}}} \frac{|u(y)-c|}{|y|^{n+2s}}dy + ||\mu||_{L^p(B_2)} \right) \\
		& \quad + C (\Gamma+2^{-2sN}) r^{n/2+2s} \left ( \sum_{k=1}^{N+1} 2^{-2sk} \mean{B_{2^{k}}} |u-c|dx + \int_{\mathbb{R}^n \setminus B_{2^{N+1}}} \frac{|u(y)-c|}{|y|^{n+2s}}dy + ||\mu||_{L^p(B_2)} \right ) \\
		& \quad + C r^{n/2+2s} \left (\Gamma \sum_{k=1}^{N+1} 2^{k(\alpha -2s)} \mean{B_{2^{k}}} |u-c|dx + \int_{\mathbb{R}^n \setminus B_{2^{N+1}}} \frac{|u(y)-c|}{|y|^{n+2s}}dy \right) \\
		& \quad + C r^{2s-\frac{n}{p}+\frac{n}{2}} ||\mu||_{L^p(B_2)} \\
		& \leq C r^{\alpha+\beta+n/2} \Bigg ( (\Gamma+2^{-2sN}) \sum_{k=1}^{N+1} 2^{k(\alpha -2s)} \mean{B_{2^{k}}} |u-c|dx + (\Gamma+1) \int_{\mathbb{R}^n \setminus B_{2^N}} \frac{|u(y)-c|}{|y|^{n+2s}}dy \\ & \quad + (\Gamma+1) ||\mu||_{L^p(B_2)} \Bigg ) + C r^{2s-\frac{n}{p}+\frac{n}{2}} ||\mu||_{L^p(B_2)} ,
	\end{align*}
where $C$ depends only on $n,s,\Lambda,\alpha,p,\beta$. Dividing by $r^{n/2}$ on both sides now finishes the proof.
\end{proof}

By a scaling argument, we obtain the following comparison estimate.
\begin{corollary}[Comparison estimate] \label{compestR} 
	Let $s \in (1/2,1)$, fix an integer $N \geq 0$ and let $p>\frac{n}{2s}$. Moreover, let $\alpha \in (0,1)$, $\Gamma>0$ and fix some $\beta \in (s,\min \{2s-\alpha,2s-n/p,1\})$. Then there exists some small enough radius $R_0=R_0(n,s,\Lambda,\alpha,\Gamma,p,\beta) \in (0,1)$ such that the following is true for any $R>0$ with $2^{-N}R \in (0,R_0]$. Let $x_0 \in \mathbb{R}^n$ and assume that $A \in \mathcal{L}_0(\Lambda)$ is $C^\alpha$ in $B_{8R}(x_0)$ with respect to $\Gamma$ and let $\mu \in L^p(B_{2^{-N}R}(x_0))$. In addition, assume that $u \in W^{s,2}(B_{2^{-N} R}(x_0)) \cap L^1_{2s}(\mathbb{R}^n)$ is a weak solution of $$L_A u =\mu \text{ in } B_{2^{-N} R}(x_0).$$ Moreover, fix some affine function $\ell=\mathcal{A} \cdot x + \mathcal{B}$, some $z \in B_{2^{-N-4}R}(x_0)$ and some $r \in (0,2^{-N-4}R]$. Furthermore, let $v \in W^{s,2}(B_{2^{-N} R}(x_0)) \cap L^1_{2s}(\mathbb{R}^n)$ be the unique weak solution of 
	$$
		\begin{cases} \normalfont
			L_{A_z} v = 0 & \text{ in } B_{r}(z) \\
			v = u_\ell & \text{ a.e. in } \mathbb{R}^n \setminus B_{r}(z),
		\end{cases}
	$$
	where $u_\ell := u-\ell$ and $$ A_z (x,y) :=
	\frac{1}{2} \left (A(x-y+x_0,x_0)+A(y-x+x_0,x_0) \right ) .$$
	Then we have
	\begin{align*}
		& \left (r^{-n} \int_{\mathbb{R}^n} |u_\ell -v|^2 dx \right )^\frac{1}{2} \\
		& \leq C (1+\Gamma) \left (\frac{r}{2^{-N}R} \right )^{\alpha + \beta} \Bigg ( 2^{-\alpha N} \sum_{k=0}^{N} 2^{k(\alpha -2s)} \mean{B_{2^{k-N}R}(x_0)} |u-c|dx \\ & \quad + (2^{-N}R)^{2s} \int_{\mathbb{R}^n \setminus B_{R}(x_0)} \frac{|u(y)-c|}{|x_0-y|^{n+2s}}dy + (2^{-N}R)^{2s} ||\mu||_{L^p(B_{2^{-N}R}(x_0))} \Bigg )\\
		& \quad + C \left (\frac{r}{2^{-N}R} \right )^{2s-\frac{n}{p}} (2^{-N}R)^{2s-\frac{n}{p}} ||\mu||_{L^p(B_{2^{-N}R}(x_0))},
	\end{align*}
	where $C$ depends only on $n,s,\Lambda,\alpha,p,\beta$.
\end{corollary}

\begin{proof}
	Consider the scaled functions $u_1,v_1,\ell_1 \in W^{s,2}(B_{2}) \cap L_{2s}^1(\mathbb{R}^n)$ given by
	$$ u_1(x):=u(2^{-N-1}Rx+x_0), \quad v_1(x):=v(2^{-N-1}Rx+x_0),\quad \ell_1(x) :=\ell (2^{-N-1}Rx+x_0)$$
	and note that $u_1$ is a weak solution of $L_{A_1} u_1 = \mu_1$ in $B_{2}$, where
	$$A_1(x,y):= A(2^{-N-1}R x+x_0,2^{-N-1}R y+x_0), \quad \mu_1(x):=f(2^{-N-1}Rx+x_0),$$
	while $v_1$ is the weak solution of
	$$
		\begin{cases} \normalfont
			L_{(A_z)_1} v_1 = 0 & \text{ in } B_{2^{N+1} R^{-1} r}(2^{N+1} R^{-1} z -x_0) \\
			v = u_1-\ell_1 & \text{ a.e. in } \mathbb{R}^n \setminus B_{2^{N+1} R^{-1} r}(2^{N+1} R^{-1} z -x_0),
		\end{cases}
	$$
	where $$(A_z)_1(x,y):= A_z(2^{-N-1}R x+x_0,2^{-N-1}R y+x_0).$$
	First of all, observe that $2^{N+1} R^{-1} r \leq 1/8$ and $2^{N+1} R^{-1} z -x_0 \in B_{1/8}$. \par Also, observe that in view of the assumption that $A$ is $C^\alpha$ in $B_{8R}(x_0)$ with respect to $\Gamma$, we have
	\begin{equation} \label{holdkernel0}
		\sup_{x,y \in B_{4R}(x_0)} \sup_{h \in B_{4R}} |A(x+h,y+h)-A(x,y)| \leq \Gamma |h|^\alpha.
	\end{equation}
	Next, fix $x,y,h \in B_{2^{N+3}}$ and observe that in view of \eqref{holdkernel0} with $x$ replaced by $2^{-N-1}R x + x_0 \in B_{4R}(x_0)$, with $y$ replaced by $2^{-N-1}R y + x_0 \in B_{4R}(x_0)$ and with $h$ replaced by $2^{-N-1}R h \in B_{4R}$, $A_1 \in \mathcal{L}_0(\Lambda)$ satisfies
	\begin{align*}
		& |A_1(x+h,y+h)-A_1(x,y)| \\ & = |A(2^{-N-1}R x+x_0 +2^{-N-1}R h,2^{-N-1}R y+x_0 +2^{-N-1}R h)-A(2^{-N-1}R x+x_0,2^{-N-1}R y+x_0)| \\ & \leq \Gamma (2^{-N-1}R)^\alpha |h|^\alpha .
	\end{align*}
	Thus, $A_1$ satisfies the condition (\ref{holdkernel00}) from Lemma \ref{compest} with $\Gamma$ replaced by $\Gamma (2^{-N-1}R)^\alpha$. Now let $\delta=\delta(n,s,\Lambda,\beta,p)>0$ be given by Proposition \ref{homreg1}. Defining $$R_0:=\min \left \{ \left ( \frac{\delta}{\Gamma} \right )^{1/\alpha},\frac{1}{2} \right \}$$ and $$ \widetilde A_1 (x,y) := \begin{cases} \normalfont
		\frac{1}{2}(A_1(x-y,0)+A_1(0,y-x)) & \text{if } (x,y) \in B_1 \times B_1 \\
		A(x,y) & \text{if } (x,y) \notin B_1\times B_1,
	\end{cases} $$ we observe that
	$$||A_1-\widetilde A_1||_{L^\infty(\mathbb{R}^n \times \mathbb{R}^n)} \leq \Gamma (2^{-N}R)^\alpha \leq \Gamma R_0^\alpha \leq \delta.$$
	Thus, $A_1$ also satisfies the assumption (\ref{boundsxzs}) from Lemma \ref{compest} with respect to $\widetilde A_1 \in \mathcal{L}_1(B_1,\Lambda)$.
	Therefore, using Lemma \ref{compest}, for any $0<R \leq R_0$ we obtain the estimate
	\begin{align*}
		& \left (r^{-n} \int_{\mathbb{R}^n} |u_\ell -v|^2 dx \right )^\frac{1}{2} \\ & = \left ((2^{N+1} R^{-1} r)^{-n} \int_{\mathbb{R}^n} |u_1-\ell_1 -v_1|^2 dx \right )^\frac{1}{2} \\
		& \leq C \left (\frac{r}{2^{-N}R} \right )^{\alpha+\beta} \Bigg ( (\Gamma 2^{-\alpha N} R^\alpha+2^{-2sN}) \sum_{k=1}^{N+1} 2^{k(\alpha -2s)} \mean{B_{2^{k}}} |u_1-c|dx \\ & \quad + (\Gamma 2^{-\alpha N} R^\alpha+1) \int_{\mathbb{R}^n \setminus B_{2^{N+1}}} \frac{|u_1(y)-c|}{|y|^{n+2s}}dy + (\Gamma+1) 2^{-\alpha N} R^\alpha||\mu_1||_{L^q(B_2)} \Bigg ) \\ & \quad + C\left (\frac{r}{2^{-N}R} \right )^{2s-\frac{n}{q}} ||\mu_1||_{L^p(B_2)} \\
		& = C \left (\frac{r}{2^{-N}R} \right )^{\alpha + \beta} \Bigg ( (\Gamma 2^{-\alpha N} R^\alpha + 2^{-2sN}) \sum_{k=1}^{N+1} 2^{k(\alpha -2s)} \mean{B_{2^{k-N-1}R}(x_0)} |u-c|dx \\ & \quad + (\Gamma 2^{-\alpha N} R^\alpha +1) (2^{-N}R)^{2s} \int_{\mathbb{R}^n \setminus B_{R}(x_0)} \frac{|u(y)-c|}{|x_0-y|^{n+2s}}dy + (\Gamma+1) 2^{-(2s+\alpha) N} R^{2s-\frac{n}{p}+\alpha} ||\mu||_{L^p(B_{2^{-N}R}(x_0))} \Bigg ) \\
		& \quad + C \left (\frac{r}{2^{-N}R} \right )^{2s-\frac{n}{p}} (2^{-N}R)^{2s-\frac{n}{p}} ||\mu||_{L^p(B_{2^{-N}R}(x_0))} \\
		& \leq C (1+\Gamma) \left (\frac{r}{2^{-N}R} \right )^{\alpha + \beta} \Bigg ( 2^{-\alpha N} \sum_{k=0}^{N} 2^{k(\alpha -2s)} \mean{B_{2^{k-N}R}(x_0)} |u-c|dx \\ & \quad + (2^{-N}R)^{2s} \int_{\mathbb{R}^n \setminus B_{R}(x_0)} \frac{|u(y)-c|}{|x_0-y|^{n+2s}}dy + (2^{-N}R)^{2s-\frac{n}{p}} ||\mu||_{L^p(B_{2^{-N}R}(x_0))} \Bigg )\\
		& \quad + C \left (\frac{r}{2^{-N}R} \right )^{2s-\frac{n}{p}} (2^{-N}R)^{2s-\frac{n}{p}} ||\mu||_{L^p(B_{2^{-N}R}(x_0))},
	\end{align*}
	where $C=C(n,s,\Lambda,\alpha,p,\beta)$, which finishes the proof.
\end{proof}

\subsection{$C^{1,\gamma}$ regularity under H\"older coefficients} \label{Holdsec}

\begin{lemma}(Geometric iteration) \label{thm:GradOscDec} 
	Let $s \in (1/2,1)$, $x_0 \in \mathbb{R}^n$ and $p > \frac{n}{2s-1}$. Moreover, fix an integer $N \geq 0$ and let $$0<\gamma < \min\{\alpha,2s-n/p-1\}=:\gamma_0.$$ Then there exists some small enough radius $R_0=R_0(n,s,\Lambda,\alpha,p,\gamma,\Gamma) \in (0,1)$ such that the following is true for any $R>0$ with $2^{-N}R \in (0,R_0]$. Suppose that $A \in \mathcal{L}_0(\Lambda)$ is $C^\alpha$ in $B_{8R}(x_0)$ with respect to $\alpha \in (0,2s-1]$ and $\Gamma>0$, assume that $\mu \in L^p(B_R(x_0))$ and let $u \in W^{s,2}(B_{2^{-N} R}(x_0)) \cap L^1_{2s}(\ern)$ be a weak solution to $$L_{A} u = \mu \text{ in } B_{2^{-N} R}(x_0).$$
	Then there is a large enough integer $m \geq 4$ that depends only on $n,s,\Lambda,\alpha,p,\gamma$, such that for any affine function $\ell(x)=\mathcal{A} \cdot x+\mathcal{B} $, any $z \in B_{2^{-N-4} R}(x_0)$ and all integers $j \geq 1$, there exists an affine function $\ell_j(x)=\mathcal{A}_j \cdot (x-z)+\mathcal{B}_j $ such that for any $c \in \mathbb{R}$, we have
	\begin{equation} \label{gradoscdecay}
		\begin{aligned}
			& \sum_{k=0}^{mj} 2^{-2sk} \mean{B_{2^{k-mj-N} R}(z)} |u_\ell-{\ell_j}|dx \\ & \quad + (2^{-mj-N} R)^{2s} \int_{\mathbb{R}^n \setminus B_{2^{-N} R}(z)} \frac{|u_\ell (y)-{\ell_j}(y)| }{|z-y|^{n+2s}}dy \\ & \leq 2^{(n+2+\gamma)m} 2^{-(1+\gamma) mj} \Bigg ( \mean{B_{2^{-N} R}(x_0)} |u_\ell| dx + (2^{-N} R)^{2s} \int_{\mathbb{R}^n \setminus B_{2^{-N} R}(x_0)} \frac{|u_\ell(y)| }{|x_0-y|^{n+2s}}dy \\
			& \quad + (1+\Gamma) \Bigg ( 2^{-\alpha N} \sum_{k=0}^{N} 2^{k(\alpha -2s)} \mean{B_{2^{k-N}R}(x_0)} |u-c|dx \\ & \quad + (2^{-N}R)^{2s} \int_{\mathbb{R}^n \setminus B_{R}(x_0)} \frac{|u(y)-c|}{|x_0-y|^{n+2s}}dy + (2^{-N}R)^{2s-\frac{n}{p}} ||\mu||_{L^p(B_{2^{-N}R}(x_0))} \Bigg ) \Bigg ),
		\end{aligned}
	\end{equation}
	where $u_{\ell}:=u-\ell$.
\end{lemma}

\begin{proof}
	First of all, let $$\beta:=1+\frac{\gamma_0+\gamma}{2}-\alpha \in (s,\min \{2s-\alpha,2s-n/p,1\})$$ and consider the corresponding maximal radius $R_0=R_0(n,s,\Lambda,\alpha,p,\beta,\Gamma)>0$ given by Corollary \ref{compestR} and fix some $R \in (0,R_0]$. 
	Moreover, let the integer $m=m(n,s,\Lambda,\alpha,p,\gamma) \geq 4$ to be chosen large enough. \par 
	We now proceed by induction in $j$. In the case when $j=1$, we set $\ell_1:=0$ and estimate
	\begin{align*}
		& \sum_{k=0}^{m} 2^{-2sk} \mean{B_{2^{k-m-N} R}(z)} |u_\ell-{\ell_1}|dx \\ & \quad + (2^{-m-N} R)^{2s} \int_{\mathbb{R}^n \setminus B_{2^{-N}R}(z)} \frac{|u_\ell-{\ell_1}(y)| }{|z-y|^{n+2s}}dy \\ 
		& \leq \sum_{k=0}^{m} 2^{-2sk} 2^{(m-k)n} \mean{B_{2^{-N}R}(x_0)} |u_\ell|dx + (2^{-m-N} R)^{2s} \int_{\mathbb{R}^n \setminus B_{2^{-N}R}(x_0)} \frac{|u_\ell(y)| }{|x_0-y|^{n+2s}}dy \\ 
		& \leq \underbrace{\sum_{k=0}^{\infty} 2^{-(n+2s)k}}_{\leq 2} 2^{(n+1+\gamma)m} 2^{-(1+\gamma) m} \mean{B_{2^{-N}R}(x_0)} |u_\ell|dx \\ & \quad + 2^{-(1+\gamma)m} (2^{-N}R)^{2s} \int_{\mathbb{R}^n \setminus B_{2^{-N}R}(x_0)} \frac{|u_\ell(y)| }{|x_0-y|^{n+2s}}dy \\ 
		& \leq 2^{(n+2+\gamma)m} 2^{-(1+\gamma) m} \left ( \mean{B_{2^{-N}R}(x_0)} |u_\ell| dx + (2^{-N}R)^{2s} \int_{\mathbb{R}^n \setminus B_{2^{-N}R}(x_0)} \frac{|u_\ell(y)| }{|x_0-y|^{n+2s}}dy \right ) .
	\end{align*}
	
	Next, suppose that the assertion is true for some $j \geq 1$ and let us prove it for $j+1$. \par 
	In what follows, all constants $C_i$ are assumed to satisfy $C_i \geq 1$ and will only depend on $n,s,\Lambda,\alpha,p,\gamma$. Consider the frozen coefficient
	$$ \widetilde A(x,y) :=
	\frac{1}{2} \left (A(x-y+z,z)+A(y-x+z,z) \right ) $$
	
	and the weak solution $v \in W^{s,2}(B_{2^{-N}R}(x_0)) \cap L^1_{2s}(\mathbb{R}^n))$ of
	$$
		\begin{cases} \normalfont
			L_{\widetilde A} v = 0 & \text{ in } B_{2^{-mj-N} R}(z) \\
			v = u_\ell-{\ell_j} & \text{ a.e. in } \mathbb{R}^n \setminus B_{2^{-mj-N} R}(z).
		\end{cases}
	$$
	
	By Corollary \ref{compestR} with $r=2^{-mj-N} R$ and with $\ell$ replaced by the affine function $\ell+\ell_j$, we have
	\begin{equation} \label{compest0}
		\begin{aligned}
			& \left (\mean{B_{2^{-mj-N} R}(z)} |u_\ell-{\ell_j}-v|^2 dx \right )^\frac{1}{2} \\
			& \leq C_1 (1+\Gamma) 2^{-(\beta+\alpha) mj} \Bigg ( 2^{-\alpha N} \sum_{k=1}^{N} 2^{k(\alpha -2s)} \mean{B_{2^{k-N}R}(x_0)} |u-c|dx \\ & \quad + (2^{-N}R)^{2s} \int_{\mathbb{R}^n \setminus B_{R}(x_0)} \frac{|u(y)-c|}{|x_0-y|^{n+2s}}dy + (2^{-N}R)^{2s-\frac{n}{p}} ||\mu||_{L^p(B_{2^{-N}R}(x_0))} \Bigg ) \\
			& \quad + C 2^{-(2s-\frac{n}{p}) mj} (2^{-N}R)^{2s-\frac{n}{p}} ||\mu||_{L^p(B_{2^{-N}R}(x_0))} \\
			& \leq C_1 2^{-(1+\frac{\gamma_0+\gamma}{2}) mj} \textbf{R}(u),
		\end{aligned}
	\end{equation}
	where 
	\begin{align*}
	\textbf{R}(u):& =(1+\Gamma) \Bigg ( 2^{-\alpha N} \sum_{k=0}^{N} 2^{k(\alpha -2s)} \mean{B_{2^{k-N}R}(x_0)} |u-c|dx \\ & \quad + (2^{-N}R)^{2s} \int_{\mathbb{R}^n \setminus B_{R}(x_0)} \frac{|u(y)-c|}{|x_0-y|^{n+2s}}dy + (2^{-N}R)^{2s-\frac{n}{p}} ||\mu||_{L^p(B_{2^{-N}R}(x_0))} \Bigg ).
	\end{align*}
	
	Now in view of (\ref{compest0}) along with the fact that $v=u_{\ell}-\ell_j$ a.e.\ in $\mathbb{R}^n \setminus B_{2^{-mj-N} R}(z)$ and using the induction hypothesis, we deduce
	\begin{equation} \label{vest}
		\begin{aligned}
			& \mean{B_{2^{-mj-N}R}(z)}|v|dx + (2^{-mj-N} R)^{2s} \int_{\mathbb{R}^n \setminus B_{2^{-mj-N}R}(z)} \frac{|v(y)|}{|z-y|^{n+2s}}dy \\ 
			& \leq C_2 \Bigg ( \mean{B_{2^{-mj-N}R}(z)}|u_\ell-{\ell_j}-v|dx \\ & \quad + \mean{B_{2^{-mj-N}R}(z)}|u_\ell-{\ell_j}|dx+ (2^{-mj-N} R)^{2s} \int_{\mathbb{R}^n \setminus B_{2^{-mj-N}R}(z)} \frac{|u_\ell-{\ell_j}(y)|}{|z-y|^{n+2s}}dy \Bigg ) \\
			& \leq C_2 \Bigg ( \left ( \mean{B_{2^{-mj-N}R}(z)}|u_\ell-{\ell_j}-v|^2dx \right )^\frac{1}{2} \\& \quad + \mean{B_{2^{-mj-N}R}(z)}|u_\ell-{\ell_j}|dx + (2^{-mj-N} R)^{2s} \sum_{k=1}^{mj} \int_{B_{2^{k-mj-N}R}(z) \setminus B_{2^{k-1-mj-N}R}(z)} \frac{|u_\ell-{\ell_j}(y)|}{|z-y|^{n+2s}}dy  \\
			& \quad +(2^{-mj-N} R)^{2s} \int_{\mathbb{R}^n \setminus B_{2^{-N}R}(z)} \frac{|u_\ell-{\ell_j}(y)|}{|z-y|^{n+2s}}dy \Bigg ) \\
			& \leq C_3 2^{-(1+\frac{\gamma_0+\gamma}{2}) mj} \textbf{R}(u) + \sum_{k=0}^{mj} 2^{-2sk} \mean{B_{2^{k-mj-N}}(z)} |u_\ell-{\ell_j}|dx \\& \quad +(2^{-mj-N} R)^{2s} \int_{\mathbb{R}^n \setminus B_{2^{-N}R}(z)} \frac{|u_\ell-{\ell_j}(y)|}{|z-y|^{n+2s}}dy  \\
			& \leq C_4 2^{(n+2+\gamma)m} 2^{-(1+\gamma) mj} \left (\mean{B_{2^{-N}R}(x_0)} |u_\ell| dx + (2^{-N}R)^{2s} \int_{\mathbb{R}^n \setminus B_{2^{-N}R}(x_0)} \frac{|u_\ell(y)|}{|x_0-y|^{n+2s}}dy + \textbf{R}(u) \right ) .
		\end{aligned}
	\end{equation}
	
	Now combining (\ref{vest}) with Corollary \ref{cor:locTIreg}, for any $k \in \{0,...,m-1\}$ and the affine function $\widetilde \ell_{j+1}$ given by $$\widetilde \ell_{j+1}(x)=v(z)+\nabla v(z) \cdot (x-z),$$ we obtain
	\begin{equation} \label{grod2}
		\begin{aligned}
			& ||v-\widetilde \ell_{j+1}||_{L^\infty(B_{2^{k-m(j+1)-N}R}(z))} \\
			& \leq C_5 2^{(1+\frac{\gamma_0+\gamma}{2})(k-m)} \Bigg ( \mean{B_{2^{-mj-N}R}(z)}|v|dx + (2^{-mj-N} R)^{2s} \int_{\mathbb{R}^n \setminus B_{2^{-mj-N}R}(z)} \frac{|v(y)|}{|z-y|^{n+2s}}dy \Bigg )\\
			& \leq C_6 2^{(1+\frac{\gamma_0+\gamma}{2})k} 2^{(n+2+\gamma)m} 2^{\frac{\gamma-\gamma_0}{2} m} 2^{-(1+\gamma) m(j+1)} \\ & \quad \times \left (\mean{B_{2^{-N}R}(x_0)} |u_\ell| dx + (2^{-N}R)^{2s} \int_{\mathbb{R}^n \setminus B_{2^{-N}R}(x_0)} \frac{|u_\ell(y)|}{|x_0-y|^{n+2s}}dy + \textbf{R}(u) \right ) .
		\end{aligned}
	\end{equation}
	
	Next, set $$\ell_{j+1}(x):= \ell_j(x) + \widetilde \ell_{j+1}(x).$$
	In view of (\ref{compest0}) and (\ref{grod2}), for any $k \in \{0,...,m-1\}$ we obtain
	\begin{align*}
		& 2^{-2sk} \mean{B_{2^{k-m(j+1)-N} R}(z)} |u_\ell-{\ell_{j+1}}|dx \\
		& \leq 2^{-2sk} \left (\mean{B_{2^{k-m(j+1)-N} R}(z)} |u_\ell-{\ell_j}-v|dx + \mean{B_{2^{k-m(j+1)-N} R}(z)} |v-\widetilde \ell_{j+1}|dx \right ) \\
		& \leq C_7 2^{-2sk} \left (2^{(m-k)n/2} \left (\mean{B_{2^{-mj-N} R}(z)} |u_\ell-{\ell_j}-v|^2 dx \right )^\frac{1}{2} + ||v-\widetilde \ell_{j+1}||_{L^\infty(B_{2^{k-m(j+1)-N}R}(z))} \right ) \\
		& \leq C_8 2^{-2sk} 2^{mn/2} 2^{-(1+\frac{\gamma_0+\gamma}{2}) mj} \textbf{R}(u) \\ & \quad + C_8 2^{-2sk} 2^{(1+\frac{\gamma_0+\gamma}{2})k} 2^{(n+2+\gamma)m} 2^{\frac{\gamma-\gamma_0}{2} m} 2^{-(1+\gamma) m(j+1)} \\ & \quad \times \left ( \mean{B_{2^{-N}R}(x_0)} |u_\ell| dx + (2^{-N}R)^{2s} \int_{\mathbb{R}^n \setminus B_{2^{-N}R}(x_0)} \frac{|u_\ell(y)|}{|x_0-y|^{n+2s}}dy + \textbf{R}(u) \right ) \Bigg ) \\
		& \leq C_8 2^{(1+\frac{\gamma_0+\gamma}{2}-2s)k} 2^{{\frac{\gamma-\gamma_0}{2} m}} \left (2^{mn/2} 2^{-(1+\gamma) mj} + 2^{(n+2+\gamma)m} 
		2^{-(1+\gamma) m(j+1)} \right ) \\ & \quad \times \left ( \mean{B_{2^{-N}R}(x_0)} |u_\ell| dx + (2^{-N}R)^{2s} \int_{\mathbb{R}^n \setminus B_{2^{-N}R}(x_0)} \frac{|u_\ell(y)|}{|x_0-y|^{n+2s}}dy + \textbf{R}(u) \right ) \\
		& \leq C_8 2^{(1+\frac{\gamma_0+\gamma}{2}-2s)k} 2^{{\frac{\gamma-\gamma_0}{2} m}} 2^{(n+2+\gamma)m} 
		2^{-(1+\gamma) m(j+1)} \\ & \quad \times \left ( \mean{B_{2^{-N}R}(x_0)} |u_\ell| dx + (2^{-N}R)^{2s} \int_{\mathbb{R}^n \setminus B_{2^{-N}R}(x_0)} \frac{|u_\ell(y)|}{|x_0-y|^{n+2s}}dy + \textbf{R}(u) \right ).
	\end{align*}
	Summing the previous estimate over $k=0,...,m-1$ leads to
	\begin{align*}
		& \sum_{k=0}^{m-1} 2^{-2sk} \mean{B_{2^{k-m(j+1)-N} R}(z)} |u_\ell-{\ell_{j+1}}|dx \\
		& \leq C_8 2^{{\frac{\gamma-\gamma_0}{2} m}} \underbrace{\sum_{k=0}^{\infty} 2^{(1+\frac{\gamma_0+\gamma}{2}-2s)k}}_{<\infty} 2^{(n+2+\gamma)m} 
		2^{-(1+\gamma) m(j+1)} \\ & \quad \times \left ( \mean{B_{2^{-N}R}(x_0)} |u_\ell| dx + (2^{-N}R)^{2s} \int_{\mathbb{R}^n \setminus B_{2^{-N}R}(x_0)} \frac{|u_\ell(y)|}{|x_0-y|^{n+2s}}dy + \textbf{R}(u) \right ) \\
		& \leq C_9 2^{{\frac{\gamma-\gamma_0}{2} m}} 2^{(n+2+\gamma)m} 
		2^{-(1+\gamma) m(j+1)} \\ & \quad \times \left ( \mean{B_{2^{-N}R}(x_0)} |u_\ell| dx + (2^{-N}R)^{2s} \int_{\mathbb{R}^n \setminus B_{2^{-N}R}(x_0)} \frac{|u_\ell(y)|}{|x_0-y|^{n+2s}}dy + \textbf{R}(u) \right ).
	\end{align*}
	Choosing $m$ large enough such that $C_9 2^{{\frac{\gamma-\gamma_0}{2} m}} \leq 1/3$ now yields
	\begin{equation} \label{locdec}
		\begin{aligned}
			& \sum_{k=0}^{m-1} 2^{-2sk} \mean{B_{2^{k-m(j+1)-N} R}(z)} |u_\ell-{\ell_{j+1}}|dx \\
			& \leq \frac{1}{3} 2^{(n+2+\gamma)m} 
			2^{-(1+\gamma) m(j+1)} \\ & \quad \times \left ( \mean{B_{2^{-N}R}(x_0)} |u_\ell| dx + (2^{-N}R)^{2s} \int_{\mathbb{R}^n \setminus B_{2^{-N}R}(x_0)} \frac{|u_\ell(y)|}{|x_0-y|^{n+2s}}dy + \textbf{R}(u) \right ).
		\end{aligned}
	\end{equation} 
	
	By (\ref{LipschitzTI}), Proposition \ref{thm:bnd} and (\ref{vest}), for any $k \in \{0,...,mj\}$ we have 
	\begin{align*}
		& ||\ell_{j+1}-\ell_j||_{L^\infty(B_{2^{k-mj-N}R}(z))} \\
		& = ||\widetilde \ell_{j+1}||_{L^\infty(B_{2^{k-mj-N}R}(z))} \\
		& \leq 2^{k-mj-N}R |\nabla v(z)| + |v(z)| \\
		& \leq C_{10} 2^k \Bigg ( \mean{B_{2^{-mj-N}R}(z)}|v|dx + (2^{-mj-N} R)^{2s} \int_{\mathbb{R}^n \setminus B_{2^{-mj-N}R}(z)} \frac{|v(y)|}{|z-y|^{n+2s}}dy \Bigg )\\
		& \leq C_{11} 2^{k} 2^{(n+2+\gamma)m} 
		2^{-(1+\gamma) mj} \\ & \quad \times \left ( \mean{B_{2^{-N}R}(x_0)} |u_\ell| dx + (2^{-N}R)^{2s} \int_{\mathbb{R}^n \setminus B_{2^{-N}R}(x_0)} \frac{|u_\ell(y)|}{|x_0-y|^{n+2s}}dy + \textbf{R}(u) \right ).
	\end{align*}
	
	Combining the previous display with the induction hypothesis now leads to
	\begin{align*}
		& \sum_{k=m}^{m(j+1)} 2^{-2s k} \mean{B_{2^{k-(j+1)m-N} R}(z)} |(u_\ell- \ell_{j+1})|dx \\
		& = 2^{-2sm} \sum_{k=0}^{mj} 2^{-2s k} \mean{B_{2^{k-jm-N} R}(z)} |(u_\ell- \ell_{j+1})|dx \\
		& \leq 2^{-2sm} \sum_{k=0}^{mj} 2^{-2s k} \mean{B_{2^{k-jm-N} R}(z)} |(u_\ell- \ell_{j})|dx \\
		& \quad +2^{-2sm} \sum_{k=0}^{mj} 2^{-2sk} ||\ell_{j+1}-\ell_j||_{L^\infty(B_{2^{k-mj-N}R}(z))} \\
		& \leq C_{12} \underbrace{\sum_{k=0}^{\infty} 2^{(1-2s)k}}_{<\infty} 2^{(n+2+\gamma)m} 2^{-2sm} 
		2^{-(1+\gamma) mj} \\ & \quad \times \left ( \mean{B_{2^{-N}R}(x_0)} |u_\ell| dx + (2^{-N}R)^{2s} \int_{\mathbb{R}^n \setminus B_{2^{-N}R}(x_0)} \frac{|u_\ell(y)|}{|x_0-y|^{n+2s}}dy + \textbf{R}(u) \right )\\
		& = C_{13}2^{(1+\gamma-2s)m} 2^{(n+2+\gamma)m} 
		2^{-(1+\gamma) m(j+1)} \\ & \quad \times \left ( \mean{B_{2^{-N}R}(x_0)} |u_\ell| dx + (2^{-N}R)^{2s} \int_{\mathbb{R}^n \setminus B_{2^{-N}R}(x_0)} \frac{|u_\ell(y)|}{|x_0-y|^{n+2s}}dy + \textbf{R}(u) \right ).
	\end{align*}
	Therefore, requiring $m$ to be large enough such that $C_{13}2^{(1+\gamma-2s)m} \leq 1/3$ yields
	\begin{equation} \label{middec}
		\begin{aligned}
			& \sum_{k=m}^{m(j+1)} 2^{-2s k} \mean{B_{2^{k-(j+1)m-N} R}(z)} |(u_\ell- \ell_{j+1})|dx \\
			& \leq \frac{1}{3} 2^{(n+2+\gamma)m} 
			2^{-(1+\gamma) m(j+1)} \left ( \mean{B_{R}(x_0)} |u_\ell| dx + R^{2s} \int_{\mathbb{R}^n \setminus B_{R}(x_0)} \frac{|u_\ell(y)|}{|x_0-y|^{n+2s}}dy + \textbf{R}(u) \right ).
		\end{aligned}
	\end{equation}
	
	Finally, by the induction hypothesis, (\ref{LipschitzTI}), Proposition \ref{thm:bnd} and (\ref{vest}), we observe that
	
	\begin{align*}
		& (2^{-m(j+1)-N} R)^{2s} \int_{\mathbb{R}^n \setminus B_{2^{-N}R}(z)} \frac{|u_\ell-{\ell_{j+1}}(y)| }{|z-y|^{n+2s}}dy \\
		& \leq 2^{-2sm} (2^{-mj-N} R)^{2s} \left ( \int_{\mathbb{R}^n \setminus B_{2^{-N}R}(z)} \frac{|u_\ell-{\ell_j}(y)| }{|z-y|^{n+2s}}dy
		+ \int_{\mathbb{R}^n \setminus B_{2^{-N}R}(z)} \frac{|\ell_{j+1}(y)-\ell_j(y)| }{|z-y|^{n+2s}}dy \right ) \\
		& = 2^{-2sm} (2^{-mj-N} R)^{2s} \\ & \quad \times \left ( \int_{\mathbb{R}^n \setminus B_{2^{-N}R}(z)} \frac{|u_\ell-{\ell_j}(y)| }{|z-y|^{n+2s}}dy
		+ \int_{\mathbb{R}^n \setminus B_{2^{-N}R}(z)} \frac{|\nabla v(z) \cdot (y-z)+v(z)| }{|z-y|^{n+2s}}dy \right ) \\
		& \leq 2^{-2sm} 2^{(n+2+\gamma)m} 2^{-(1+\gamma) mj} \\ & \quad \times \left ( \mean{B_{2^{-N}R}(x_0)} |u_\ell| dx + (2^{-N}R)^{2s} \int_{\mathbb{R}^n \setminus B_{2^{-N}R}(x_0)} \frac{|u_\ell(y)| }{|x_0-y|^{n+2s}}dy + \textbf{R}(u) \right ) \\
		& \quad + 2^{-2sm} (2^{-mj-N} R)^{2s} |\nabla v(z)| \int_{\mathbb{R}^n \setminus B_{2^{-N}R}(z)} \frac{dy}{|z-y|^{n+2s-1}}dy \\
		& \quad + 2^{-2sm} (2^{-mj-N} R)^{2s} |v(z)| \int_{\mathbb{R}^n \setminus B_{2^{-N}R}(z)} \frac{dy}{|z-y|^{n+2s}}dy \\
		& \leq 2^{(1+\gamma-2s)m} 2^{(n+2+\gamma)m} 2^{-(1+\gamma) m(j+1)} \\ & \quad \times \left ( \mean{B_{2^{-N} R}(x_0)} |u_\ell| dx + (2^{-N} R)^{2s} \int_{\mathbb{R}^n \setminus B_{2^{-N} R}(x_0)} \frac{|u_\ell(y)| }{|x_0-y|^{n+2s}}dy + \textbf{R}(u) \right ) \\
		& \quad + C_{14} 2^{-2sm} (2^{-mj})^{2s} \left ( 2^{-N} R |\nabla v(z)| + |v(z)| \right ) \\
		& \leq 2^{(1+\gamma-2s)m} 2^{(n+2+\gamma)m} 2^{-(1+\gamma) m(j+1)} \\ & \quad \times \left ( \mean{B_{2^{-N}R}(x_0)} |u_\ell| dx + (2^{-N} R)^{2s} \int_{\mathbb{R}^n \setminus B_{2^{-N}R}(x_0)} \frac{|u_\ell(y)| }{|x_0-y|^{n+2s}}dy + \textbf{R}(u) \right ) \\
		& \quad + C_{15} 2^{-2sm} \left (\mean{B_{2^{-mj-N}R}(z)}|v|dx + (2^{-mj-N} R)^{2s} \int_{\mathbb{R}^n \setminus B_{2^{-mj-N}R}(z)} \frac{|v(y)|}{|z-y|^{n+2s}}dy \right )\\
		& \leq C_{16} 2^{(1+\gamma-2s)m} 2^{(n+2+\gamma)m} 2^{-(1+\gamma) m{j+1}} \\ & \quad \times \left ( \mean{B_{2^{-N}R}(x_0)} |u_\ell| dx + (2^{-N} R)^{2s} \int_{\mathbb{R}^n \setminus B_{2^{-N}R}(x_0)} \frac{|u_\ell(y)| }{|x_0-y|^{n+2s}}dy + \textbf{R}(u) \right ) .
	\end{align*}
	
	Thus, requiring $m$ to be large enough such that $C_{16}2^{(1+\gamma-2s)m} \leq 1/3$ yields
	\begin{equation} \label{taildec}
		\begin{aligned}
			& (2^{-m(j+1)-N} R)^{2s} \int_{\mathbb{R}^n \setminus B_{2^{-N}R}(z)} \frac{|u_\ell-{\ell_{j+1}}(y)| }{|z-y|^{n+2s}}dy \\
			& \leq \frac{1}{3} 2^{(n+2+\gamma)m} 2^{-(1+\gamma) m(j+1)} \\& \quad \times \left ( \mean{B_{2^{-N}R}(x_0)} |u_\ell| dx + (2^{-N} R)^{2s} \int_{\mathbb{R}^n \setminus B_{2^{-N}R}(x_0)} \frac{|u_\ell(y)| }{|x_0-y|^{n+2s}}dy + \textbf{R}(u) \right ) .
		\end{aligned}
	\end{equation}
	By combining the previous display with the estimates (\ref{locdec}), (\ref{middec}) and (\ref{taildec}), we arrive at
	\begin{align*}
		& \sum_{k=0}^{m(j+1)} 2^{-2sk} \mean{B_{2^{k-m(j+1)} 2^{-N}R}(z)} |u_\ell-{\ell_{j+1}}|dx \\ & \quad + (2^{-m(j+1)} 2^{-N}R)^{2s} \int_{\mathbb{R}^n \setminus B_{2^{-N}R}(z)} \frac{|u_\ell-{\ell_{j+1}}(y)| }{|z-y|^{n+2s}}dy \\ & \leq 2^{(n+2+\gamma)m} 2^{-(1+\gamma) m(j+1)} \left ( \mean{B_{2^{-N}R}(x_0)} |u_\ell| dx + (2^{-N} R)^{2s} \int_{\mathbb{R}^n \setminus B_{2^{-N}R}(x_0)} \frac{|u_\ell(y)| }{|x_0-y|^{n+2s}}dy + \textbf{R}(u) \right ) ,
	\end{align*}
	which proves the desired estimate (\ref{gradoscdecay}) for $j+1$. Thus, the induction is complete and the proof is finished.
\end{proof}

Due to its flexibility, the gradient oscillation decay estimate given by the next Corollary will be a central tool in the proofs of all our main results.
\begin{corollary}[Gradient oscillation decay] \label{thm:GradOscDec1} 
	Let $s \in (1/2,1)$, $R>0$, $x_0 \in \mathbb{R}^n$, $p > \frac{n}{2s-1}$ and assume that $\mu \in L^p(B_R(x_0))$. Moreover, fix an integer $N \geq 0$ and let $$0<\gamma < \min\{\alpha,2s-n/p-1\}=:\gamma_0.$$ Then there exists some small enough radius $R_0=R_0(n,s,\Lambda,\alpha,p,\gamma,\Gamma) \in (0,1)$ such that the following is true for any $R>0$ with $2^{-N}R \in (0,R_0]$. Suppose that $A \in \mathcal{L}_0(\Lambda)$ is $C^\alpha$ in $B_{8R}(x_0)$ with respect to some $\alpha \in (0,2s-1]$ and some $\Gamma>0$ and let $u \in W^{s,2}(B_{2^{-N}R}(x_0)) \cap L^1_{2s}(\ern)$ be a weak solution to $$L_{A} u = \mu \text{ in } B_{2^{-N} R}(x_0).$$
	Then for any affine function $\ell(x)=\mathcal{A} \cdot x+\mathcal{B} $, any $c \in \mathbb{R}$ and any $\rho \in (0,1/16]$, we have
	\begin{equation} \label{gradoscdecay1}
		\begin{aligned}
			& \osc_{B_{\rho 2^{-N}R}(x_0)} \nabla u \\ & \leq C \rho^\gamma (2^{-N}R)^{-1} \Bigg ( \mean{B_{2^{-N} R}(x_0)} |u_\ell| dx + (2^{-N} R)^{2s} \int_{\mathbb{R}^n \setminus B_{2^{-N} R}(x_0)} \frac{|u_\ell(y)| }{|x_0-y|^{n+2s}}dy \\
			& \quad + (1+\Gamma) \Bigg ( 2^{-\alpha N} \sum_{k=0}^{N} 2^{k(\alpha -2s)} \mean{B_{2^{k-N}R}(x_0)} |u-c|dx \\ & \quad + (2^{-N}R)^{2s} \int_{\mathbb{R}^n \setminus B_{R}(x_0)} \frac{|u(y)-c|}{|x_0-y|^{n+2s}}dy + (2^{-N}R)^{2s-\frac{n}{p}} ||\mu||_{L^p(B_R(x_0))} \Bigg ) \Bigg ),
		\end{aligned}
	\end{equation}
	where $u_{\ell}:=u-\ell$ and $C$ depends only on $n,s,\Lambda,\alpha,p,\gamma$.
\end{corollary}

\begin{proof}
	Let $m=m(n,s,\Lambda,\alpha) \geq 4$ be given by Lemma \ref{thm:GradOscDec} and fix some $z \in B_{2^{-N-4}R}(x_0)$ and some $r \leq 2^{-N-4}R$. First of all, assume that $r \geq 2^{-m-N} R$. In that case for $\ell_0=0$ we have
	\begin{align*}
		& \mean{B_{r}(z)} |u_\ell-{\ell_0}|dx \\ 
		& \leq \left (\frac{2^{-N}R}{r} \right)^n \mean{B_{2^{-N}R}(x_0)} |u_\ell|dx \\ 
		& \leq 2^{(n+1+\gamma)m} \left (\frac{r}{2^{-N} R} \right )^{1+\gamma} \mean{B_{2^{-N}R}(x_0)} |u_\ell|dx .
	\end{align*}
	If on the other hand $r < 2^{-m-N} R$, then there exists some $j \geq 1$ such that $2^{-(j+1)-N} R \leq r < 2^{-jm-N} R$. Then by Lemma \ref{thm:GradOscDec}, there exists some affine function $\ell_j$ such that
	\begin{align*}
		& \mean{B_{r}(z)} |u_\ell-{\ell_j}|dx \\
		& \leq 2^{mn} \mean{B_{2^{-mj-N} R}(z)} |u_\ell-{\ell_j}|dx \\ & \leq 2^{(2n+2+\gamma)m} \left (\frac{r}{2^{-N} R} \right )^{1+\gamma} \Bigg ( \mean{B_{2^{-N} R}(x_0)} |u_\ell| dx + (2^{-N} R)^{2s} \int_{\mathbb{R}^n \setminus B_{2^{-N} R}(x_0)} \frac{|u_\ell(y)| }{|x_0-y|^{n+2s}}dy \\
		& \quad + (1+\Gamma) \Bigg ( 2^{-\alpha N} \sum_{k=0}^{N} 2^{k(\alpha -2s)} \mean{B_{2^{k-N}R}(x_0)} |u-c|dx \\ & \quad + (2^{-N}R)^{2s} \int_{\mathbb{R}^n \setminus B_{R}(x_0)} \frac{|u(y)-c|}{|x_0-y|^{n+2s}}dy + (2^{-N}R)^{2s-\frac{n}{p}} ||\mu||_{L^p(B_R(x_0))} \Bigg ) \Bigg ).
	\end{align*}
In other words, for any $z \in B_{2^{-N-4}R}(x_0)$ and any $0<r \leq 2^{-N-4}R$, there exists an affine function $\ell_{r,z}$ such that
\begin{align*}
	& \mean{B_{r}(z)} |u_\ell-{\ell_{r,z}}|dx \\
	& \leq C_1 \left (\frac{r}{2^{-N} R} \right )^{1+\gamma} \Bigg ( \mean{B_{2^{-N} R}(x_0)} |u_\ell| dx + (2^{-N} R)^{2s} \int_{\mathbb{R}^n \setminus B_{2^{-N} R}(x_0)} \frac{|u_\ell(y)| }{|x_0-y|^{n+2s}}dy \\
	& \quad + (1+\Gamma) \Bigg (2^{-\alpha N} \sum_{k=0}^{N} 2^{k(\alpha -2s)} \mean{B_{2^{k-N}R}(x_0)} |u-c|dx \\ & \quad + (2^{-N}R)^{2s} \int_{\mathbb{R}^n \setminus B_{R}(x_0)} \frac{|u(y)-c|}{|x_0-y|^{n+2s}}dy + (2^{-N}R)^{2s-\frac{n}{p}} ||\mu||_{L^p(B_R(x_0))} \Bigg ),
\end{align*}
where $C_1:=2^{(2n+2+\gamma)m}$ depends only on $n,s,\Lambda,\alpha,p,\gamma$.
Therefore, by Campanato's characterization of H\"older spaces (see \cite{Campanato}), we obtain that
\begin{equation} \label{HE}
\begin{aligned}
	& [\nabla u]_{C^\gamma(B_{2^{-N-4}R}(x_0))} \\ & = [\nabla u_\ell]_{C^\gamma(B_{2^{-N-4}R}(x_0))} \\
	& \leq C_1 \left (\frac{1}{2^{-N} R} \right )^{1+\gamma} \Bigg ( \mean{B_{2^{-N} R}(x_0)} |u_\ell| dx + (2^{-N} R)^{2s} \int_{\mathbb{R}^n \setminus B_{2^{-N} R}(x_0)} \frac{|u_\ell(y)| }{|x_0-y|^{n+2s}}dy \\
	& \quad + (1+\Gamma) \Bigg ( 2^{-\alpha N} \sum_{k=0}^{N} 2^{k(\alpha -2s)} \mean{B_{2^{k-N}R}(x_0)} |u-c|dx \\ & \quad + (2^{-N}R)^{2s} \int_{\mathbb{R}^n \setminus B_{R}(x_0)} \frac{|u(y)-c|}{|x_0-y|^{n+2s}}dy + (2^{-N}R)^{2s-\frac{n}{p}} ||\mu||_{L^p(B_R(x_0))} \Bigg ) \Bigg ),
\end{aligned}
\end{equation}
where $C$ depends only on $n,s,\Lambda,\alpha,p,\gamma$. Note that for any $\rho \in (0,1/16]$ and all $x,y \in B_{\rho 2^{-N} R}(x_0)$ we have
\begin{align*}
	|\nabla u(x)-\nabla u(y)| \leq 2^\gamma \rho^\gamma (2^{-N}R)^\gamma \frac{|\nabla u(x)-\nabla u(y)|}{|x-y|^\gamma} \leq 2^\gamma \rho^\gamma (2^{-N}R)^\gamma [\nabla u]_{C^\gamma(B_{2^{-N-4}R}(x_0))}.
\end{align*}
Taking the supremum over $x,y \in B_{\rho 2^{-N} R}(x_0)$ now leads to
\begin{align*}
	\osc_{B_{\rho 2^{-N} R}(x_0)} \leq 2^\gamma \rho^\gamma (2^{-N}R)^\gamma [\nabla u]_{C^\gamma(B_{2^{-N-4}R}(x_0))},
\end{align*}
so that combining the previous display with \eqref{HE} yields the desired estimate (\ref{gradoscdecay1}). Therefore, the proof is finished.  
\end{proof}

Theorem \ref{Holdgrad} now follows in view of standard covering and interpolation arguments, which we include for the sake of completeness.

\begin{proof}[Proof of Theorem \ref{Holdgrad}]
	Note that $A$ being $C^\alpha$ with respect to $\Gamma$ in $\Omega$ clearly implies that $A$ is $C^{\widetilde \alpha}$ with respect to $\Gamma \max\{\textnormal{diam}(\Omega),1\}$ in $\Omega$ for any $\widetilde \alpha \in (0,\alpha]$, so that without loss of generality we can assume that $\alpha \in (0,2s-1]$. \par
	First, we prove the desired estimate \eqref{C1gest} in the case when $R=1$, $x_0=0$ and $\Omega=B_1$. Fix some $\rho \in (0,1)$. Moreover, let $R_0=R_0(n,s,\Lambda,\alpha,p,\gamma,\Gamma) \in (0,1)$ be given by Corollary \ref{thm:GradOscDec1}. There exists some small enough radius $r \in (0,R_0)$ such that for any point $z \in \overline B_\rho$ we have $B_{r}(z) \subset B_1$. Since $\left \{B_{r/128}(z) \right \}_{z \in \overline B_\rho}$ is an open covering of $\overline B_{\rho}$ and $\overline B_{\rho}$ is compact, there exists a finite subcover $\left \{B_{r/128}(z_i) \right \}_{i=1}^l$ of $\overline B_{\rho}$ and hence of $B_{\rho}$. 
	Fix $x,y \in B_\rho$ with $x \neq y$. Then $x \in B_{r/128}(z_i)$ for some $i=1,...,l$. 
	If $|x-y| < r/128$, then in particular $y \in B_{r/64}(z_i)$, so that by Corollary \ref{thm:GradOscDec1} (or more precisely, by the estimate \eqref{HE}) with $N=\ell=c=0$ and Lemma \ref{tr} we have
	\begin{equation} \label{locH}
	\begin{aligned}
	&\frac{|\nabla u(x)-\nabla u(y)|}{|x-y|^\gamma} \\ & \leq [\nabla u]_{C^\gamma(B_{r/64}(z_i))} \\ & \leq C_1 \left ( ||u||_{L^1(B_{r/8}(z_i))} + \int_{\mathbb{R}^n \setminus B_{r/8}(z_i)} \frac{|u(y)|}{|z_i-y|^{n+2s}}dy + ||\mu||_{L^p(B_{r/8}(z_i))} \right ) \\
	& \leq C_2 \left (||u||_{L^1(B_1)} + \int_{\mathbb{R}^n \setminus B_{1}} \frac{|u(y)|}{|y|^{n+2s}}dy + ||\mu||_{L^p(B_1)} \right ),
	\end{aligned}
	\end{equation}
	where $C_1$ and $C_2$ depend only on $n,s,\Lambda,\alpha,\Gamma,p,\gamma$ and $r$. 
	Moreover, for any $x \in B_{\rho}$ and $z_i$ as above in view of a standard interpolation inequality (see \cite[Lemma 6.32]{GT}) and Lemma \ref{tr} we have
	\begin{align*}
		|\nabla u(x)| & \leq \sup_{B_{r/64}(z_i)} |\nabla u| \\ 
		& \leq C_3 \sup_{B_{r/64}(z_i)} |u| + [\nabla u]_{C^\gamma(B_{r/64}(z_i))}
		\\ & \leq C_4 \left ( ||u||_{L^1(B_{r/8}(z_i))} + \int_{\mathbb{R}^n \setminus B_{r/8}(z_i)} \frac{|u(y)|}{|z_i-y|^{n+2s}}dy + ||\mu||_{L^p(B_{r/8}(z_i))} \right ) \\
		& \leq C_5 \left (||u||_{L^1(B_1)} + \int_{\mathbb{R}^n \setminus B_{1}} \frac{|u(y)|}{|y|^{n+2s}}dy + ||\mu||_{L^p(B_1)} \right ),
	\end{align*}
	where all constants depend only on $n,s,\Lambda,\alpha,\Gamma,p,\gamma$ and $r$. Therefore, we obtain that
	\begin{align*}
		\sup_{B_{\rho}} |\nabla u|
		\leq C_5 \left ( ||u||_{L^1(B_1)} + \int_{\mathbb{R}^n \setminus B_{1}} \frac{|u(y)|}{|y|^{n+2s}}dy + ||\mu||_{L^p(B_1)} \right ).
	\end{align*}
	
	Now if $|x-y| \geq r/128$, then by the previous display we have
	\begin{equation} \label{nonlocH}
	\begin{aligned}
		\frac{|\nabla u(x)-\nabla u(y)|}{|x-y|^{\gamma}} & \leq 2 \left (\frac{128}{r} \right )^\gamma \sup_{B_\rho} |\nabla u| \\
		& \leq C_6 \left (||u||_{L^1(B_1)} + \int_{\mathbb{R}^n \setminus B_{1}} \frac{|u(y)|}{|y|^{n+2s}}dy + ||\mu||_{L^p(B_1)} \right ),
	\end{aligned}
	\end{equation}
	where $C_6$ depends only on $n,s,\Lambda,\alpha,\Gamma,p,\gamma$ and $r$. \par 
	 Combining \eqref{locH} with \eqref{nonlocH} now leads to the estimate
	 \begin{equation} \label{gradoscdecays1}
	 	\begin{aligned}
	 		[\nabla u]_{C^\gamma(B_{\rho})} \leq C_7 \left ( \mean{B_{1}} |u| dx + \int_{\mathbb{R}^n \setminus B_{1}} \frac{|u(y)| }{|y|^{n+2s}}dy + ||\mu||_{L^p(B_1)} \right ),
	 	\end{aligned}
	 \end{equation}
	 where $C_7$ depends only on $n,s,\Lambda,\alpha,\Gamma,p,\gamma$ and $\rho$, which is because $r$ also depends only on the aforementioned quantities. \par 
	 Next, let us prove the desired estimate \eqref{C1gest} for any $R>0$ and any $x_0 \in \Omega$ such that $B_R(x_0) \Subset \Omega$, which together with Proposition \ref{thm:bnd} and interpolation then in particular yields $u \in C^{1,\gamma}_{loc}(\Omega)$. Consider the scaled functions
	 $$ u_1(x):=u(Rx+x_0), \quad \mu_1(x):=\mu(Rx+x_0)$$
	 and the scaled coefficient
	 $$ A_1(x,y):=A(Rx+x_0,Ry+x_0)$$
	 and observe that $u_1 \in W^{s,2}(B_1) \cap L^1_{2s}(\mathbb{R}^n)$ is a weak solution of $L_{A_1} u_1=\mu_1$ in $B_1$ and that $A_1$ belongs to $\mathcal{L}_0(\Lambda)$ and is $C^\alpha$ in $B_1$ with respect to $\Gamma R^{\alpha}$ and therefore also with respect to $\Gamma \max\{\text{diam}(\Omega),1\}^\alpha$. Therefore, by the first part of the proof along with rescaling we deduce
	 \begin{align*}
	 	R^{1+\gamma}[\nabla u]_{C^\gamma(B_{\rho R}(x_0))} & = [\nabla u_1]_{C^\alpha(B_{\rho})} \\
	 	& \leq C_7 \left ( \mean{B_{1}} |u_1| dx + \int_{\mathbb{R}^n \setminus B_{1}} \frac{|u_1(y)|}{|y|^{n+2s}}dy +  ||\mu_1||_{L^q(B_1)} \right ) \\ & = C \Bigg ( \mean{B_R(x_0)} |u|dx + R^{2s} \int_{\mathbb{R}^n \setminus B_{R}(x_0)} \frac{|u(y)|}{|x_0-y|^{n+2s}}dy \\ & \quad + R^{2s-\frac{n}{p}} ||\mu||_{L^p(B_R(x_0))} \Bigg ),
	 \end{align*}
 	where $C$ that depends only on $n,s,\Lambda,\alpha,\Gamma,p,\gamma$ and $\rho$, so that the proof is finished.
	\end{proof}

\section{Higher differentiability under measure data}

This section is devoted to proving Theorem \ref{HD} and some useful consequences of it.

Our starting point is the following comparison estimate, which follows from \cite[Lemma 3.3]{KMS2} applied with $p=2$.

\begin{proposition} \label{meascomp} 
	Let $s \in (1/2,1)$, $R>0$, $x_0 \in \ern$ and consider some $A \in \mathcal{L}_0(\Lambda)$.
	Moreover, let $\mu \in C_0^\infty(\mathbb{R}^n)$ and assume that $u \in W^{s,2}(B_{2R}(x_0)) \cap L^1_{2s}(\mathbb{R}^n)$ is a weak solution of $L_{A} u = \mu$ in $B_R(x_0)$. Moreover, consider the weak solution $v \in W^{s,2}(B_{2R}(x_0)) \cap L^1_{2s}(\mathbb{R}^n)$ of
	\begin{equation} \label{constcof31xy}
		\begin{cases} \normalfont
			L_{A} v = 0 & \text{ in } B_{R}(x_0) \\
			v = u & \text{ a.e. in } \mathbb{R}^n \setminus B_{R}(x_0).
		\end{cases}
	\end{equation}
	Then for any $q \in \Big [1,\frac{n}{n-2s} \Big )$, we have the comparison estimate
	$$
	\left (\mean{B_{R}(x_0)} |u-v|^qdx \right )^{1/q} \leq C R^{2s-n} |\mu|(B_R(x_0)),
	$$
	where $C$ depends only on $n,s,\Lambda$ and $q$.
\end{proposition}

 We remark that while in \cite{KMS2} $u$ is assumed to belong to $W^{s,2}(\ern)$, an inspection of the proof shows that it is enough to assume that $u \in W^{s,2}(B_{2R}(x_0)) \cap L^1_{2s}(\mathbb{R}^n)$ as stated above, since the existence of a unique weak solution $v \in W^{s,2}(B_{2R}(x_0)) \cap L^1_{2s}(\mathbb{R}^n)$ to \eqref{constcof31xy} is in that case guaranteed by \cite[Remark 3]{existence}.
 
\begin{lemma} \label{thm:higherdiffscalemeasdata} 
	Let $s \in (1/2,1)$ and assume that $A \in \mathcal{L}_0(\Lambda)$ is $C^\alpha$ in $B_2$ with respect to some $\alpha \in (0,1)$ and some $\Gamma>0$.
	Moreover, let $\mu \in C_0^\infty(\mathbb{R}^n)$ and assume that $u \in W^{s,2}(B_1) \cap L_{2s}^1(\mathbb{R}^n)$ is a weak solution to $L_{A} u = \mu$ in $B_1$. Then for any 
	\begin{equation} \label{ranges}
		1 < t<\min\{1+\alpha,2s\}, \quad q \in \bigg [1,\frac{n}{n+t-2s} \bigg ),
	\end{equation}
	we have
	\begin{equation} \label{hdaffinenicem}
		\begin{aligned}
			||u||_{N^{t,q}(B_{\rho})} & \leq C \left ( \left (\mean{B_1} |u|^q dx \right )^{1/q} + \int_{\mathbb{R}^n \setminus B_1} \frac{|u(y)|}{|y|^{n+2s}}dy + |\mu|(B_1) \right ),
		\end{aligned}
	\end{equation}
	for some small radius $\rho=\rho(n,s,\Lambda,\alpha,\Gamma,t,q) \in (0,1)$, while $C$ also depends only on $n,s,\Lambda,\alpha,\Gamma,t,q$.
\end{lemma}

\begin{proof}
	\textbf{Step 1}: Incremental higher differentiability by approximation. \par
	Fix some $1<t<\min\{1+\alpha,2s\}$ and some $q \in \Big [1,\frac{n}{n+t-2s} \Big )$, set $$\gamma:=\frac{t+\min \{1+\alpha,2s+n/q-n\}}{2}-1 \in (0,\min \{\alpha,2s+n/q-n-1\})$$ and let $R_0=R_0(n,s,\Lambda,\alpha,\Gamma,\gamma) \in (0,1)$ be given by Corollary \ref{thm:GradOscDec1} applied with $p=0$. \par 
	Moreover, set $R_1:=R_0/36 \in (0,1/36)$ and fix some $z \in B_{1/2}$ and some $\beta \in (0,1)$ to be chosen. In addition, fix some $|h|>0$ such that $|h|^\beta<R_1$.
	
	Denote by $v_z \in W^{s,2}(B_1) \cap L_{2s}^1(\mathbb{R}^n)$ the weak solution of the problem
	\begin{equation} \label{constcof3}
		\begin{cases} \normalfont
			L_{A} v_z = 0 & \text{ in } B_{18|h|^\beta}(z) \\
			v_z = u & \text{ a.e. in } \mathbb{R}^n \setminus B_{18|h|^\beta}(z).
		\end{cases}
	\end{equation}
	Throughout the proof, unless indicated otherwise all constants $C$ will depend only on $n,s,\Lambda,\alpha,\Gamma,t,q$. \par 
	In view of Proposition \ref{meascomp}, we have
	\begin{equation} \label{mcomp5}
		\left (\int_{B_{18|h|^\beta}(z)} |u-v_z|^q dx \right )^{1/q} \leq C |h|^{n\beta /q-n\beta+ 2s\beta}|\mu|(B_{18|h|^\beta}(z)),
	\end{equation}
	where $C_1=C_1(n,s,\Lambda,q,\beta)>0$. \par
	Set 
	\begin{equation} \label{wtt}
	\widetilde t:=1+\gamma \in \left (t,\min\{1+\alpha,2s+n/q-n\} \right ),
	\end{equation}
	$$
	\widehat \gamma:=\frac{t+3 \min \{1+\alpha,2s+n/q-n\}}{4}-1 \in (\gamma,\min \{\alpha,2s+n/q-n-1\})
	$$
	and also 
	\begin{equation} \label{wtt1}
		\widehat t:=1+ \widehat \gamma \in (\widetilde t,\min \{1+\alpha,2s+n/q-n\}).
	\end{equation}
	Moreover, let $m_0 \geq 1$ be the smallest integer such that $2^{m_0} 18|h|^\beta>1/4$. Note that $A$ being $C^\alpha$ with respect to $\Gamma$ in $B_2$ clearly implies that $A$ is $C^{\widehat \gamma}$ with respect to $2 \Gamma$ in $B_2$. Thus, using (\ref{mcomp5}) and Corollary \ref{thm:GradOscDec1} applied with $c=(u)_{B_{18|h|^\beta}(z)}$, $\rho=|h|^{1-\beta}/16 \in (0,1/16)$, $N=m_0+1$, $R=2^{m_0+1} 16|h|^\beta$ (note that $2^{-N}R <R_0$) and with $\alpha$ replaced by $\gamma$, along with standard properties of difference quotients, Lemma \ref{tailestz} and Lemma \ref{tr}, for any affine function $\ell$ we deduce
	\begin{align*}
		& \left (\int_{B_{|h|^\beta}(z)} |\tau^2_h u|^q dx \right )^{1/q} \\ & \leq \left ( \int_{B_{|h|^\beta}(z)} |\tau^2_h (u- v_{z})|^q dx  \right )^{1/q} + \left ( \int_{B_{|h|^\beta}(z)} |\tau^2_h v_{z}|^q dx \right )^{1/q} \\
		& \leq \left ( \int_{B_{|h|^\beta}(z)} |\tau^2_h (u- v_{z})|^q dx \right )^{1/q} + C |h| \left ( \int_{B_{2|h|^\beta}(z)} |\tau_{h} \nabla v_{z}|^q dx \right )^{1/q} \\
		& \leq C \left ( \int_{B_{18|h|^\beta}(z)} |u- v_{z}|^q dx \right )^{1/q} + C |h|^{n\beta/q+1} \sup_{x^\prime \in B_{2|h|^\beta}(z)} |\tau_{h} \nabla v_{z}(x^\prime)| \\
		& \leq C \left ( \int_{B_{18|h|^\beta}(z)} |u- v_{z}|^q dx \right )^{1/q} + C |h|^{n\beta/q+1} \sup_{x^\prime \in B_{2|h|^\beta}(z)} \osc_{B_{|h|}(x^\prime)} \nabla v_{z} \\
		& \leq C |h|^{n\beta /q-n\beta+ 2s\beta} |\mu|(B_{18|h|^\beta}(z)) \\ & \quad + C |h|^{n \beta /q +\widetilde t (1-\beta)} \sup_{x^\prime \in B_{2|h|^\beta}(z)} \Bigg ( \mean{B_{16|h|^\beta}(x^\prime)} |v_z-\ell| dx + |h|^{2s\beta } \int_{\mathbb{R}^n \setminus B_{16|h|^\beta}(x^\prime)} \frac{|v_z(y)-\ell(y)|}{|x^\prime-y|^{n+2s}}dy \Bigg ) \\
		& \quad + C |h|^{n\beta /q +\widehat \gamma \beta +\widetilde t (1-\beta)} \sup_{x^\prime \in B_{2|h|^\beta}(z)} \Bigg ( \sum_{k=0}^{m_0+1} 2^{(\widehat \gamma-2s)k} \mean{B_{2^{k}16|h|^\beta}(x^\prime)} |v_z-(u)_{B_{18|h|^\beta}(z)}|dx \\ & \quad + |h|^{2s\beta} \sup_{x^\prime \in B_{2|h|^\beta}(z)} \int_{\mathbb{R}^n \setminus B_{2^{m_0+1}16 |h|^\beta}(x^\prime)} \frac{|v_z(y)-(u)_{B_{18|h|^\beta}(z)}|}{|x^\prime-y|^{n+2s}}dy \Bigg ) \\
		& \leq C |h|^{n\beta /q-n\beta+ 2s\beta} |\mu|(B_{18|h|^\beta}(z)) \\ & \quad + C |h|^{n \beta /q +\widetilde t (1-\beta)} \bigg ( \mean{B_{18|h|^\beta}(z)} |v_z-\ell| dx + |h|^{2s\beta } \int_{\mathbb{R}^n \setminus B_{18|h|^\beta}(z)} \frac{|v_z(y)-\ell(y)|}{|z-y|^{n+2s}}dy \bigg ) \\
		& \quad + C |h|^{n\beta /q +\gamma \beta +\widetilde t (1-\beta)} \Bigg ( \sum_{k=0}^{m_0+1} 2^{(\widehat \gamma-2s)k} \mean{B_{2^{k}18|h|^\beta}(z)} |v_z-(u)_{B_{18|h|^\beta}(z)}|dx \\ & \quad + |h|^{2s\beta} \int_{\mathbb{R}^n \setminus B_{2^{m_0+1}18 |h|^\beta}(z)} \frac{|v_z(y)-(u)_{B_{18|h|^\beta}(z)}|}{|z-y|^{n+2s}}dy \Bigg ) \\
		& \leq C |h|^{n\beta /q-n\beta+ 2s\beta} |\mu|(B_{18|h|^\beta}(z)) \\ & \quad + C |h|^{n \beta /q +\widetilde t (1-\beta)} \bigg ( \mean{B_{18|h|^\beta}(z)} |u-\ell| dx + |h|^{2s\beta } \int_{\mathbb{R}^n \setminus B_{18|h|^\beta}(z)} \frac{|u(y)-\ell(y)|}{|z-y|^{n+2s}}dy \bigg ) \\
		& \quad + C |h|^{n\beta /q +\gamma \beta +\widetilde t (1-\beta)} \Bigg ( \sum_{k=0}^{m_0+1} 2^{(\widehat \gamma-2s)k} \mean{B_{18|h|^\beta}(z)} |u-(u)_{B_{2^{k}18|h|^\beta}(z)}|dx \\ & \quad + |h|^{2s\beta} \int_{\mathbb{R}^n \setminus B_{2^{m_0+1}18 |h|^\beta}(z)} \frac{|u(y)-(u)_{B_{18|h|^\beta}(z)}|}{|z-y|^{n+2s}}dy \Bigg ).
	\end{align*}
	
	Using a similar reasoning as in the proof of Lemma \ref{tailest} along with H\"older's inequality and the fractional or classical Poincar\'e inequality, for any $t^\prime \in [0,1]$ we obtain
	\begin{equation} \label{helpest}
		\begin{aligned}
			& \sum_{k=0}^{m_0+1} 2^{(\widehat \gamma-2s)k} \mean{B_{2^{k}18|h|^\beta}(z)} |u-(u)_{B_{18|h|^\beta}(z)}|dx \\ & \quad + |h|^{2s\beta} \int_{\mathbb{R}^n \setminus B_{2^{m_0+1}18 |h|^\beta}(z)} \frac{|u(y)-(u)_{B_{18|h|^\beta}(z)}|}{|z-y|^{n+2s}}dy \\
			& \leq C \sum_{k=0}^{m_0+1} 2^{(\widehat \gamma-2s)k} \mean{B_{2^{k}18|h|^\beta}(z)} |u-(u)_{B_{2^k 18|h|^\beta}(z)}|dx \\ & \quad +C |h|^{2s\beta} \int_{\mathbb{R}^n \setminus B_{2^{m_0+1}18 |h|^\beta}(z)} \frac{|u(y)-(u)_{B_{2^{m_0+1}18|h|^\beta}(z)}|}{|z-y|^{n+2s}}dy \\
			& \leq C \sum_{k=0}^{m_0+1} 2^{(\widehat \gamma-2s)k} \left (\mean{B_{2^{k}18|h|^\beta}(z)} |u-(u)_{B_{2^k 18|h|^\beta}(z)}|^q dx \right )^{1/q} \\ & \quad + C |h|^{2s\beta} \left (\mean{B_1}|u| dx + \int_{\mathbb{R}^n \setminus B_1} \frac{|u(y)|}{|y|^{n+2s}}dy \right ) \\
			& \leq C |h|^{\beta (t^\prime -n/q)} \sum_{k=0}^{m_0+1} 2^{(t^\prime+\widehat \gamma-2s-n/q)k} M_{k,t^\prime }(z) \\ & \quad + C |h|^{2s\beta} \left (\left (\mean{B_1}|u|^q dx \right )^{1/q} + \int_{\mathbb{R}^n \setminus B_1} \frac{|u(y)|}{|y|^{n+2s}}dy \right ),
		\end{aligned}
	\end{equation}
	where 
	$$ M_{k,t^\prime }(z):= \begin{cases} \normalfont
		\left (\int_{B_{2^{k}18|h|^\beta}(z)} |u|^qdx \right )^{1/q} & \text{ if } t^\prime =0 \\
		\left ( \int_{B_{2^{k}18|h|^\beta}(z)} \int_{B_{2^{k}18|h|^\beta}(z)} \frac{|u(x)-u(y)|^q}{|x-y|^{n+t^\prime q }}dydx \right )^{1/q} & \text{ if } t^\prime \in (0,1) \\
		\left ( \int_{B_{2^{k}18|h|^\beta}(z)} |\nabla u|^q dx \right )^{1/q} & \text{ if } t^\prime =1.
	\end{cases} $$
	
	Combining the previous two displays now yields 
	\begin{equation} \label{hest}
		\begin{aligned}
			& \left (\int_{B_{|h|^\beta}(z)} |\tau^2_h u|^q dx \right )^{1/q} \\
			& \leq C |h|^{n\beta /q-n\beta+ 2s\beta} |\mu|(B_{18|h|^\beta}(z)) \\ & \quad + C |h|^{n \beta /q +\widetilde t (1-\beta)} \bigg ( \mean{B_{18|h|^\beta}(z)} |u-\ell| dx + |h|^{2\beta s} \int_{\mathbb{R}^n \setminus B_{18|h|^\beta}(z)} \frac{|u(y)-\ell(y)|}{|z-y|^{n+2s}}dy \bigg ) \\
			& \quad + C |h|^{t^\prime \beta+\gamma \beta +\widetilde t (1-\beta)} \sum_{k=0}^{m_0+1} 2^{(t^\prime+\widehat \gamma-2s-n/q)k} M_{k,t^\prime }(z) \\ & \quad + C |h|^{(n/q+2s)\beta+\widetilde t (1-\beta)} \left (\left (\mean{B_1}|u|^q dx \right )^{1/q} + \int_{\mathbb{R}^n \setminus B_1} \frac{|u(y)|}{|y|^{n+2s}}dy \right )
		\end{aligned}
	\end{equation}
	for any $t^\prime \in [0,1]$.
	
	\textbf{Step 2}: $N^{t_\star,q}$ regularity for some $t_\star>1$ by iteration. \par 
	In order to proceed, observe that in view of Lemma \ref{tailest} and H\"older's inequality, we have
	
	\begin{align*}
		& \mean{B_{18|h|^\beta}(z)} |u-(u)_{B_{18|h|^\beta}(z)}| dx + |h|^{2\beta s} \int_{\mathbb{R}^n \setminus B_{18|h|^\beta}(z)} \frac{|u(y)-(u)_{B_{18|h|^\beta}(z)}|}{|z-y|^{n+2s}}dy \\
		& \leq C \sum_{k=0}^{m_0+1} 2^{-2sk} \mean{B_{2^{k}18|h|^\beta}(z)} |u-(u)_{B_{2^k18|h|^\beta}(z)}|dx \\ & \quad + C |h|^{2s\beta} \int_{\mathbb{R}^n \setminus B_{2^{m_0+1}18 |h|^\beta}(z)} \frac{|u(y)-(u)_{B_{2^{m_0+1}18|h|^\beta}(z)}|}{|z-y|^{n+2s}}dy \\
		& \leq C \sum_{k=0}^{m_0+1} 2^{(\widehat \gamma-2s)k} \left ( \mean{B_{2^{k}18|h|^\beta}(z)} |u-(u)_{B_{2^k18|h|^\beta}(z)}|dx \right )^{1/q} \\ & \quad + C |h|^{2s\beta} \int_{\mathbb{R}^n \setminus B_{2^{m_0+1}18 |h|^\beta}(z)} \frac{|u(y)-(u)_{B_{2^{m_0+1}18|h|^\beta}(z)}|}{|z-y|^{n+2s}}dy,
	\end{align*}
	so that by the same reasoning as in \eqref{helpest} we obtain
	\begin{align*}
		& \mean{B_{18|h|^\beta}(z)} |u-(u)_{B_{18|h|^\beta}(z)}| dx + |h|^{2s\beta } \int_{\mathbb{R}^n \setminus B_{18|h|^\beta}(z)} \frac{|u(y)-(u)_{B_{18|h|^\beta}(z)}|}{|z-y|^{n+2s}}dy \\
		& \leq C |h|^{\beta (t^\prime -n/q)} \sum_{k=0}^{m_0+1} 2^{(t^\prime+\widehat \gamma-2s-n/q)k} M_{k,t^\prime }(z) \\ & \quad + C |h|^{2s\beta} \left (\left (\mean{B_1}|u|^q dx \right )^{1/q} + \int_{\mathbb{R}^n \setminus B_1} \frac{|u(y)|}{|y|^{n+2s}}dy \right ).
	\end{align*}
	Combining the previous display with \eqref{hest} applied with $\ell=(u)_{B_{18|h|^\beta}(z)}$ now yields
	
	\begin{equation} \label{hest1}
		\begin{aligned}
			& \left (\int_{B_{|h|^\beta}(z)} |\tau^2_h u|^q dx \right )^{1/q} \\
			& \leq C |h|^{n\beta /q-n\beta+ 2s\beta} |\mu|(B_{18|h|^\beta}(z))
			+ C|h|^{t^\prime \beta +\widetilde t (1-\beta)} \sum_{k=0}^{m_0+1} 2^{(t^\prime+\widehat \gamma-2s-n/q)k} M_{k,t^\prime }(z) \\ & \quad + C|h|^{(n/q+2s)\beta+\widetilde t (1-\beta)} \left (\left (\mean{B_1}|u|^q dx \right )^{1/q} + \int_{\mathbb{R}^n \setminus B_1} \frac{|u(y)|}{|y|^{n+2s}}dy \right )
		\end{aligned}
	\end{equation}
	for any $t^\prime \in [0,1]$. \par 
	Now observe that by
	some combinatorial considerations, there exists some $N \in \mathbb{N}$ with \newline $N \leq c|h|^{-\beta n}$ for some $c=c(n)$ and a finite sequence of points $\{z_j\}_{j=1}^N$ such that 
	\begin{equation} \label{Bcover}
		B_{1/2} \subset \bigcup_{j=1}^N B_{|h|^\beta}(z_j),
	\end{equation}
	and such that for any integer $k \geq 0$ we have
	\begin{equation} \label{overlap}
		\sum_{j=1}^N \chi_{B_{2^k|h|^\beta}(z_j)} \leq c^\prime 2^{nk},
	\end{equation}
	where $c^\prime=c^\prime(n)$. Summing the estimates deduced in \eqref{hest1} over $j=1,...,N$ now yields
	\begin{equation} \label{mt0}
		\begin{aligned}
			& \left ( \int_{B_{1/2}} |\tau^2_h u|^q dx \right )^{1/q} \\
			& \leq C \sum_{j=1}^N \left ( \int_{B_{|h|^\beta}(z_j)} |\tau^2_h u|^q dx \right )^{1/q} \\
			& \leq C \sum_{j=1}^N \Bigg (|h|^{n\beta /q-n\beta+ 2s\beta} |\mu|(B_{18|h|^\beta}(z_j)) \\
			& \quad + |h|^{\beta t^\prime +\widetilde t(1-\beta)} \sum_{k=0}^{m_0+1} 2^{(t^\prime+\widehat \gamma-2s-n/q)k} M_{k,t^\prime }(z_j) \\& \quad + |h|^{\beta (n/q+2s)+\widetilde t(1-\beta)} \left ( \left (\mean{B_{1}} |u|^q dx \right)^{1/q} + \int_{\mathbb{R}^n \setminus B_{1}} \frac{|u(y)|}{|y|^{n+2s}}dy \right ) \Bigg ) \\
			& \leq C \Bigg (|h|^{n\beta /q-n\beta+ 2s\beta} |\mu|(B_1) + |h|^{\beta t^\prime +\widetilde t(1-\beta)} \underbrace{\sum_{k=1}^{\infty} 2^{(\widehat t-2s-n/q+n)k}}_{< \infty} M_{t^\prime } \\& \quad + |h|^{n\beta /q-n\beta+2s\beta+\widetilde t(1-\beta)} \left ( \left ( \mean{B_{1}} |u|^q dx \right )^{1/q} + \int_{\mathbb{R}^n \setminus B_{1}} \frac{|u(y)|}{|y|^{n+2s}}dy \right ) \Bigg ) \\
			& \leq C \Bigg ( |h|^{\beta t^\prime +\widetilde t(1-\beta)} M_{t^\prime } \\ & \quad + |h|^{n\beta /q-n\beta+ 2s\beta} \left (|\mu|(B_1) + \left (\mean{B_{1}} |u|^q dx\right )^{1/q} + \int_{\mathbb{R}^n \setminus B_{1}} \frac{|u(y)|}{|y|^{n+2s}}dy \right ) \Bigg ) 
		\end{aligned}
	\end{equation}
	for any $t^\prime \in [0,1)$, where $$ M_{t^\prime }:= \begin{cases} \normalfont
		\left (\int_{B_1} |u|^qdx \right )^{1/q} & \text{ if } t^\prime =0 \\
		\left (\int_{B_1} \int_{B_1} \frac{|u(x)-u(y)|^q}{|x-y|^{n+t^\prime q }}dydx \right )^{1/q} & \text{ if } t^\prime >0
	\end{cases} $$
	and we also used \eqref{wtt1} in order to ensure the convergence of the infinite series. \par
	In the case when $t^\prime \in (0,1)$, for any $i \in \mathbb{N}$ we consider
	the scaled functions $u_i(x):=u(2^{-i}x)$, $\mu_i(x):=\mu(2^{-i}x)$, $A_i(x,y):=A(2^{-i}x,2^{-i}y)$ and note that $u_i \in W^{s,2}(B_{2^i}) \cap L_{2s}^1(\ern) \subset W^{s,2}(B_{1}) \cap L_{2s}^1(\ern)$ is a weak solution of $L_{A_i} u_i=\mu_i$ in $B_{2^i} \supset B_1$. In addition, $\mu_i$ belongs to $C_0^\infty(\ern)$ and $A_i$ is $C^\alpha$ in $B_{2^{i+1}} \supset B_2$. Therefore, the estimate \eqref{mt0} also holds with $u$ replaced by $u_i$ and with $\mu$ replaced with $\mu_i$, so that for any $t^\prime \in (0,1)$ and any $i \in \mathbb{N}$ we have
	\begin{equation} \label{mt0scaled}
		\begin{aligned}
			& \left ( \int_{B_{2^{-i-1}}} |\tau^2_h u|^q dx \right )^{1/q} \\
			& = C \left ( \int_{B_{1/2}} |\tau^2_h u_i|^q dx \right )^{1/q} \\
			& \leq C \Bigg (|h|^{\beta t^\prime +\widetilde t(1-\beta)} \left (\int_{B_1} \int_{B_1} \frac{|u_i(x)-u_i(y)|^q}{|x-y|^{n+t^\prime q }}dydx \right )^{1/q} \\& \quad + |h|^{n\beta /q-n\beta+ 2s\beta} \left (|\mu_i|(B_1)+ \left (\mean{B_{1}} |u_i|^q dx \right)^{1/q} + \int_{\mathbb{R}^n \setminus B_{1}} \frac{|u_i(y)|}{|y|^{n+2s}}dy \right ) \Bigg ) \\
			& = C \Bigg (|h|^{\beta t^\prime +\widetilde t(1-\beta)} \left (\int_{B_{2^{-i}}} \int_{B_{2^{-i}}} \frac{|u(x)-u(y)|^q}{|x-y|^{n+t^\prime q }}dydx \right )^{1/q} \\& \quad + |h|^{n\beta /q-n\beta+ 2s\beta}  \left (|\mu|(B_{2^{-i}})+ \left ( \mean{B_{2^{-i}}} |u|^q dx \right)^{1/q} + \int_{\mathbb{R}^n \setminus B_{2^{-i}}} \frac{|u(y)|}{|y|^{n+2s}}dy \right ) \Bigg ) \\
			& \leq C \Bigg (|h|^{\beta t^\prime +\widetilde t(1-\beta)} \left (\int_{B_{2^{-i}}} \int_{B_{2^{-i}}} \frac{|u(x)-u(y)|^q}{|x-y|^{n+t^\prime q }}dydx \right )^{1/q} \\& \quad + |h|^{n\beta /q-n\beta+ 2s\beta} \left (|\mu|(B_1)+ \left ( \mean{B_{1}} |u|^qdx \right)^{1/q} + \int_{\mathbb{R}^n \setminus B_{1}} \frac{|u(y)|}{|y|^{n+2s}}dy \right ) \Bigg ),
		\end{aligned}
	\end{equation}
	where $C$ additionally depends on $i$.
	
	Next, choose $$\beta:=\frac{t/2}{n/q-n+2s}+\frac{1}{2}$$ and note that by the assumption \eqref{ranges} we have $\beta \in (0,1)$ and 
	\begin{equation} \label{gammabeta}
		n\beta /q-n\beta+ 2s\beta>t.
	\end{equation}
	Now (\ref{mt0}) applied with $t^\prime =0$ along with the observation that $n\beta /q-n\beta+ 2s\beta > \widetilde t(1-\beta)$ yields 
	\begin{align*}
		& \left ( \int_{B_{1/2}} |\tau^2_h u|^q dx \right )^{1/q} \\
		& \leq C |h|^{\widetilde t(1-\beta)} \left (|\mu|(B_1) + \left ( \mean{B_{1}} |u|^qdx \right)^{1/q} + \int_{\mathbb{R}^n \setminus B_{1}} \frac{|u(y)|}{|y|^{n+2s}}dy \right )
	\end{align*}
	for any $0<|h|<R_1^{1/\beta}$ and therefore taking into account Remark \ref{Niksmall},
	\begin{equation} \label{t0}
			||u||_{N^{\widetilde t(1-\beta),q}(B_{1/2})} 
			\leq C \left (|\mu|(B_1) + \left ( \mean{B_{1}} |u|^qdx \right)^{1/q} + \int_{\mathbb{R}^n \setminus B_{1}} \frac{|u(y)|}{|y|^{n+2s}}dy \right ) .
	\end{equation}
	
	Next, define recursively the sequence $\{t_i\}_{i=0}^\infty$ by 
	\begin{equation} \label{tidef}
		\begin{aligned}
			&t_0 :=\widetilde t(1-\beta) \in (0,1), \\ & t_{i}:=\beta \left (t_{i-1}-\frac{(1-\beta)(\widetilde t-t_{i-1})}{2 \beta} \right)+\widetilde t(1-\beta)=t_{i-1}+(\widetilde t -t_{i-1})(1-\beta)/2, \text{ } i \in \mathbb{N}.
		\end{aligned}
	\end{equation} 
	Observe that the sequence $\{t_i\}_{i =0}^\infty$ is strictly monotone increasing with 
	\begin{equation} \label{limti}
		\lim_{i \to \infty} t_i=\widetilde t
	\end{equation}
	and that for any $i \in \mathbb{N}$ we have $$ t_i \in (t_{i-1},t_{i-1}+\widetilde t(1-\beta)).$$ In addition, note that we can assume that $t_i \neq 1$ for any $i$, since otherwise we can slightly reduce $\beta$ such that (\ref{gammabeta}) still remains valid. Now define $$i_\star := \max \{i \geq 0 \mid t_i<1\}.$$ In particular, since the sequence $\{t_i\}_{i=0}^\infty$ is strictly increasing and $t_i \neq 1$ for all i, we clearly have $$t_i <1 \text{ for any } i \leq i_\star \text{ and } t_{i_{\star}+1}>1.$$
	Set 
	\begin{equation} \label{tstar}
		t_\star:=\min\{t,t_{i_{\star}+1}\}>1.
	\end{equation}
	Thus, for any $i \in \{1,...,i_\star+1\}$ in view of (\ref{mt0scaled}) applied with $i$ and $t^\prime=t_{i-1} \in (0,1)$, along with \eqref{gammabeta} and Proposition \ref{WNrel} for any $0<|h|<R_1^{1/\beta}$ we obtain 
	\begin{align*}
		& \left ( \int_{B_{2^{-i-1}}} |\tau^2_h u|^q dx \right )^{1/q} \\
		& \leq C \Bigg (|h|^{\beta t_{i-1} +\widetilde t(1-\beta)} \left (\int_{B_{2^{-i}}} \int_{B_{2^{-i}}} \frac{|u(x)-u(y)|^q}{|x-y|^{n+t^\prime q }}dydx \right )^{1/q} \\& \quad + |h|^{t} \left (|\mu|(B_1)+ \left ( \mean{B_{1}} |u|^qdx \right)^{1/q} + \int_{\mathbb{R}^n \setminus B_{1}} \frac{|u(y)|}{|y|^{n+2s}}dy \right ) \Bigg ) \\
		& \leq C |h|^{\min\{t,t_i\}} \left ( ||u||_{N^{t_{i-1},q}(B_{2^{-i}})} + |\mu|(B_{1})+ \left ( \mean{B_{1}} |u|^qdx \right)^{1/q} + \int_{\mathbb{R}^n \setminus B_{1}} \frac{|u(y)|}{|y|^{n+2s}}dy \right ) .
	\end{align*}
	Therefore, taking into account Remark \ref{Niksmall}, for $i \in \{1,...,i_\star+1\}$ we obtain the iterative scheme of estimates
	
	\begin{equation} \label{t1}
		\begin{aligned}
			& ||u||_{N^{\min \{t,t_i\},q}(B_{2^{-i-1}})} \\
			& \leq C \left (||u||_{N^{t_{i-1},q}(B_{2^{-i}})} + |\mu|(B_1) + \left ( \mean{B_{1}} |u|^qdx \right)^{1/q} + \int_{\mathbb{R}^n \setminus B_{1}} \frac{|u(y)|}{|y|^{n+2s}}dy \right ) .
		\end{aligned}
	\end{equation}
	Iterating these estimates and finally applying \eqref{t0} yields
	\begin{equation} \label{hdaffinenicem2}
		\begin{aligned}
			& ||u||_{N^{t_\star,q}(B_{\rho_\star})} \\ & \leq C \bigg ( ||u||_{N^{t_{0},q}(B_{1/2})} + |\mu|(B_1) + \left ( \mean{B_{1}} |u|^qdx \right)^{1/q} + \int_{\mathbb{R}^n \setminus B_1} \frac{|u(y)|}{|y|^{n+2s}}dy \bigg ) \\
			& \leq C \bigg (\left ( \mean{B_{1}} |u|^qdx \right)^{1/q} + \int_{\mathbb{R}^n \setminus B_1} \frac{|u(y)|}{|y|^{n+2s}}dy + |\mu|(B_1) \bigg ),
		\end{aligned}
	\end{equation}
	where $\rho_\star:=\left (\frac{1}{2} \right)^{i_\star+2}$.
	
	\textbf{Step 3}: Improved differentiability. \par
	Since in general $t_\star<t$, let us improve the differentiability gain up to order $t$ in this case. \par 
	As in step 1, fix $z \in B_{1/2}$ and some $|h|>0$ such that $|h|^\beta<R_1$. Moreover, let $m_0 \geq 1$ again be the smallest integer such that $2^{m_0} 18|h|^\beta>1/4$. By Lemma \ref{tailestaffine} and H\"older's inequality, for the affine function
	\begin{equation} \label{affinechoice}
		\ell(x):=(u)_{B_{18|h|^\beta}(z)} + (\nabla u)_{B_{18|h|^\beta}(z)} \cdot (x-z)
	\end{equation}
	and any $\gamma^\prime \in (0,\gamma)$,
	we have
	\begin{equation} \label{annuli2x}
		\begin{aligned}
			& \mean{B_{18|h|^\beta}(z)} |u-\ell| dx + |h|^{2s\beta} \int_{\mathbb{R}^n \setminus B_{18|h|^\beta}(z)} \frac{|u-\ell(y)| }{|z-y|^{n+2s}}dy \\
			& \leq C |h|^\beta \Bigg ( \sum_{k=0}^{m_0+1} 2^{(1-2s)k} \mean{B_{2^k 18|h|^\beta}(z)} |\nabla u-(\nabla u)_{B_{2^k 18|h|^\beta}(z)}|dx \\ & \quad + |h|^{(2s-1)\beta} \sum_{k=0}^{m_0+1} \mean{B_{2^k 18|h|^\beta}(z)} |\nabla u|dx \\ & \quad + |h|^{(2s-1)\beta} \int_{\mathbb{R}^n \setminus B_{2^{m_0+1}18|h|^\beta}(z)} \frac{|u(y)-(u)_{B_{2^{m_0+1}18|h|^\beta}(z)}|} {|z-y|^{n+2s}} dy \Bigg ) \\
			& \leq C |h|^\beta \sum_{k=0}^{m_0+1} 2^{(1-2s)k} \left (\mean{B_{2^k 18|h|^\beta}(z)} |\nabla u-(\nabla u)_{B_{2^k 18|h|^\beta}(z)}|^q dx \right )^{1/q} \\ & \quad + C |h|^{\widetilde t \beta} \sum_{k=0}^{m_0+1} 2^{(\widehat t-2s)k} \left ( \mean{B_{2^k 18|h|^\beta}(z)} |\nabla u|^q dx \right )^{1/q} \\ & \quad + C |h|^{2s\beta} \left (\left (\mean{B_1}|u|^q dx \right )^{1/q} + \int_{\mathbb{R}^n \setminus B_1} \frac{|u(y)|} {|y|^{n+2s}} dy \right ),
		\end{aligned}
	\end{equation}
	where we also used that for any $k \in \{0,...,m_0+1\}$, $$|h|^{2s\beta} = |h|^{\widetilde t \beta} (2^{(m_0+1)}|h|)^{(2s-\widetilde t) \beta} 2^{(\widetilde t-2s)(m_0+1)} \leq |h|^{\widetilde t \beta} 2^{(\widehat t-2s)k} .$$
	
	In view of \eqref{hest} with respect to the affine function $\ell$ defined in \eqref{affinechoice} and $t^\prime=1$, along with \eqref{annuli2x} and the fractional Poincar\'e inequality we obtain
	\begin{equation} \label{hest2}
		\begin{aligned}
			& \left (\int_{B_{|h|^\beta}(z)} |\tau^2_h u|^q dx \right )^{1/q} \\
			& \leq C |h|^{n\beta /q-n\beta+ 2s\beta} |\mu|(B_{18|h|^\beta}(z)) \\ & \quad + C |h|^{n\beta /q+\widetilde t(1-\beta)+\beta} \sum_{k=0}^{m_0+1} 2^{(1-2s)k} \left (\mean{B_{2^k 18|h|^\beta}(z)} |\nabla u-(\nabla u)_{B_{2^k 18|h|^\beta}(z)}|^q dx \right )^{1/q} \\ & \quad + C |h|^{n\beta /q+\widetilde t \beta+\widetilde t(1-\beta)} \sum_{k=0}^{m_0+1} 2^{(\widehat t-2s)k} \left ( \mean{B_{2^k 18|h|^\beta}(z)} |\nabla u|^q dx \right )^{1/q} \\ & \quad + C |h|^{(n/q+2s)\beta+\widetilde t(1-\beta)} \left (\left (\mean{B_1}|u|^q dx \right )^{1/q} + \int_{\mathbb{R}^n \setminus B_1} \frac{|u(y)|} {|y|^{n+2s}} dy \right ) \\
			& \leq C |h|^{n\beta /q-n\beta+ 2s\beta} |\mu|(B_{18|h|^\beta}(z)) \\ & \quad + C |h|^{(1+\gamma^\prime)\beta+\widetilde t(1-\beta)} \sum_{k=0}^{m_0+1} 2^{(1+\gamma^\prime-2s-n/q)k} \left ( \int_{B_{2^{k}18|h|^\beta}(z)} \int_{B_{2^{k}18|h|^\beta}(z)} \frac{|\nabla u(x)-\nabla u(y)|^q}{|x-y|^{n+\gamma^\prime q }}dydx \right )^{1/q} \\ & \quad + C |h|^{\widetilde t} \sum_{k=0}^{m_0+1} 2^{(\widehat t-2s-n/q)k} \left ( \int_{B_{2^k 18|h|^\beta}(z)} |\nabla u|^q dx \right )^{1/q} \\ & \quad + C |h|^{(n/q+2s)\beta} \left (\left (\mean{B_1}|u|^q dx \right )^{1/q} + \int_{\mathbb{R}^n \setminus B_1} \frac{|u(y)|} {|y|^{n+2s}} dy \right ).
		\end{aligned}
	\end{equation}
	
	As in step 2, there exists some $N \in \mathbb{N}$ with $N \leq c|h|^{-\beta n}$ for some $c=c(n)$ and a finite sequence of points $\{z_j\}_{j=1}^N$ such that \eqref{Bcover} and \eqref{overlap} are satisfied. Summing the estimates deduced in \eqref{hest2} over $j=1,...,N$ now yields
		\begin{align*}
			& \left ( \int_{B_{1/2}} |\tau^2_h u|^q dx \right )^{1/q} \\
			& \leq C \sum_{j=1}^N \left ( \int_{B_{|h|^\beta}(z_j)} |\tau^2_h u|^q dx \right )^{1/q} \\
			& \leq C \sum_{j=1}^N \Bigg (|h|^{n\beta /q-n\beta+ 2s\beta} |\mu|(B_{18|h|^\beta}(z)) \\ & \quad + |h|^{(1+\gamma^\prime)\beta+\widetilde t(1-\beta)} \sum_{k=0}^{m_0+1} 2^{(1+\gamma^\prime-2s-n/q)k} \left ( \int_{B_{2^{k}18|h|^\beta}(z)} \int_{B_{2^{k}18|h|^\beta}(z)} \frac{|\nabla u(x)-\nabla u(y)|^q}{|x-y|^{n+\gamma^\prime q }}dydx \right )^{1/q} \\ & \quad + |h|^{\widetilde t} \sum_{k=0}^{m_0+1} 2^{(\widehat t-2s-n/q)k} \left ( \int_{B_{2^k 18|h|^\beta}(z)} |\nabla u|^q dx \right )^{1/q} \\ & \quad + |h|^{(n/q+2s)\beta} \left (\left (\mean{B_1}|u|^q dx \right )^{1/q} + \int_{\mathbb{R}^n \setminus B_1} \frac{|u(y)|} {|y|^{n+2s}} dy \right ) \Bigg ) \\
			& \leq C \Bigg (|h|^{n\beta /q-n\beta+ 2s\beta} |\mu|(B_1) + |h|^{\beta t^\prime +\widetilde t(1-\beta)} \underbrace{\sum_{k=1}^{\infty} 2^{(\widetilde t-2s-n/q+n)k}}_{< \infty} \left ( \int_{B_{1}} \int_{B_1} \frac{|\nabla u(x)-\nabla u(y)|^q}{|x-y|^{n+\gamma^\prime q }}dydx \right )^{1/q} \\
			& \quad + |h|^{\widetilde t} \underbrace{\sum_{k=1}^{\infty} 2^{(\widehat t-2s-n/q+n)k}}_{< \infty} \left ( \int_{B_{1}} |\nabla u|^q dx \right )^{1/q} \\
			& \quad + |h|^{n\beta /q-n\beta+2s\beta} \left ( \left ( \mean{B_{1}} |u|^q dx \right )^{1/q} + \int_{\mathbb{R}^n \setminus B_{1}} \frac{|u(y)|}{|y|^{n+2s}}dy \right ) \Bigg ) \\
			& \leq C \Bigg ( |h|^{(1+\gamma^\prime)\beta +\widetilde t(1-\beta)} \left ( \int_{B_{1}} \int_{B_1} \frac{|\nabla u(x)-\nabla u(y)|^q}{|x-y|^{n+\gamma^\prime q }}dydx \right )^{1/q} + |h|^{\widetilde t} \left ( \int_{B_{1}} |\nabla u|^q dx \right )^{1/q} \\
			& \quad + |h|^{n\beta /q-n\beta+ 2s\beta} \left (|\mu|(B_1) + \left (\mean{B_{1}} |u|^q dx\right )^{1/q} + \int_{\mathbb{R}^n \setminus B_{1}} \frac{|u(y)|}{|y|^{n+2s}}dy \right ) \Bigg ) 
		\end{align*}
	for any $\gamma^\prime \in (0,\gamma)$, where we used \eqref{wtt} and \eqref{wtt1} in order to ensure the convergence of the infinite sums. \par
	In view of a scaling argument similar to the one we carried out in step 2, for any $\gamma^\prime \in (0,\gamma)$ and any $i \in \mathbb{N}$ we have
	\begin{equation} \label{mt01scaled}
		\begin{aligned}
			& \left ( \int_{B_{2^{-i-1}}} |\tau^2_h u|^q dx \right )^{1/q} \\
			& \leq C \Bigg ( |h|^{(1+\gamma^\prime)\beta  +\widetilde t(1-\beta)} \left ( \int_{B_{2^{-i}}} \int_{B_{2^{-i}}} \frac{|\nabla u(x)-\nabla u(y)|^q}{|x-y|^{n+\gamma^\prime q }}dydx \right )^{1/q} \\
			& \quad + |h|^{t} \left (|\mu|(B_1) + \left ( \int_{B_{2^{-i}}} |\nabla u|^q dx \right )^{1/q} + \left ( \mean{B_{1}} |u|^q dx\right )^{1/q} + \int_{\mathbb{R}^n \setminus B_{1}} \frac{|u(y)|}{|y|^{n+2s}}dy \right ) \Bigg )  ,
		\end{aligned}
	\end{equation}
	where $C$ additionally depends on $i$ and we also used that in view of \eqref{gammabeta}, we have $$t<\min \{\widetilde t,n\beta /q-n\beta+ 2s\beta\}.$$ \par 
	We now revisit the sequence $\{t_i\}_{i=0}^\infty$ defined in \eqref{tidef}. Since in this step we assume that $t_\star<t$, taking also into consideration \eqref{limti}, we obtain that there clearly exists some integer $\widetilde i_\star >i_\star+1$ such that $t_{\widetilde i_\star} \geq t$. \par 
	Hence, for any $i \in \{i_\star+2,...,t_{\widetilde i_\star}\}$ in view of (\ref{mt01scaled}) applied with $i$ and $\gamma^\prime=t_{i-1}-1 \in (0,\gamma)$, along with Proposition \ref{WNrel} for any $0<|h|<R_1^{1/\beta}$ we obtain 
	\begin{align*}
		& \left ( \int_{B_{2^{-i-1}}} |\tau^2_h u|^q dx \right )^{1/q} \\
		& \leq C \Bigg ( |h|^{t_{i-1}\beta  +\widetilde t(1-\beta)} \left ( \int_{B_{2^{-i}}} \int_{B_{2^{-i}}} \frac{|\nabla u(x)-\nabla u(y)|^q}{|x-y|^{n+\gamma^\prime q }}dydx \right )^{1/q} \\
		& \quad + |h|^{t} \left (|\mu|(B_1) + \left ( \int_{B_{2^{-i}}} |\nabla u|^q dx \right )^{1/q} + \left ( \mean{B_{1}} |u|^q dx\right )^{1/q} + \int_{\mathbb{R}^n \setminus B_{1}} \frac{|u(y)|}{|y|^{n+2s}}dy \right ) \Bigg ) \\
		& \leq C |h|^{\min\{t,t_i\}} \left ( ||u||_{N^{t_{i-1},q}(B_{2^{-i}})} + |\mu|(B_{1})+ \left ( \mean{B_{1}} |u|^qdx \right)^{1/q} + \int_{\mathbb{R}^n \setminus B_{1}} \frac{|u(y)|}{|y|^{n+2s}}dy \right ) .
	\end{align*}
	Therefore, again taking into account Remark \ref{Niksmall}, for $i \in \{i_\star+2,...,t_{\widetilde i_\star}\}$ we obtain the iterative scheme of estimates
	
	\begin{equation} \label{t11}
		\begin{aligned}
			& ||u||_{N^{\min \{t,t_i\},q}(B_{2^{-i-1}})} \\
			& \leq C \left (||u||_{N^{t_{i-1},q}(B_{2^{-i}})} + |\mu|(B_1) + \left ( \mean{B_{1}} |u|^qdx \right)^{1/q} + \int_{\mathbb{R}^n \setminus B_{1}} \frac{|u(y)|}{|y|^{n+2s}}dy \right ) .
		\end{aligned}
	\end{equation}
	Iterating these estimates and finally applying \eqref{hdaffinenicem2} yields
	\begin{equation} \label{hdaffinenicem21}
		\begin{aligned}
			& ||u||_{N^{t,q}(B_{\rho})} \\ & \leq C \bigg ( ||u||_{N^{t_{\star},q}(B_{\rho_\star})} + |\mu|(B_1) + \left ( \mean{B_{1}} |u|^qdx \right)^{1/q} + \int_{\mathbb{R}^n \setminus B_1} \frac{|u(y)|}{|y|^{n+2s}}dy \bigg ) \\
			& \leq C \bigg (\left ( \mean{B_{1}} |u|^qdx \right)^{1/q} + \int_{\mathbb{R}^n \setminus B_1} \frac{|u(y)|}{|y|^{n+2s}}dy + |\mu|(B_1) \bigg ),
		\end{aligned}
	\end{equation}
	where $\rho:=\left (\frac{1}{2} \right)^{\widetilde i_\star+1}$.
	In particular, since $\widetilde i_\star$ only depends on $n,s,\Lambda,\alpha,\Gamma,t$ and $q$, the same is true for $\rho$.
\end{proof}

By scaling and covering arguments, we arrive at the following version of Theorem \ref{HD} for weak solutions instead of SOLA.
\begin{corollary} \label{cor:higherdiff} 
	Let $s \in (1/2,1)$, $\mu \in C_0^\infty(\mathbb{R}^n)$ and let $u \in W^{s,2}_{loc}(\Omega) \cap L^1_{2s}(\ern)$ be a weak solution of $L_A u=\mu$ in a bounded domain $\Omega$. In addition, assume that $A \in \mathcal{L}_0(\Lambda)$ is $C^\alpha$ in $\Omega$ with respect to some $\Gamma>0$ and some $\alpha>0$. Then for all 
	\begin{equation} \label{range1}
	0 < \gamma<\min \{\alpha,2s-1\}, \quad q \in \bigg [1,\frac{n}{n+1+\gamma-2s} \bigg ),
	\end{equation}
	we have $u \in W^{1+\gamma,q}_{loc}(\Omega)$. Moreover, for any $R>0$ and any $x_0 \in \Omega$ such that $B_R(x_0) \subset \Omega$, we have
	\begin{equation} \label{HDE1}
		\begin{aligned}
			& \left (\mean{B_{R/2}(x_0)} |\nabla u|^q dx \right )^{1/q} + R^{\gamma }\left (\mean{B_{R/2}(x_0)} \int_{B_{R/2}(x_0)} \frac{|\nabla u(x)-\nabla u(y)|^q}{|x-y|^{n+\gamma q}} dydx \right )^{1/q} \\ & \leq \frac{C}{R} \left (\left (\mean{B_{R}(x_0)} |u|^q dx \right )^{1/q} + R^{2s}\int_{\mathbb{R}^n \setminus B_R(x_0)} \frac{|u(y)|}{|x_0-y|^{n+2s}}dy + R^{2s-n} |\mu|(B_{3R/4}(x_0)) \right ),
		\end{aligned}
	\end{equation}
	where $C$ depends only on $n,s,\Lambda,\alpha,\gamma,q,\Gamma$ and $\max\{\textnormal{diam}(\Omega),1\}$.
\end{corollary}

\begin{proof}
	First, we prove the result in the case when $R=1$, $x_0=0$ and $\Omega=B_1$. \par 
	Let the radius $\rho=\rho(n,s,\Lambda,\alpha,\gamma,q,\Gamma) \in (0,1)$ be given by Lemma \ref{thm:higherdiffscalemeasdata}. \par 
	Now fix some $t$ and some $q$ in the ranges \eqref{range1}. Then there exists some small enough $\varepsilon >0$ such that for $t:=1+\gamma+\varepsilon>1+\gamma$, we have
	$$
		1 < t<\min \{\alpha,2s-1\}, \quad q \in \bigg [1,\frac{n}{n+t-2s} \bigg ).
	$$ In other words, $t$ and $q$ satisfy the assumptions \eqref{ranges} from Lemma \ref{thm:higherdiffscalemeasdata}.
	
For $z \in \overline B_{1/2}$, consider the scaled functions $u_z \in W^{s,2}(B_1) \cap L^1_{2s}(\ern)$ and $\mu_z \in C_0^\infty(\ern)$ given by
$$ u_z(x):=u(x/8+z), \quad \mu_z(x):=R^{2s} \mu(x/8+z)$$
and also 
$$A_z(x,y):= A(x/8+z,y/8+z).$$
We note that $u_z$ is a weak solution of $L_{A_z} u_z = \mu_z$ in $B_1$ and observe that $A_z \in \mathcal{L}_0(\Lambda)$ is $C^\alpha$ in $B_2$ with respect to $\Gamma/8$ and thus also with respect to $\Gamma$.
Thus, we can apply Lemma \ref{thm:higherdiffscalemeasdata} to $u_z$, which together with Proposition \ref{WNrel}, changing variables and Lemma \ref{tr} leads to
\begin{align*}
	||u||_{W^{1+\gamma,q}(B_{\rho/8}(z))}
	& \leq C ||u||_{N^{t,q}(B_{\rho/8}(z))} \\ & = C ||u_z||_{N^{t,q}(B_{\rho})} \\ & \leq C \left ( \left (\mean{B_1} |u_z|^q dx \right )^{1/q} + \int_{\mathbb{R}^n \setminus B_1} \frac{|u_z(y)|}{|y|^{n+2s}}dy + |\mu_z|(B_1) \right ) \\
	& \leq C \left ( \left (\mean{B_{1/8}(z)} |u|^q dx \right )^{1/q} + \int_{\mathbb{R}^n \setminus B_{1/8}(z)} \frac{|u(y)|}{|z-y|^{n+2s}}dy + |\mu|(B_{1/8}(z)) \right ) \\
	& \leq C \left ( \left (\mean{B_{1}} |u|^q dx \right )^{1/q} + \int_{\mathbb{R}^n \setminus B_{1}} \frac{|u(y)|}{|y|^{n+2s}}dy + |\mu|(B_{3/4}) \right ),
\end{align*}
where all constants depend only on $n,s,\Lambda,\alpha,\Gamma,t,q$. Since $\left \{B_{\rho /8}(z) \right \}_{z \in \overline B_{1/2}}$ is an open covering of $\overline B_{1/2}$ and $\overline B_{1/2}$ is compact, there exists a finite subcover $\left \{B_{\rho /8}(z_i) \right \}_{i=1}^l$ of $B_{1/2}$. Therefore, applying the last display to $z=z_i$ and summing the resulting estimates over $i=1,...,l$, we arrive at the estimate
\begin{align*}
	||u||_{W^{1+\gamma,q}(B_{1/2})}
	\leq C \left ( \left (\mean{B_{1}} |u|^q dx \right )^{1/q} + \int_{\mathbb{R}^n \setminus B_{1}} \frac{|u(y)|}{|y|^{n+2s}}dy + |\mu|(B_{3/4}) \right ),
\end{align*}
where $C$ depends only on $n,s,\Lambda,\alpha,\Gamma,q$.
This implies the estimate \eqref{HDE1} in the case when $R=1$, $x_0=0$ and $\Omega=B_1$. \par 
Next, let us prove the desired estimate \eqref{HDE1} for any $R \in (0,1]$, $x_0 \in \Omega$ such that $B_R(x_0) \Subset \Omega$.
Consider the scaled functions
$$ u_1(x):=u(Rx+x_0), \quad \mu_1(x):=R^{2s} \mu(Rx+x_0)$$
and the scaled coefficient
$$ A_1(x,y):=A(Rx+x_0,Ry+x_0)$$
and observe that $u_1 \in W^{s,2}(B_1) \cap L^1_{2s}(\mathbb{R}^n)$ is a weak solution of $L_{A_1} u_1=\mu_1$ in $B_1$ and that $A_1$ belongs to $\mathcal{L}_0(\Lambda)$ and is $C^\alpha$ in $B_1$ with respect to $\Gamma R^{\alpha}$ and thus also with respect to $\Gamma \max\{\textnormal{diam}(\Omega),1\}^\alpha$. Therefore, by the first part of the proof along with rescaling we deduce that
\begin{align*}
	& \left (\mean{B_{R/2}(x_0)} |\nabla u|^q dx \right )^{1/q} + R^{\gamma }\left (\mean{B_{R/2}(x_0)} \int_{B_{R/2}(x_0)} \frac{|\nabla u(x)-\nabla u(y)|^q}{|x-y|^{n+\gamma q}} dydx \right )^{1/q} \\
	& = C R^{-1} \left (\int_{B_{1/2}} |\nabla u|^q dx \right )^{1/q} + C R^{-1}\left (\int_{B_{1/2}} \int_{B_{1/2}} \frac{|\nabla u(x)-\nabla u(y)|^q}{|x-y|^{n+\gamma q}} dydx \right )^{1/q} \\
	& \leq C R^{-1} ||u||_{W^{1+\gamma,q}(B_{1/2})} \\
	& \leq \frac{C}{R} \left (\left (\int_{B_{1}} |u|^q dx \right )^{1/q} + \int_{\mathbb{R}^n \setminus B_1} \frac{|u(y)|}{|x_0-y|^{n+2s}}dy + |\mu|(B_1) \right ) \\
	& = \frac{C}{R} \left (\left (\mean{B_{R}(x_0)} |u|^q dx \right )^{1/q} + R^{2s}\int_{\mathbb{R}^n \setminus B_R(x_0)} \frac{|u(y)|}{|x_0-y|^{n+2s}}dy + R^{2s-n} |\mu|(B_{3R/4}(x_0)) \right ),
\end{align*}
where all constants depend only on $n,s,\Lambda,\alpha,\Gamma,\gamma,q$ and $\max\{\textnormal{diam}(\Omega),1\}$, so that the proof of the estimate \eqref{HDE1} is finished. \par 
Finally, since for $q$ in the range given by \eqref{range1} we always have $q<2$, we also conclude that the right-hand side of \eqref{HDE1} is always finite, so that $u \in W^{1+\gamma,q}_{loc}(\Omega)$. Thus, the proof is finished.
\end{proof}

As a consequence of the gradient estimate \eqref{HDE1}, we obtain the following comparison estimate on the gradient level.

\begin{corollary} \label{cor:gradcomp} 
	Let $s \in (1/2,1)$, $R \in (0,1]$, $x_0 \in \ern$ and assume that $A \in \mathcal{L}_0(\Lambda)$ is $C^\alpha$ in $B_{2R}(x_0)$ with respect to some $\alpha>0$ and some $\Gamma>0$.
	Moreover, let $\mu \in C_0^\infty(\mathbb{R}^n)$ and assume that $u \in W^{s,2}(B_{2R}(x_0)) \cap L^1_{2s}(\mathbb{R}^n)$ is a weak solution of $L_{A} u = \mu$ in $\Omega$. Moreover, consider the weak solution $v \in W^{s,2}(B_{2R}(x_0)) \cap L^1_{2s}(\mathbb{R}^n)$ of
	\begin{equation} \label{constcof31xz}
		\begin{cases} \normalfont
			L_{A} v = 0 & \text{ in } B_{R}(x_0) \\
			v = u & \text{ a.e. in } \mathbb{R}^n \setminus B_{R}(x_0).
		\end{cases}
	\end{equation}
	Then for any $q \in \Big [1,\frac{n}{n+1-2s} \Big )$, we have the comparison estimate
	$$
	\left (\mean{B_{R/2}(x_0)} |\nabla u-\nabla v|^q dx \right )^{1/q} \leq C R^{2s-1-n} |\mu|(B_R(x_0)),
	$$
	where $C$ depends only on $n,s,\Lambda,\alpha,\Gamma,q$.
\end{corollary}

\begin{proof}
	Observe that $w:=u-v \in W^{s,2}_0(B_R(x_0)) \subset W^{s,2}(B_{2R}(x_0)) \cap L^1_{2s}(\ern)$ is a weak solution of $L_A w = \mu$ in $B_R(x_0)$. Therefore, using the estimate \eqref{HDE1} from Corollary \ref{cor:higherdiff} with respect to $w$ and then the comparison estimate from Proposition \ref{meascomp}, we obtain
	\begin{align*}
		& \left (\mean{B_{R/2}(x_0)} |\nabla u-\nabla v|^q dx \right )^{1/q} = \left (\mean{B_{R/2}(x_0)} |\nabla w|^q dx \right )^{1/q} \\
		& \leq \frac{C}{R} \Bigg (\left (\mean{B_{R}(x_0)} |w|^q dx \right )^{1/q} + R^{2s} \underbrace{\int_{\mathbb{R}^n \setminus B_R(x_0)} \frac{|w(y)|}{|x_0-y|^{n+2s}}dy}_{=0} + R^{2s-n} |\mu|(B_{R}(x_0)) \Bigg ) \\
		& \leq C R^{2s-1-n} |\mu|(B_{R}(x_0)),
	\end{align*}
	where all constants depends only on $n,s,\Lambda,\alpha,\Gamma,q$, finishing the proof.
\end{proof}

Together with a compactness argument, Corollary \ref{cor:higherdiff} also enables us to upgrade the convergence of the approximation sequence as given in the definition of SOLA to the gradient level.
\begin{corollary}[Upgrade of convergence] \label{upgrade} 
	Let $s \in (1/2,1)$ and assume that $A \in \mathcal{L}_0(\Lambda)$ is $C^\alpha$ for some $\alpha>0$ in a domain $\Omega$.
	Moreover, let $\mu \in \mathcal{M}(\mathbb{R}^n)$ and assume that $u$ is a SOLA of $L_{A} u = \mu$ in $\Omega$.
	In addition, let $\{u_j\}_{j=1}^\infty$ be the approximating sequence from Definition \ref{SOLA}. Then for all
	$$
		0 \leq \gamma<\min \{\alpha,2s-1\}, \quad q \in \bigg [1,\frac{n}{n+1+\gamma-2s} \bigg ),
	$$
	any $R>0$ and any $x_0 \in \Omega$ such that $B_{R}(x_0) \Subset \Omega$, we have $ u \in W^{1+\gamma,q}(B_{R/2}(x_0))$ and
	\begin{equation} \label{upc}
		u_j \to u \text{ in } W^{1+\gamma,q}(B_{R/2}(x_0)).
	\end{equation}
\end{corollary}

\begin{proof}
	Consider the approximating sequences $\{\mu_j\}_{j=1}^\infty$, $\{g_j\}_{j=1}^\infty$ from Definition \ref{SOLA}. There exists some sufficiently small $\varepsilon>0$ such that
	$$ 0 \leq \gamma+\varepsilon<\min \{\alpha,2s-1\}, \quad q \in \bigg [1,\frac{n}{n+1+\gamma+\varepsilon-2s} \bigg ). $$
	Since by Definition \ref{SOLA} we have $u \in W^{h,q}(\Omega)$ for any $h \in (0,s)$ and any $q \in [1,\frac{n}{n-s})$, by the fractional Sobolev embedding we have $u \in L^q_{loc}(\Omega)$ for any $q \in \big [1,\frac{n}{n-2s} \big )$ and in particular for any $q$ in the range \eqref{rangesm}. 
	Thus, together with Corollary \ref{cor:higherdiff} applied with $\gamma + \varepsilon$ instead of $\gamma$ and the convergence properties of all the sequences involved, up to passing to subsequences we obtain
	\begin{align*}
		& ||u_j||_{W^{1+\gamma+\varepsilon,q}(B_{R/2}(x_0))} \\ & \leq C \left ( \left (\int_{B_R(x_0)} |u_j|^q dx \right )^{1/q} + \int_{\mathbb{R}^n \setminus B_R(x_0)} \frac{|u_j(y)|}{|x_0-y|^{n+2s}}dy+ |\mu_j|(B_{3R/4}(x_0)) \right ) \\
		& \leq C \left ( \left (\int_{B_{R}(x_0)} |u|^q dx \right )^{1/q} + \int_{\mathbb{R}^n \setminus B_R(x_0)} \frac{|u(y)|}{|x_0-y|^{n+2s}}dy+ |\mu|(B_R(x_0)) \right )<\infty,
	\end{align*}
	where $C$ depends on $n,s,\Lambda,\alpha,\Gamma,R,\textnormal{diam}(\Omega),\gamma,q$, but not on $j$, so that the sequence $\{u_j\}_{j=1}^\infty$ is bounded in $W^{1+\gamma+\varepsilon,q}(B_{R/2}(x_0))$. \par Therefore, since e.g.\ by \cite[Proposition 4.6]{Tr3}, the space $W^{1+\gamma+\varepsilon,q}(B_{R/2}(x_0))$ is compactly embedded into $W^{1+\gamma,q}(B_{R/2}(x_0))$, up to passing to another subsequence there exists a function $\widetilde u \in W^{1+\gamma,q}(B_{R/2}(x_0))$ such that
	$$ u_j \to \widetilde u \text { in } W^{1+\gamma,q}(B_{R/2}(x_0)).$$
	But since we already know that $$ u_j \to \widetilde u \text { in } L^1(B_{R}(x_0)),$$ by uniqueness of the limit we deduce that $\widetilde u = u$ a.e.\ in $B_{R/2}(x_0)$, which in particular yields the convergence of gradients (\ref{upc}), so that the proof is finished.
\end{proof}

Theorem \ref{HD} now follows by combining Corollary \ref{cor:higherdiff} with the approximation properties of SOLA given by Corollary \ref{upgrade} and Definition \ref{SOLA}.
\begin{proof}[Proof of Theorem \ref{HD}]
Consider the approximating sequences $\{u_j\}_{j=1}^\infty$, $\{\mu_j\}_{j=1}^\infty$, $\{g_j\}_{j=1}^\infty$ from Definition \ref{SOLA}. Since for any $j$, $u_j$ is a weak solution of $L_A u_j=\mu_j$ in $\Omega$ and $\mu_j \in C_0^\infty(\ern)$, we can apply the estimate \eqref{HDE1} from Corollary \ref{cor:higherdiff} to $u_j$. Together with the convergence properties of the approximating sequences from Corollary \ref{upgrade} and Definition \ref{SOLA}, for any $\gamma$, $q$ as in \eqref{rangesm} and all $R>0$, $x_0 \in \Omega$ with $B_R(x_0) \Subset \Omega$ we obtain
\begin{align*}
	& \left (\mean{B_{R/2}(x_0)} |\nabla u|^q dx \right )^{1/q} + R^{\gamma }\left (\mean{B_{R/2}(x_0)} \int_{B_{R/2}(x_0)} \frac{|\nabla u(x)-\nabla u(y)|^q}{|x-y|^{n+\gamma q}} dydx \right )^{1/q} \\
	& \leq \lim_{j \to \infty} \left (\left (\mean{B_{R/2}(x_0)} |\nabla u_j|^q dx \right )^{1/q} + R^{\gamma }\left (\mean{B_{R/2}(x_0)} \int_{B_{R/2}(x_0)} \frac{|\nabla u_j(x)-\nabla u_j(y)|^q}{|x-y|^{n+\gamma q}} dydx \right )^{1/q} \right ) \\
	& \leq \frac{C}{R} \lim_{j \to \infty} \left (\left (\mean{B_{R}(x_0)} |u_j|^q dx \right )^{1/q} + R^{2s}\int_{\mathbb{R}^n \setminus B_R(x_0)} \frac{|u_j(y)|}{|x_0-y|^{n+2s}}dy + R^{2s-n} |\mu_j|(B_{3R/4}(x_0)) \right ) \\
	& \leq \frac{C}{R} \left (\left (\mean{B_{R}(x_0)} |u|^q dx \right )^{1/q} + R^{2s}\int_{\mathbb{R}^n \setminus B_R(x_0)} \frac{|u(y)|}{|x_0-y|^{n+2s}}dy + R^{2s-n} |\mu|(B_{R}(x_0)) \right ),
\end{align*}
where $C$ depends only on $n,s,\Lambda,\alpha,\Gamma,\gamma,q$ and $\max\{\text{diam} (\Omega),1\}$, proving the estimate \eqref{HDE}. \par 
Finally, since by the same reasoning as in the proof of Corollary \ref{upgrade} we have $u \in L^q_{loc}(\Omega)$ for any $q$ in the range \eqref{rangesm}, in view of standard covering arguments we obtain that $u \in W^{1+\gamma,q}_{loc}(\Omega)$, so that the proof is finished.
\end{proof}

\section{Gradient potential estimates}

The following excess decay lemma can be thought of as an oscillation-type inequality for $\nabla u$ tailored to measure data problems. In addition to being a central tool in the proof of the gradient potential estimates from Theorem \ref{GPE}, it might also be interesting for its own sake.

\begin{lemma}(Excess decay) \label{thm:GradOscDec2} 
	Let $s \in (1/2,1)$, $R>0$, $x_0 \in \mathbb{R}^n$, $\mu \in C_0^\infty(\mathbb{R}^n)$ and suppose that $A \in \mathcal{L}_0(\Lambda)$ is $C^\alpha$ in $B_{8R}(x_0)$ for some $\alpha \in (0,1)$. In addition, fix some $\gamma \in (0,\min\{\alpha,2s-1\})$ and some $\widetilde \gamma \in (\gamma,\min \{\alpha,2s-1 \}]$. Then there exists some small enough radius $R_0=R_0(n,s,\Lambda,\gamma,\widetilde \gamma, \Gamma) \in (0,1)$ such that the following is true for any $R \in (0,R_0]$. Let $u \in W^{s,2}(B_{2R}(x_0)) \cap L^1_{2s}(\mathbb{R}^n)$ be a weak solution to $L_{A} u = \mu$ in $B_{2R}(x_0)$.
	Then for all integers $m \geq 1$, $N \geq 0$, we have
	\begin{equation} \label{gradoscdecay2}
		\begin{aligned}
			E_R(\nabla u,x_0,2^{-(N+m)}) & \leq C_0 2^{-\gamma m} E_R(\nabla u,x_0,2^{-N}) \\
			& \quad + C_0 2^{-\gamma m} 2^{-\widetilde \gamma N} \sum_{k=0}^{N} 2^{k(1+\widetilde \gamma -2s)} |(\nabla u)_{B_{2^{k-N}R}(x_0)}| \\
			& \quad + C_0 2^{-\gamma m} (2^{-N}R)^{2s-1} \int_{\mathbb{R}^n \setminus B_{R}(x_0)} \frac{|u  - \left ( u \right )_{B_{R}(x_0)}| }{|x_0-y|^{n+2s}}dy \\
			& \quad + C_0 2^{nm} (2^{-N}R)^{2s-1-n} |\mu|(B_{2^{-N} R}(x_0)),
		\end{aligned}
	\end{equation}
	where $C_0 \geq 1$ depends only on $n,s,\Lambda,\gamma,\widetilde \gamma,\Gamma$ and for any integer $i \geq 0$,
	\begin{align*}
	E_R(\nabla u,x_0,2^{-i}):& = \sum_{k=0}^{i} 2^{(1-2s)k} \mean{B_{2^{k-i}R}(x_0)} |\nabla u - \left ( \nabla u \right )_{B_{2^{k-i}R}(x_0)} | dx .
	\end{align*}
\end{lemma}

\begin{proof}
First of all, let $R_0=R_0(n,s,\Lambda,\gamma,\widetilde \gamma,\Gamma) \in (0,1)$ be given by Corollary \ref{thm:GradOscDec1} with $p=0$ and with $\alpha$ replaced by $\widetilde \gamma$ and fix some $R \in (0,R_0]$. Now consider the weak solution $v \in W^{s,2}(B_{2R}(x_0)) \cap L^1_{2s}(\mathbb{R}^n)$ of
\begin{equation} \label{constcof31x}
	\begin{cases} \normalfont
		L_{A} v = 0 & \text{ in } B_{2^{-N} R}(x_0) \\
		v = u & \text{ a.e. in } \mathbb{R}^n \setminus B_{2^{-N} R}(x_0).
	\end{cases}
\end{equation}
Using the comparison estimate from Corollary \ref{cor:gradcomp}, we obtain that
\begin{align*}
	& \sum_{k=0}^{m-4} 2^{(1-2s)k} \mean{B_{2^{k-m-N}R}(x_0)} |\nabla u - \left ( \nabla u \right )_{B_{2^{k-m-N}R}(x_0)}| dx \\
	& \leq \sum_{k=0}^{m-4} 2^{(1-2s)k} \mean{B_{2^{k-m-N}R}(x_0)} |\nabla v - \left ( \nabla v \right )_{B_{2^{k-m-N}R}(x_0)}| dx \\
	& \quad + 2 \sum_{k=0}^{\infty} 2^{(1-2s)k} 2^{mn} \mean{B_{2^{-N-1}R}(x_0)} |\nabla u-\nabla v|dx \\
	& \leq \sum_{k=0}^{m-4} 2^{(1-2s)k} \mean{B_{2^{k-m-N}R}(x_0)} |\nabla v - \left ( \nabla v \right )_{B_{2^{k-m-N}R}(x_0)}| dx \\
	& \quad + C_1 2^{nm} (2^{-N}R)^{2s-1-n} |\mu|(B_{2^{-N} R}(x_0)).
\end{align*}
Observe that $A$ being $C^\alpha$ with respect to $\Gamma$ in $B_{8R}(x_0)$ clearly implies that $A$ is $C^{\widetilde \gamma}$ with respect to $8 \Gamma$ in $B_{8R}(x_0)$.
Thus, for any $k \in \{0,...,m-4\}$, Corollary \ref{thm:GradOscDec1} with $\rho=2^{k-m} \in (0,1/16]$, $$\ell(x):=\left ( u \right )_{B_{2^{-N}R}(x_0)} + \left (\nabla u \right )_{B_{2^{-N}R}(x_0)} \cdot (x-x_0), \quad c:=\left ( u \right )_{B_{2^{-N}R}(x_0)}$$ and with $\alpha$ replaced by $\widetilde \gamma$ along with the comparison estimate from Proposition \ref{meascomp} yields
\begin{align*}
	& 2^{(1-2s)k} \mean{B_{2^{k-m-N}R}(x_0)} |\nabla v - \left ( \nabla v \right )_{B_{2^{k-m-N}R}(x_0)}| dx \\
	& \leq 2^{(1-2s)k} \osc_{B_{2^{k-m-N}R}(x_0)} \nabla v \\
	& \leq C_2 2^{(1-2s)k} 2^{(k-m)\gamma} (2^{-N}R)^{-1} \Bigg ( \mean{B_{2^{-N} R}(x_0)} |v_\ell| dx + (2^{-N} R)^{2s} \int_{\mathbb{R}^n \setminus B_{2^{-N} R}(x_0)} \frac{|v_\ell(y)| }{|x_0-y|^{n+2s}}dy \\
	& \quad + 2^{-\widetilde \gamma N} \sum_{i=0}^{N} 2^{k(\widetilde \gamma -2s)} \mean{B_{2^{i-N}R}(x_0)} |v-\left ( u \right )_{B_{2^{-N}R}(x_0)}|dx \\ & \quad + 2^{-2sN} R^{2s} \int_{\mathbb{R}^n \setminus B_{R}(x_0)} \frac{|v(y)-\left ( u \right )_{B_{2^{-N}R}(x_0)}|}{|x_0-y|^{n+2s}}dy \Bigg ) \\
	& \leq C_3 2^{(1+\gamma-2s)k} 2^{-\gamma m} (2^{-N}R)^{-1} \Bigg ( \mean{B_{2^{-N} R}(x_0)} |u_\ell| dx + (2^{-N} R)^{2s} \int_{\mathbb{R}^n \setminus B_{2^{-N} R}(x_0)} \frac{|u_\ell(y)| }{|x_0-y|^{n+2s}}dy \\
	& \quad + 2^{-\widetilde \gamma N} \sum_{i=0}^{N} 2^{k(\widetilde \gamma -2s)} \mean{B_{2^{i-N}R}(x_0)} |u-\left ( u \right )_{B_{2^{-N}R}(x_0)}|dx \\ & \quad + 2^{-2sN} R^{2s} \int_{\mathbb{R}^n \setminus B_{R}(x_0)} \frac{|u(y)-\left ( u \right )_{B_{2^{-N}R}(x_0)}|}{|x_0-y|^{n+2s}}dy \\
	& \quad + \underbrace{\sum_{i=0}^{\infty} 2^{i(\gamma -2s-n)}}_{<\infty} (2^{-N} R)^{2s-n} |\mu|(B_{2^{-N} R}(x_0)) \Bigg ).
\end{align*}
Let us further estimate the terms on the right-hand side. First of all, by Lemma \ref{tailestaffine} we have
\begin{equation} \label{annuli2}
\begin{aligned}
	& \mean{B_{2^{-N} R}(x_0)} |u_\ell| dx + (2^{-N} R)^{2s} \int_{\mathbb{R}^n \setminus B_{2^{-N} R}(x_0)} \frac{|u_\ell(y)| }{|x_0-y|^{n+2s}}dy \\
	& \leq C_4 2^{-N} R \Bigg ( \sum_{i=0}^{N} 2^{(1-2s)i} \mean{B_{2^{k-N}R}(x_0)} |\nabla u-(\nabla u)_{B_{2^{k-N}R}(x_0)}|dx \\ & \quad + 2^{(1-2s)N} \sum_{i=0}^{N} \mean{B_{2^{i-N}R}(x_0)} |\nabla u|dx + (2^{-N} R)^{2s-1}  \int_{\mathbb{R}^n \setminus B_R(x_0)} \frac{|u(y)-(u)_{B_R(x_0)}|} {|x_0-y|^{n+2s}} dy \Bigg ) \\
	& \leq C_4 2^{-N} R \Bigg ( E_R(\nabla u,x_0,2^{-N})  + 2^{-\gamma N} \sum_{i=0}^{N} 2^{(1+\gamma-2s)i} \mean{B_{2^{i-N}R}(x_0)} |\nabla u|dx \\ &  + 2^{(1-2s) N} R^{2s-1}  \int_{\mathbb{R}^n \setminus B_R(x_0)} \frac{|u(y)-(u)_{B_R(x_0)}|} {|x_0-y|^{n+2s}} dy \Bigg ).
\end{aligned}
\end{equation}

Next, by a similar reasoning as in the proof of Lemma \ref{tailest} along with the classical Poincar\'e inequality we obtain
\begin{align*}
	& 2^{-\widetilde \gamma N} \sum_{i=0}^{N} 2^{i(\widetilde \gamma -2s)} \mean{B_{2^{i-N}R}(x_0)} |u-\left ( u \right )_{B_{2^{-N}R}(x_0)}|dx \\ & \quad + 2^{-2sN} R^{2s} \int_{\mathbb{R}^n \setminus B_{R}(x_0)} \frac{|u(y)-\left ( u \right )_{B_{2^{-N}R}(x_0)}|}{|x_0-y|^{n+2s}}dy \\
	& \leq C_5 2^{-\widetilde \gamma N} \sum_{i=1}^{N} 2^{i(\widetilde \gamma -2s)} \mean{B_{2^{i-N}R}(x_0)} |u-\left ( u \right )_{B_{2^{i-N}R}(x_0)}|dx \\ & \quad + 2^{-2sN} R^{2s} \int_{\mathbb{R}^n \setminus B_{R}(x_0)} \frac{|u(y)-\left ( u \right )_{B_{R}(x_0)}|}{|x_0-y|^{n+2s}}dy \\
	& \leq C_6 2^{-N} R 2^{-\widetilde \gamma N} \sum_{i=0}^{N} 2^{i(1+\widetilde \gamma -2s)} \mean{B_{2^{i-N}R}(x_0)} |\nabla u|dx \\ & \quad + 2^{-N} R 2^{(1-2s) N} R^{2s} \int_{\mathbb{R}^n \setminus B_{R}(x_0)} \frac{|u(y)-\left ( u \right )_{B_{R}(x_0)}|}{|x_0-y|^{n+2s}}dy.
\end{align*}
Combining all the above observations now leads to
\begin{align*}
	& \sum_{k=0}^{m-4} 2^{(1-2s)k} \mean{B_{2^{k-m-N}R}(x_0)} |\nabla u - \left ( \nabla u \right )_{B_{2^{k-m-N}R}(x_0)}| dx \\
	& \leq C_7 \underbrace{\sum_{k=0}^{\infty} 2^{(1+\gamma-2s)k}}_{<\infty} 2^{-\gamma m} \Bigg ( E_R(\nabla u,x_0,2^{-N}) +2^{-\widetilde \gamma N} \sum_{i=0}^{N} 2^{i(1+\widetilde \gamma -2s)} \mean{B_{2^{i-N}R}(x_0)} |\nabla u|dx \\ & \quad + 2^{-2s N} R^{2s} \int_{\mathbb{R}^n \setminus B_{R}(x_0)} \frac{|u(y)-\left ( u \right )_{B_{R}(x_0)}|}{|x_0-y|^{n+2s}}dy + 2^{(n+\gamma)m} (2^{-N} R)^{2s-1-n} |\mu|(B_{2^{-N} R}(x_0)) \Bigg ).
\end{align*}
In order to proceed, observe that
\begin{align*}
	& 2^{-\widetilde \gamma N} \sum_{i=0}^{N} 2^{i(1+\widetilde \gamma -2s)} \mean{B_{2^{i-N}R}(x_0)} |\nabla u|dx \\
	& \leq \sum_{i=0}^{N} 2^{i(1 -2s)} \mean{B_{2^{i-N}R}(x_0)} |\nabla u-(\nabla u)_{B_{2^{i-N}R}(x_0)}|dx \\
	& \quad + 2^{-\widetilde \gamma N} \sum_{i=0}^{N} 2^{i(1+\widetilde \gamma -2s)} |(\nabla u)_{B_{2^{i-N}R}(x_0)}| \\
	& = E_R(\nabla u,x_0,2^{-N}) + 2^{-\widetilde \gamma N} \sum_{i=0}^{N} 2^{i(1+\widetilde \gamma -2s)} |(\nabla u)_{B_{2^{i-N}R}(x_0)}|.
\end{align*}
By combining the previous two displays, we arrive at
\begin{align*}
	& \sum_{k=0}^{m-4} 2^{(1-2s)k} \mean{B_{2^{k-m-N}R}(x_0)} |\nabla u - \left ( \nabla u \right )_{B_{2^{k-m-N}R}(x_0)}| dx \\
	& \leq C_8 2^{-\gamma m} \Bigg ( E_R(\nabla u,x_0,2^{-N}) +2^{-\widetilde \gamma N} \sum_{k=0}^{N} 2^{k(1+\widetilde \gamma -2s)} \mean{B_{2^{k-N}R}(x_0)} |\nabla u|dx \\ & \quad + 2^{(1-2s) N} R^{2s} \int_{\mathbb{R}^n \setminus B_{R}(x_0)} \frac{|u(y)-\left ( u \right )_{B_{R}(x_0)}|}{|x_0-y|^{n+2s}}dy + 2^{(n+\gamma)m} (2^{-N} R)^{2s-1-n} |\mu|(B_{2^{-N} R}(x_0)) \Bigg ).
\end{align*}

Finally, observe that
\begin{align*}
	& \sum_{k=m-3}^{N+m} 2^{(1-2s)k} \mean{B_{2^{k-m-N}R}(x_0)} |\nabla u - \left ( \nabla u \right )_{B_{2^{k-m-N}R}(x_0)}| dx \\
	& \leq C_9 \sum_{k=m}^{N+m} 2^{(1-2s)k} \mean{B_{2^{k-m-N}R}(x_0)} |\nabla u - \left ( \nabla u \right )_{B_{2^{k-m-N}R}(x_0)}| dx \\
	& = C_9 2^{(1-2s)m} \sum_{k=0}^{N} 2^{(1-2s)k} \mean{B_{2^{k-m-N}R}(x_0)} |\nabla u - \left ( \nabla u \right )_{B_{2^{k-m-N}R}(x_0)}| dx \\
	& \leq C_9 2^{-\gamma m} \sum_{k=0}^{N} 2^{(1-2s)k} \mean{B_{2^{k-m-N}R}(x_0)} |\nabla u - \left ( \nabla u \right )_{B_{2^{k-m-N}R}(x_0)}| dx \\
	& = C_9 2^{-\gamma m}E_R(\nabla u,x_0,2^{-N}).
\end{align*}
Combining the previous two displays now yields the desired estimate (\ref{gradoscdecay2}) with $C_0=C_{8}+C_9$, so that the proof is finished.

\end{proof}
We are now in the position to prove Theorem \ref{GPE}.
\begin{proof}[Proof of Theorem \ref{GPE}]
Set $$\gamma:=\min \{\alpha,2s-1\}/2>0, \quad \widetilde \gamma:=3\min \{\alpha,2s-1\}/4 \in (\gamma,\min \{\alpha,2s-1 \}]$$ and let $C_0 = C_0(n,s,\Lambda,\gamma,\widetilde \gamma,\Gamma) \geq 1$ be given by Lemma \ref{thm:GradOscDec2}. Moreover, let $R_0=R_0(n,s,\Lambda,\gamma,\widetilde \gamma,\Gamma) \in (0,1)$ be given by Lemma \ref{thm:GradOscDec2}. In addition, fix some $x_0 \in \Omega$. \par 
First, fix some $R \in (0,R_0]$ such that $B_{8R}(x_0) \subset \Omega$. \par 
Let $\{u_j\}_{j=1}^\infty$, $\{\mu_j\}_{j=1}^\infty$ and $\{g_j\}_{j=1}^\infty$ be the approximating sequences from Definition \ref{SOLA}.
Let $m \geq 1$ to be chosen large enough in a way such that $m$ only depends on $n,s,\Lambda,\alpha,\Gamma$. In particular, we require $m$ to be large enough such that
\begin{equation} \label{llarge}
C_0 2^{-\gamma m/2} \leq \frac{1}{2},
\end{equation}
observe that this in particular implies that for any integer $i \geq 0$,
\begin{equation} \label{llargec}
C_0 2^{-\gamma m(i+1)/2} \leq 2^{-(i+1)}.
\end{equation}
Moreover, for any integer $i \geq 0$, set $$
E_i:=E_R(\nabla u_j,x_0,2^{-im}).$$ By Lemma \ref{thm:GradOscDec2} with $N=im$ and (\ref{llarge}), (\ref{llargec}), for any integer $i \geq 0$ we have
\begin{equation} \label{Ejdec}
	\begin{aligned}
	E_{i+1} & \leq \frac{1}{2} E_{i} + 2^{-\gamma m/2} 2^{- (i+1)} \sum_{k=0}^{im} 2^{k(1+\widetilde \gamma -2s)} |(\nabla u_j)_{B_{2^{k-im}R}(x_0)}| \\ & \quad + 2^{- (i+1)} R^{2s-1} \int_{\mathbb{R}^n \setminus B_{R}(x_0)} \frac{|u_j  - \left ( u_j \right )_{B_{R}(x_0)}| }{|x_0-y|^{n+2s}}dy \\ & \quad + C_0 2^{nm} (2^{-im}R)^{2s-1-n} |\mu_j|(B_{2^{-im} R}(x_0)) .
	\end{aligned}
\end{equation}
For $l\in \mathbb{N}$, summing (\ref{Ejdec}) over $i \in \{0,...,l-1\}$ leads to
\begin{align*}
	\sum_{i=1}^{l} E_{i} & \leq \frac{1}{2} \sum_{i=0}^{l-1} E_{i} + 2^{-\gamma m/2} \sum_{i=0}^{l-1} 2^{- (i+1)} \sum_{k=0}^{im} 2^{k(1+\widetilde \gamma -2s)} |(\nabla u_j)_{B_{2^{k-im}R}(x_0)}| \\ & \quad + \sum_{i=0}^{l-1} 2^{- (i+1)} R^{2s-1} \int_{\mathbb{R}^n \setminus B_{R}(x_0)} \frac{|u_j(y)  - \left ( u_j \right )_{B_{R}(x_0)}| }{|x_0-y|^{n+2s}}dy \\ & \quad + C_0 2^{nm} \sum_{i=0}^{l-1} (2^{-im}R)^{2s-1-n} |\mu_j|(B_{2^{-im} R}(x_0)),
\end{align*}
so that by reabsorbing the first term on the right-hand side of the previous inequality we obtain
\begin{align*}
	\sum_{i=1}^{l} E_{i} & \leq 2 E_0 + 2^{-\gamma m/2} \sum_{i=0}^{l-1} 2^{-i} \sum_{k=0}^{im} 2^{k(1+\widetilde \gamma -2s)} |(\nabla u_j)_{B_{2^{k-im}R}(x_0)}| \\ & \quad + \sum_{i=0}^{l-1} 2^{-i} R^{2s-1} \int_{\mathbb{R}^n \setminus B_{R}(x_0)} \frac{|u_j(y) - \left ( u_j \right )_{B_{R}(x_0)}| }{|x_0-y|^{n+2s}}dy \\ & \quad + 2C_0 2^{nm} \sum_{i=0}^{l-1} (2^{-im}R)^{2s-1-n} |\mu_j|(B_{2^{-im} R}(x_0)).
\end{align*}
Thus, together with Corollary \ref{cor:higherdiff} we obtain
\begin{align*}
	& |(\nabla u_j)_{B_{2^{-(l+1)m}R}(x_0)}| \\ & \leq \sum_{i=0}^l (|(\nabla u_j)_{B_{2^{-(i+1)m}R}(x_0)}-(\nabla u_j)_{B_{2^{-im}R}(x_0)}|)+|(\nabla u_j)_{B_{R}(x_0)}| \\
	& \leq \sum_{i=0}^l \mean{B_{2^{-(i+1)m}R}(x_0)} |\nabla u_j - (\nabla u_j)_{B_{2^{-im}R}(x_0)}| dx + |(\nabla u_j)_{B_{R}(x_0)}| \\
	& \leq 2^{nm} \sum_{i=0}^l E_i + \mean{B_{R}(x_0)} |\nabla u_j| dx \\
	& \leq 2^{nm+2} \mean{B_{R}(x_0)} |\nabla u_j| dx + 2^{nm} 2^{-\gamma m/2} \sum_{i=0}^{l-1} 2^{-i} \sum_{k=0}^{im} 2^{k(1+\widetilde \gamma -2s)} |(\nabla u_j)_{B_{2^{k-im}R}(x_0)}| \\ & \quad + 2^{nm} \underbrace{\sum_{i=0}^{\infty} 2^{-i}}_{<\infty} R^{2s-1} \int_{\mathbb{R}^n \setminus B_{R}(x_0)} \frac{|u_j(y) - \left ( u_j \right )_{B_{R}(x_0)}| }{|x_0-y|^{n+2s}}dy \\ & \quad +2C_0 2^{n(m+1)} \sum_{i=0}^{l-1} (2^{-im}R)^{2s-1-n} |\mu_j|(B_{2^{-im} R}(x_0)) \\
	& \leq C_2 \Bigg (R^{-1} \mean{B_{2R}(x_0)} |u_j| dx + R^{2s-1} \int_{\mathbb{R}^n \setminus B_{2R}(x_0)} \frac{|u_j(y)| }{|x_0-y|^{n+2s}}dy \\& \quad + 2^{-\gamma m/2} \sum_{i=0}^{l-1} 2^{-i} \sum_{k=0}^{im} 2^{k(1+\widetilde \gamma -2s)} |(\nabla u_j)_{B_{2^{k-im}R}(x_0)}| \\ & \quad + \sum_{i=0}^{l-1} (2^{-im}R)^{2s-1-n} |\mu_j|(B_{2^{-im} R}(x_0)) + R^{2s-1-n} |\mu_j|(B_{2R}(x_0)) \Bigg ),
\end{align*}
where $C_2$ depends only on $n,s,\Lambda,\alpha,\Gamma$, which is in particular because $m$ only depends on the aforementioned quantities. Now in view of Definition \ref{SOLA} and Corollary \ref{upgrade}, we can pass to the limit as $j \to \infty$ to obtain
\begin{equation} \label{avest}
	\begin{aligned}
		& |(\nabla u)_{B_{2^{-(l+1)m}R}(x_0)}| \\
		& \leq C_2 \Bigg (R^{-1} \mean{B_{2R}(x_0)} |u| dx + R^{2s-1} \int_{\mathbb{R}^n \setminus B_{2R}(x_0)} \frac{|u(y)| }{|x_0-y|^{n+2s}}dy \\& \quad + 2^{-\gamma m/2} \sum_{i=0}^{l-1} 2^{-i} \sum_{k=0}^{im} 2^{k(1+\widetilde \gamma -2s)} |(\nabla u)_{B_{2^{k-im}R}(x_0)}| \\ & \quad + \sum_{i=0}^{l-1} (2^{-im}R)^{2s-1-n} |\mu|(B_{2^{-im} R}(x_0)) + R^{2s-1-n} |\mu|(\overline B_{2R}(x_0)) \Bigg ).
	\end{aligned}
\end{equation}
For the last two terms on the right-hand side of \eqref{avest}, we estimate
\begin{align*}
	& \sum_{i=0}^{l-1} (2^{-im}R)^{2s-1-n} |\mu|(\overline B_{2^{-im} R}(x_0)) + R^{2s-1-n} |\mu|(\overline B_{2R}(x_0)) \\ & \leq \sum_{i=0}^{\infty} (2^{-im}R)^{2s-1-n} |\mu|(\overline B_{2^{-im} R}(x_0)) + R^{2s-1-n} |\mu|(\overline B_{2R}(x_0)) \\
	& \leq \frac{4^{n-2s+1}}{\log 4} \int_R^{4R} \frac{|\mu|(B_t(x_0))}{t^{n-2s+1}}\frac{dt}{t} + \frac{2^{m(n-2s+1)}}{m\log 2} \sum_{i=0}^\infty \int_{2^{-(i+1)l} R}^{2^{-il} R} \frac{|\mu|(B_t(x_0))}{t^{n-2s+1}}\frac{dt}{t} \\
	& \leq C_3 I^{|\mu|}_{2s-1}(x_0,4R),
\end{align*}
where $C_3:= \frac{2^{m(n-2s+1)}}{\log 2}$ depends only on $n,s,\Lambda,\alpha,\Gamma$. Combining the previous display with \eqref{avest} now leads to
\begin{equation} \label{avest2}
	\begin{aligned}
		|(\nabla u)_{B_{2^{-(l+1)m}R}(x_0)}|
		& \leq C_4 \Bigg (R^{-1} \mean{B_{2R}(x_0)} |u| dx + R^{2s-1} \int_{\mathbb{R}^n \setminus B_{2R}(x_0)} \frac{|u(y)| }{|x_0-y|^{n+2s}}dy \\ & \quad + 2^{-\gamma m/2} \sum_{i=0}^{l-1} 2^{-i} \sum_{k=0}^{im} 2^{k(1+\widetilde \gamma -2s)} |(\nabla u)_{B_{2^{k-im}R}(x_0)}|  + I^{|\mu|}_{2s-1}(x_0,4R) \Bigg ),
	\end{aligned}
\end{equation} 
where $C_4$ also only depends on $n,s,\Lambda,\alpha,\Gamma$. \par 
In addition, note that by Theorem \ref{HD} and Lemma \ref{tr} we have
\begin{equation} \label{indstart}
\begin{aligned}
	& |(\nabla u)_{B_{2^{-m}R}(x_0)}| \\ & \leq \mean{B_{2^{-m}R}(x_0)} |\nabla u| dx \\
	& \leq C_5 \Bigg (2^{m+1} R^{-1} \mean{B_{2^{-m} 2R}(x_0)} |u| dx + (2^{-m}2R)^{2s-1} \int_{\mathbb{R}^n \setminus B_{2^{-m}2R}(x_0)} \frac{|u(y)| }{|x_0-y|^{n+2s}}dy \\ & \quad + (2^{-m}2R)^{2s-1-n} |\mu|(B_{2^{-m}2R}(x_0)) \Bigg ) \\
	& \leq C_6 \Bigg (R^{-1} \mean{B_{8R}(x_0)} |u| dx + R^{2s-1} \int_{\mathbb{R}^n \setminus B_{8R}(x_0)} \frac{|u(y)| }{|x_0-y|^{n+2s}}dy + I^{|\mu|}_{2s-1}(x_0,8R) \Bigg ),
\end{aligned}
\end{equation}
where $C_5$ and $C_6$ depend only on $n,s,\Lambda,\alpha,\Gamma$.
Next, let us prove by induction that for any integer $l \geq 0$ we have
\begin{equation} \label{indpot}
|(\nabla u)_{B_{2^{-(l+1)m}R}(x_0)}| \leq 2 C_7 M,
\end{equation}
where
$$M:=R^{-1} \mean{B_{8R}(x_0)} |u| dx + R^{2s-1} \int_{\mathbb{R}^n \setminus B_{8R}(x_0)} \frac{|u(y)| }{|x_0-y|^{n+2s}}dy + I^{|\mu|}_{2s-1}(x_0,8R)$$
and $C_7:=\max\{C_4,C_6\}.$
First of all, the case $l=0$ follows directly from \eqref{indstart}. \par 
Next, assume that for some integer $l \geq 0$, \eqref{indpot} holds for any $i \in \{0,...,l\}$ and let us prove that \eqref{indpot} holds also for $l+1$. Indeed, in view of \eqref{avest2} and the induction hypothesis, we obtain
\begin{align*}
	& |(\nabla u)_{B_{2^{-(l+2)m}R}(x_0)}| \\
	& \leq C_4 M + C_4 2^{-\gamma m/2} \sum_{i=0}^{l} 2^{-i} \sum_{k=0}^{im} 2^{k(1+\widetilde \gamma -2s)} |(\nabla u)_{B_{2^{k-im}R}(x_0)}| \\
	& \leq C_7 M + C_4 2^{-\gamma m/2} \underbrace{\sum_{i=0}^{\infty} 2^{-i} \sum_{k=0}^{\infty} 2^{k(1+\widetilde \gamma -2s)}}_{<\infty} 2C_7 M \\
	& \leq C_7 M + 2^{-\gamma m/2} C_8 M \leq 2C_7M,
\end{align*}
where $C_8:=2C_4C_7 \sum_{i=0}^{\infty} 2^{-i} \sum_{k=0}^{\infty} 2^{k(1+\gamma -2s)}$ depends only on $n,s,\Lambda,\alpha,\Gamma$ and the last inequality was obtained by requiring $m$ to be large enough such that $$ 2^{-\gamma m/2} C_8 \leq C_7.$$ Thus, by the Lebesgue differentiation theorem, for any Lebesgue point $x_0$ of $\nabla u$ we obtain
\begin{equation} \label{NGPEprelim}
|\nabla u(x_0)| = \lim_{l \to \infty} |(\nabla u)_{B_{2^{-(l+1)m}R}(x_0)}| \leq 2C_7M. 
\end{equation} 
In the case when $R \in (0,R_0]$ and $x_0 \in \Omega$ satisfy $B_{R}(x_0) \subset \Omega$, the estimate \eqref{NGPE} now follows directly by applying the estimate \eqref{NGPEprelim} with $R$ replaced by $R/8$ and taking into account Lemma \ref{tr}. \par
Next, let us consider the remaining case when $R > R_0$ and $B_R(x_0) \subset \Omega$. Then using the previous case, we deduce
\begin{align*}
	|\nabla u(x_0)| & \leq C_7 \left ( R_0^{-1} \mean{B_{R_0}(x_0)} |u| dx + R_0^{2s-1} \int_{\mathbb{R}^n \setminus B_{R_0}(x_0)} \frac{|u(y)| }{|x_0-y|^{n+2s}}dy + I^{|\mu|}_{2s-1}(x_0,R_0) \right ) \\
	& \leq C_9 \left (\frac{\text{diam} (\Omega)}{R_0} \right )^{n+1} \\ & \quad \times \left (R^{-1} \mean{B_{R}(x_0)} |u| dx + R^{2s-1} \int_{\mathbb{R}^n \setminus B_{R}(x_0)} \frac{|u(y)| }{|x_0-y|^{n+2s}}dy + I^{|\mu|}_{2s-1}(x_0,R) \right ) \\
	& \leq C_{10} \left (\frac{R}{R_0} \right )^{n+1} \left (R^{-1} \mean{B_{R}(x_0)} |u| dx + R^{2s-1} \int_{\mathbb{R}^n \setminus B_{R}(x_0)} \frac{|u(y)| }{|x_0-y|^{n+2s}}dy + I^{|\mu|}_{2s-1}(x_0,R) \right )
\end{align*}
where $C_9$ and $C_{10}$ depend only on $n,s,\Lambda,\alpha,\Gamma$ and $\max\{\text{diam} (\Omega),1\}$. Thus, the proof is finished.
\end{proof}

\bibliographystyle{alpha}
\bibliography{biblio}

\newcommand{\etalchar}[1]{$^{#1}$}
\def\cprime{$'$} \def\cprime{$'$}
\begin{thebibliography}{BCD{\etalchar{+}}18}

\bibitem[AKM18]{AKMARMA}
Benny Avelin, Tuomo Kuusi, and Giuseppe Mingione.
\newblock Nonlinear {C}alder\'{o}n-{Z}ygmund theory in the limiting case.
\newblock {\em Arch. Ration. Mech. Anal.}, 227(2):663--714, 2018.

\bibitem[AKM19]{AKMBook}
Scott Armstrong, Tuomo Kuusi, and Jean-Christophe Mourrat.
\newblock {\em Quantitative stochastic homogenization and large-scale
  regularity}, volume 352 of {\em Grundlehren der mathematischen Wissenschaften
  [Fundamental Principles of Mathematical Sciences]}.
\newblock Springer, Cham, 2019.

\bibitem[BCD{\etalchar{+}}18]{BCDKS}
Dominic Breit, Andrea Cianchi, Lars Diening, Tuomo Kuusi, and Sebastian
  Schwarzacher.
\newblock Pointwise {C}alder\'{o}n-{Z}ygmund gradient estimates for the
  {$p$}-{L}aplace system.
\newblock {\em J. Math. Pures Appl. (9)}, 114:146--190, 2018.

\bibitem[BCF12]{BCF}
Clayton Bjorland, Luis Caffarelli, and Alessio Figalli.
\newblock Non-local gradient dependent operators.
\newblock {\em Adv. Math.}, 230(4-6):1859--1894, 2012.

\bibitem[Ber96]{bertoin}
Jean Bertoin.
\newblock {\em L\'{e}vy processes}, volume 121 of {\em Cambridge Tracts in
  Mathematics}.
\newblock Cambridge University Press, Cambridge, 1996.

\bibitem[BFV18]{BFVCalcVar}
Matteo Bonforte, Alessio Figalli, and Juan~Luis V\'{a}zquez.
\newblock Sharp boundary behaviour of solutions to semilinear nonlocal elliptic
  equations.
\newblock {\em Calc. Var. Partial Differential Equations}, 57(2):Paper No. 57,
  34, 2018.

\bibitem[BG89]{BGJFA}
Lucio Boccardo and Thierry Gallou\"{e}t.
\newblock Nonlinear elliptic and parabolic equations involving measure data.
\newblock {\em J. Funct. Anal.}, 87(1):149--169, 1989.

\bibitem[BKO23]{BKJMA}
Sun-Sig Byun, Hyojin Kim, and Jihoon Ok.
\newblock Local {H}\"{o}lder continuity for fractional nonlocal equations with
  general growth.
\newblock {\em Math. Ann.}, 387(1-2):807--846, 2023.

\bibitem[BL17]{BL}
Lorenzo Brasco and Erik Lindgren.
\newblock Higher {S}obolev regularity for the fractional {$p$}-{L}aplace
  equation in the superquadratic case.
\newblock {\em Adv. Math.}, 304:300--354, 2017.

\bibitem[BLS18]{BLS}
Lorenzo Brasco, Erik Lindgren, and Armin Schikorra.
\newblock Higher {H}\"{o}lder regularity for the fractional {$p$}-{L}aplacian
  in the superquadratic case.
\newblock {\em Adv. Math.}, 338:782--846, 2018.

\bibitem[BP16]{BP}
Lorenzo Brasco and Enea Parini.
\newblock The second eigenvalue of the fractional {$p$}-{L}aplacian.
\newblock {\em Adv. Calc. Var.}, 9(4):323--355, 2016.

\bibitem[BY19]{BYP}
Sun-Sig Byun and Yeonghun Youn.
\newblock Potential estimates for elliptic systems with subquadratic growth.
\newblock {\em J. Math. Pures Appl. (9)}, 131:193--224, 2019.

\bibitem[Caf89]{CaffFNAnnals}
Luis~A. Caffarelli.
\newblock Interior a priori estimates for solutions of fully nonlinear
  equations.
\newblock {\em Ann. of Math. (2)}, 130(1):189--213, 1989.

\bibitem[Cam64]{Campanato}
S.~Campanato.
\newblock Propriet\`a di una famiglia di spazi funzionali.
\newblock {\em Ann. Scuola Norm. Sup. Pisa Cl. Sci. (3)}, 18:137--160, 1964.

\bibitem[CC16]{Case-Chang}
Jeffrey~S. Case and Sun-Yung~Alice Chang.
\newblock On fractional {GJMS} operators.
\newblock {\em Comm. Pure Appl. Math.}, 69(6):1017--1061, 2016.

\bibitem[CC19]{CCDuke}
Xavier Cabr\'{e} and Matteo Cozzi.
\newblock A gradient estimate for nonlocal minimal graphs.
\newblock {\em Duke Math. J.}, 168(5):775--848, 2019.

\bibitem[CCV11]{CCV}
Luis Caffarelli, Chi~Hin Chan, and Alexis Vasseur.
\newblock Regularity theory for parabolic nonlinear integral operators.
\newblock {\em J. Amer. Math. Soc.}, 24(3):849--869, 2011.

\bibitem[CF00]{fife}
C.-K. Chen and P.~C. Fife.
\newblock Nonlocal models of phase transitions in solids.
\newblock {\em Adv. Math. Sci. Appl.}, 10(2):821--849, 2000.

\bibitem[CG11]{CG}
Sun-Yung~Alice Chang and Mar{\'{\i}}a del~Mar Gonz{\'a}lez.
\newblock Fractional {L}aplacian in conformal geometry.
\newblock {\em Adv. Math.}, 226(2):1410--1432, 2011.

\bibitem[Cia11]{Cianchi}
Andrea Cianchi.
\newblock Nonlinear potentials, local solutions to elliptic equations and
  rearrangements.
\newblock {\em Ann. Sc. Norm. Super. Pisa Cl. Sci. (5)}, 10(2):335--361, 2011.

\bibitem[CK20]{CKCPDE}
Jamil Chaker and Moritz Kassmann.
\newblock Nonlocal operators with singular anisotropic kernels.
\newblock {\em Comm. Partial Differential Equations}, 45(1):1--31, 2020.

\bibitem[CKW22]{CKWCalcVar}
Jamil Chaker, Minhyun Kim, and Marvin Weidner.
\newblock Regularity for nonlocal problems with non-standard growth.
\newblock {\em Calc. Var. Partial Differential Equations}, 61(6):Paper No. 227,
  2022.

\bibitem[Coz17]{CozziSob}
Matteo Cozzi.
\newblock Interior regularity of solutions of non-local equations in {S}obolev
  and {N}ikol'skii spaces.
\newblock {\em Ann. Mat. Pura Appl. (4)}, 196(2):555--578, 2017.

\bibitem[CS07]{CafSil}
Luis Caffarelli and Luis Silvestre.
\newblock An extension problem related to the fractional {L}aplacian.
\newblock {\em Comm. Partial Differential Equations}, 32(7-9):1245--1260, 2007.

\bibitem[CS09]{CSCPAM}
Luis Caffarelli and Luis Silvestre.
\newblock Regularity theory for fully nonlinear integro-differential equations.
\newblock {\em Comm. Pure Appl. Math.}, 62(5):597--638, 2009.

\bibitem[CS11a]{CSannals}
Luis Caffarelli and Luis Silvestre.
\newblock The {E}vans-{K}rylov theorem for nonlocal fully nonlinear equations.
\newblock {\em Ann. of Math. (2)}, 174(2):1163--1187, 2011.

\bibitem[CS11b]{CSARMA}
Luis Caffarelli and Luis Silvestre.
\newblock Regularity results for nonlocal equations by approximation.
\newblock {\em Arch. Ration. Mech. Anal.}, 200(1):59--88, 2011.

\bibitem[CS14]{CabS1}
Xavier Cabr{\'e} and Yannick Sire.
\newblock Nonlinear equations for fractional {L}aplacians, {I}: {R}egularity,
  maximum principles, and {H}amiltonian estimates.
\newblock {\em Ann. Inst. H. Poincar\'e Anal. Non Lin\'eaire}, 31(1):23--53,
  2014.

\bibitem[CS16]{CStinga}
Luis~A. Caffarelli and Pablo~Ra\'{u}l Stinga.
\newblock Fractional elliptic equations, {C}accioppoli estimates and
  regularity.
\newblock {\em Ann. Inst. H. Poincar\'{e} C Anal. Non Lin\'{e}aire},
  33(3):767--807, 2016.

\bibitem[CV10]{CaffVass}
Luis~A. Caffarelli and Alexis Vasseur.
\newblock Drift diffusion equations with fractional diffusion and the
  quasi-geostrophic equation.
\newblock {\em Ann. of Math. (2)}, 171(3):1903--1930, 2010.

\bibitem[CV14]{CVJDE}
Huyuan Chen and Laurent V\'{e}ron.
\newblock Semilinear fractional elliptic equations involving measures.
\newblock {\em J. Differential Equations}, 257(5):1457--1486, 2014.

\bibitem[DCKP14]{DKP2}
Agnese Di~Castro, Tuomo Kuusi, and Giampiero Palatucci.
\newblock Nonlocal {H}arnack inequalities.
\newblock {\em J. Funct. Anal.}, 267(6):1807--1836, 2014.

\bibitem[DCKP16]{DKP}
Agnese Di~Castro, Tuomo Kuusi, and Giampiero Palatucci.
\newblock Local behavior of fractional {$p$}-minimizers.
\newblock {\em Ann. Inst. H. Poincar\'{e} C Anal. Non Lin\'{e}aire},
  33(5):1279--1299, 2016.

\bibitem[DF22]{DFJMPA}
Cristiana De~Filippis.
\newblock Quasiconvexity and partial regularity via nonlinear potentials.
\newblock {\em J. Math. Pures Appl. (9)}, 163:11--82, 2022.

\bibitem[DFP19]{DFPJDE}
Cristiana De~Filippis and Giampiero Palatucci.
\newblock H\"{o}lder regularity for nonlocal double phase equations.
\newblock {\em J. Differential Equations}, 267(1):547--586, 2019.

\bibitem[DiB83]{DiBen}
E.~DiBenedetto.
\newblock {$C^{1+\alpha }$} local regularity of weak solutions of degenerate
  elliptic equations.
\newblock {\em Nonlinear Anal.}, 7(8):827--850, 1983.

\bibitem[DK12]{DongKim}
Hongjie Dong and Doyoon Kim.
\newblock On {$L_p$}-estimates for a class of non-local elliptic equations.
\newblock {\em J. Funct. Anal.}, 262(3):1166--1199, 2012.

\bibitem[DKM14]{DKMF}
Panagiota Daskalopoulos, Tuomo Kuusi, and Giuseppe Mingione.
\newblock Borderline estimates for fully nonlinear elliptic equations.
\newblock {\em Comm. Partial Differential Equations}, 39(3):574--590, 2014.

\bibitem[DL21]{DongLiu}
Hongjie Dong and Yanze Liu.
\newblock Sobolev estimates for fractional parabolic equations with space-time
  non-local operators, 2021.
\newblock arXiv: 2108.11840.

\bibitem[DM10]{DM1}
Frank Duzaar and Giuseppe Mingione.
\newblock Gradient estimates via linear and nonlinear potentials.
\newblock {\em J. Funct. Anal.}, 259(11):2961--2998, 2010.

\bibitem[DM11]{DM2}
Frank Duzaar and Giuseppe Mingione.
\newblock Gradient estimates via non-linear potentials.
\newblock {\em Amer. J. Math.}, 133(4):1093--1149, 2011.

\bibitem[DNPV12]{Hitch}
Eleonora Di~Nezza, Giampiero Palatucci, and Enrico Valdinoci.
\newblock Hitchhiker's guide to the fractional {S}obolev spaces.
\newblock {\em Bull. Sci. Math.}, 136(5):521--573, 2012.

\bibitem[DZ19]{DZ19}
Hongjie Dong and Hong Zhang.
\newblock On {S}chauder estimates for a class of nonlocal fully nonlinear
  parabolic equations.
\newblock {\em Calc. Var. Partial Differential Equations}, 58(2):Paper No. 40,
  42, 2019.

\bibitem[DZ22]{DZ22}
Hongjie Dong and Hanye Zhu.
\newblock Gradient estimates for singular parabolic {$p$}-{L}aplace type
  equations with measure data.
\newblock {\em Calc. Var. Partial Differential Equations}, 61(3):Paper No. 86,
  41, 2022.

\bibitem[Fal20]{FallCalcVar}
Mouhamed~Moustapha Fall.
\newblock Regularity results for nonlocal equations and applications.
\newblock {\em Calc. Var. Partial Differential Equations}, 59(5):Paper No. 181,
  53, 2020.

\bibitem[FGKV20]{FKV}
Guy~Fabrice Foghem~Gounoue, Moritz Kassmann, and Paul Voigt.
\newblock Mosco convergence of nonlocal to local quadratic forms.
\newblock {\em Nonlinear Anal.}, 193:111504, 22, 2020.

\bibitem[FMSY22]{FMSYPDEA}
Mouhamed~Moustapha Fall, Tadele Mengesha, Armin Schikorra, and Sasikarn Yeepo.
\newblock Calder\'{o}n-{Z}ygmund theory for non-convolution type nonlocal
  equations with continuous coefficient.
\newblock {\em Partial Differ. Equ. Appl.}, 3(2):Paper No. 24, 27, 2022.

\bibitem[FRRO22]{FRR23}
Xavier Fern\'{a}ndez-Real and Xavier Ros-Oton.
\newblock {\em Regularity theory for elliptic {PDE}}, volume~28 of {\em Zurich
  Lectures in Advanced Mathematics}.
\newblock EMS Press, Berlin, [2022] \copyright 2022.

\bibitem[FRRO23]{FeRo23}
Xavier Fernández-Real and Xavier Ros-Oton.
\newblock Schauder and {C}ordes-{N}irenberg estimates for nonlocal elliptic
  equations with singular kernels.
\newblock {\em to appear Bulletin of the AMS}, 2023.
\newblock arXiv: 2308.11383.

\bibitem[Gar19]{Garofalo}
Nicola Garofalo.
\newblock Fractional thoughts.
\newblock In {\em New developments in the analysis of nonlocal operators},
  volume 723 of {\em Contemp. Math.}, pages 1--135. Amer. Math. Soc.,
  [Providence], RI, [2019] \copyright 2019.

\bibitem[GO08]{GOsh}
Guy Gilboa and Stanley Osher.
\newblock Nonlocal operators with applications to image processing.
\newblock {\em Multiscale Model. Simul.}, 7(3):1005--1028, 2008.

\bibitem[Gru15]{Grubb}
Gerd Grubb.
\newblock Fractional {L}aplacians on domains, a development of
  {H}\"{o}rmander's theory of {$\mu$}-transmission pseudodifferential
  operators.
\newblock {\em Adv. Math.}, 268:478--528, 2015.

\bibitem[GT01]{GT}
David Gilbarg and Neil~S. Trudinger.
\newblock {\em Elliptic partial differential equations of second order}.
\newblock Classics in Mathematics. Springer-Verlag, Berlin, 2001.
\newblock Reprint of the 1998 edition.

\bibitem[GV20]{GVPisa}
Alexander Grigor'yan and Igor Verbitsky.
\newblock Pointwise estimates of solutions to nonlinear equations for nonlocal
  operators.
\newblock {\em Ann. Sc. Norm. Super. Pisa Cl. Sci. (5)}, 20(2):721--750, 2020.

\bibitem[GW82]{GWG}
Michael Gr\"{u}ter and Kjell-Ove Widman.
\newblock The {G}reen function for uniformly elliptic equations.
\newblock {\em Manuscripta Math.}, 37(3):303--342, 1982.

\bibitem[GZ03]{GZ}
C.~Robin Graham and Maciej Zworski.
\newblock Scattering matrix in conformal geometry.
\newblock {\em Invent. Math.}, 152(1):89--118, 2003.

\bibitem[Kas09]{KassCalcVar}
Moritz Kassmann.
\newblock A priori estimates for integro-differential operators with measurable
  kernels.
\newblock {\em Calc. Var. Partial Differential Equations}, 34(1):1--21, 2009.

\bibitem[KKL21]{KKL21}
Moritz Kassmann, Minhyun Kim, and Ki-Ahm Lee.
\newblock Robust near-diagonal {G}reen function estimates, 2021.
\newblock arXiv: 2111.05768.

\bibitem[KKP16]{existence}
Janne Korvenp\"{a}\"{a}, Tuomo Kuusi, and Giampiero Palatucci.
\newblock The obstacle problem for nonlinear integro-differential operators.
\newblock {\em Calc. Var. Partial Differential Equations}, 55(3):Art. 63, 29,
  2016.

\bibitem[KLL23]{KLL}
Minhyun Kim, Ki-Ahm Lee, and Se-Chan Lee.
\newblock The {W}iener criterion for nonlocal {D}irichlet problems.
\newblock {\em Comm. Math. Phys.}, 400(3):1961--2003, 2023.

\bibitem[KM94]{KM}
Tero Kilpel{\"a}inen and Jan Mal{\'y}.
\newblock The {W}iener test and potential estimates for quasilinear elliptic
  equations.
\newblock {\em Acta Math.}, 172(1):137--161, 1994.

\bibitem[KM05]{KrMin}
Jan Kristensen and Giuseppe Mingione.
\newblock The singular set of {$\omega$}-minima.
\newblock {\em Arch. Ration. Mech. Anal.}, 177(1):93--114, 2005.

\bibitem[KM06]{KrMin1}
Jan Kristensen and Giuseppe Mingione.
\newblock The singular set of minima of integral functionals.
\newblock {\em Arch. Ration. Mech. Anal.}, 180(3):331--398, 2006.

\bibitem[KM13]{KuMiARMA1}
Tuomo Kuusi and Giuseppe Mingione.
\newblock Linear potentials in nonlinear potential theory.
\newblock {\em Arch. Ration. Mech. Anal.}, 207(1):215--246, 2013.

\bibitem[KM14a]{KuMiG}
Tuomo Kuusi and Giuseppe Mingione.
\newblock Guide to nonlinear potential estimates.
\newblock {\em Bull. Math. Sci.}, 4(1):1--82, 2014.

\bibitem[KM14b]{KuMi}
Tuomo Kuusi and Giuseppe Mingione.
\newblock Riesz potentials and nonlinear parabolic equations.
\newblock {\em Arch. Ration. Mech. Anal.}, 212(3):727--780, 2014.

\bibitem[KM17]{KMJEMS}
Moritz Kassmann and Ante Mimica.
\newblock Intrinsic scaling properties for nonlocal operators.
\newblock {\em J. Eur. Math. Soc. (JEMS)}, 19(4):983--1011, 2017.

\bibitem[KM18]{KuMiV}
Tuomo Kuusi and Giuseppe Mingione.
\newblock Vectorial nonlinear potential theory.
\newblock {\em J. Eur. Math. Soc. (JEMS)}, 20(4):929--1004, 2018.

\bibitem[KMS15a]{KMS2}
Tuomo Kuusi, Giuseppe Mingione, and Yannick Sire.
\newblock Nonlocal equations with measure data.
\newblock {\em Comm. Math. Phys.}, 337(3):1317--1368, 2015.

\bibitem[KMS15b]{KMS1}
Tuomo Kuusi, Giuseppe Mingione, and Yannick Sire.
\newblock Nonlocal self-improving properties.
\newblock {\em Anal. PDE}, 8(1):57--114, 2015.

\bibitem[KPU11]{KPU}
Kenneth~H. Karlsen, Francesco Petitta, and Suleyman Ulusoy.
\newblock A duality approach to the fractional {L}aplacian with measure data.
\newblock {\em Publ. Mat.}, 55(1):151--161, 2011.

\bibitem[KR13]{KRJFA}
Tomasz Klimsiak and Andrzej Rozkosz.
\newblock Dirichlet forms and semilinear elliptic equations with measure data.
\newblock {\em J. Funct. Anal.}, 265(6):890--925, 2013.

\bibitem[Lan72]{landkof}
N.~S. Landkof.
\newblock {\em Foundations of modern potential theory}.
\newblock Die Grundlehren der mathematischen Wissenschaften, Band 180.
  Springer-Verlag, New York-Heidelberg, 1972.

\bibitem[LSW63]{LSW}
W.~Littman, G.~Stampacchia, and H.~F. Weinberger.
\newblock Regular points for elliptic equations with discontinuous
  coefficients.
\newblock {\em Ann. Scuola Norm. Sup. Pisa Cl. Sci. (3)}, 17:43--77, 1963.

\bibitem[LY87]{LiebYau1}
E.~H. Lieb and H.-T. Yau.
\newblock The {C}handrasekhar theory of stellar collapse as the limit of
  quantum mechanics.
\newblock {\em Comm. Math. Phys.}, 112(1):147--174, 1987.

\bibitem[LY88]{LiebYau2}
E.~H. Lieb and H.-T. Yau.
\newblock The stability and instability of relativistic matter.
\newblock {\em Comm. Math. Phys.}, 118(2):177--213, 1988.

\bibitem[Man86]{Manfredi}
Juan~J. Manfredi.
\newblock {\em Regularity of the gradient for a class of nonlinear possibly
  degenerate elliptic equations}.
\newblock 1986.
\newblock PhD thesis, Washington University in St. Louis.

\bibitem[Mey63]{meyers}
Norman~G. Meyers.
\newblock An {$L^{p}$}-estimate for the gradient of solutions of second order
  elliptic divergence equations.
\newblock {\em Ann. Scuola Norm. Sup. Pisa (3)}, 17:189--206, 1963.

\bibitem[Min03]{Mingione}
Giuseppe Mingione.
\newblock The singular set of solutions to non-differentiable elliptic systems.
\newblock {\em Arch. Ration. Mech. Anal.}, 166(4):287--301, 2003.

\bibitem[Min07]{MinCZMD}
Giuseppe Mingione.
\newblock The {C}alder\'{o}n-{Z}ygmund theory for elliptic problems with
  measure data.
\newblock {\em Ann. Sc. Norm. Super. Pisa Cl. Sci. (5)}, 6(2):195--261, 2007.

\bibitem[Min11]{Min}
Giuseppe Mingione.
\newblock Gradient potential estimates.
\newblock {\em J. Eur. Math. Soc. (JEMS)}, 13(2):459--486, 2011.

\bibitem[MSY21]{MSY}
Tadele Mengesha, Armin Schikorra, and Sasikarn Yeepo.
\newblock Calderon-{Z}ygmund type estimates for nonlocal {PDE} with
  {H}\"{o}lder continuous kernel.
\newblock {\em Adv. Math.}, 383:Paper No. 107692, 64, 2021.

\bibitem[Now20]{NowakNA}
Simon Nowak.
\newblock {$H^{s,p}$} regularity theory for a class of nonlocal elliptic
  equations.
\newblock {\em Nonlinear Anal.}, 195:111730, 28, 2020.

\bibitem[Now21a]{MeH}
Simon Nowak.
\newblock Higher {H}\"{o}lder regularity for nonlocal equations with irregular
  kernel.
\newblock {\em Calc. Var. Partial Differential Equations}, 60(1):Paper No. 24,
  37, 2021.

\bibitem[Now21b]{MeN}
Simon Nowak.
\newblock Higher integrability for nonlinear nonlocal equations with irregular
  kernel.
\newblock In {\em Analysis and partial differential equations on manifolds,
  fractals and graphs}, volume~3 of {\em Adv. Anal. Geom.}, pages 459--492. De
  Gruyter, Berlin, [2021] \copyright 2021.

\bibitem[Now23a]{MeI}
Simon Nowak.
\newblock Improved {S}obolev regularity for linear nonlocal equations with
  {VMO} coefficients.
\newblock {\em Math. Ann.}, 385(3-4):1323--1378, 2023.

\bibitem[Now23b]{MeV}
Simon Nowak.
\newblock Regularity theory for nonlocal equations with {VMO} coefficients.
\newblock {\em Ann. Inst. H. Poincar\'{e} C Anal. Non Lin\'{e}aire},
  40(1):61--132, 2023.

\bibitem[Pet16]{Petitta}
Francesco Petitta.
\newblock Some remarks on the duality method for integro-differential equations
  with measure data.
\newblock {\em Adv. Nonlinear Stud.}, 16(1):115--124, 2016.

\bibitem[RO16]{RosOtonBounded}
Xavier Ros-Oton.
\newblock Nonlocal elliptic equations in bounded domains: a survey.
\newblock {\em Publ. Mat.}, 60(1):3--26, 2016.

\bibitem[ROS14]{RSJMPA}
Xavier Ros-Oton and Joaquim Serra.
\newblock The {D}irichlet problem for the fractional {L}aplacian: regularity up
  to the boundary.
\newblock {\em J. Math. Pures Appl. (9)}, 101(3):275--302, 2014.

\bibitem[ROS16a]{RSDuke}
Xavier Ros-Oton and Joaquim Serra.
\newblock Boundary regularity for fully nonlinear integro-differential
  equations.
\newblock {\em Duke Math. J.}, 165(11):2079--2154, 2016.

\bibitem[ROS16b]{RSJDE}
Xavier Ros-Oton and Joaquim Serra.
\newblock Regularity theory for general stable operators.
\newblock {\em J. Differential Equations}, 260(12):8675--8715, 2016.

\bibitem[Sch16]{SchikorraMA}
Armin Schikorra.
\newblock Nonlinear commutators for the fractional {$p$}-{L}aplacian and
  applications.
\newblock {\em Math. Ann.}, 366(1-2):695--720, 2016.

\bibitem[Ser15]{Serra}
Joaquim Serra.
\newblock {$C^{\sigma+\alpha}$} regularity for concave nonlocal fully nonlinear
  elliptic equations with rough kernels.
\newblock {\em Calc. Var. Partial Differential Equations}, 54(4):3571--3601,
  2015.

\bibitem[Sil06]{Silvestre}
Luis Silvestre.
\newblock H\"{o}lder estimates for solutions of integro-differential equations
  like the fractional {L}aplace.
\newblock {\em Indiana Univ. Math. J.}, 55(3):1155--1174, 2006.

\bibitem[Tri06]{Tr3}
Hans Triebel.
\newblock {\em Theory of function spaces. {III}}, volume 100 of {\em Monographs
  in Mathematics}.
\newblock Birkh\"{a}user Verlag, Basel, 2006.

\bibitem[TW02]{TW}
Neil~S. Trudinger and Xu-Jia Wang.
\newblock On the weak continuity of elliptic operators and applications to
  potential theory.
\newblock {\em Amer. J. Math.}, 124(2):369--410, 2002.

\bibitem[Ver22]{Verbitsky}
Igor~E. Verbitsky.
\newblock Nonlinear potential estimates for sublinear problems with
  applications to elliptic semilinear and quasilinear equations, 2022.
\newblock arXiv: 2210.11008.

\bibitem[Zas02]{zaslavsky}
G.~M. Zaslavsky.
\newblock Chaos, fractional kinetics, and anomalous transport.
\newblock {\em Phys. Rep.}, 371(6):461--580, 2002.

\end{thebibliography}

\end{document}